\documentclass[twosided,a4paper]{amsart}

\usepackage{xypic}
\input xy
\xyoption{all}
\usepackage{epsfig}
\usepackage{color}
\usepackage{amsthm}
\usepackage{changepage}
\usepackage{amssymb}
\usepackage{mathdots}
\usepackage{amsmath}
\usepackage{amscd}
\usepackage{amsopn}
\usepackage{enumerate}
\usepackage{mathtools}
\usepackage{url}
\usepackage{graphicx}
\usepackage{hyperref}

\usepackage{hyperref}\hypersetup{colorlinks}

\usepackage{color} 

\definecolor{darkred}{rgb}{1,0,0} 
\definecolor{darkgreen}{rgb}{0,0.8,0}
\definecolor{darkblue}{rgb}{0,0,1}

\hypersetup{colorlinks,
linkcolor=darkblue,
filecolor=darkgreen,
urlcolor=darkred,
citecolor=darkgreen}

 \numberwithin{equation}{section}

\newtheorem {Theorem}{Theorem}

 \numberwithin{Theorem}{section}

\newtheorem {Lemma}[Theorem]    {Lemma}

\newtheorem {Proposition}[Theorem]{Proposition}
\newtheorem {Corollary}[Theorem]{Corollary}
\theoremstyle{definition}
\newtheorem{Definition}[Theorem]{Definition}

\newtheorem{Remark}[Theorem]{Remark}
\newtheorem{Example}[Theorem]{Example}

\theoremstyle{remark}
\newtheoremstyle{monThm}
{} 
{} 
{\itshape} 
{} 
{\bfseries} 
{.} 
{ } 
{} 

\theoremstyle{monThm}
\newtheorem{MonThm}{Theorem}
\newtheorem{MonCorollary}[MonThm]{Corollary}
\newtheorem{MonConjecture}[MonThm]{Conjecture}

\newtheoremstyle{TheoremForIntro} 
        {.6em}{.6em}              
        {\itshape}                      
        {}                              
        {\bfseries}                     
        { }                             
        { }                             
        {\thmname{#1}\thmnote{ \itshape #3}}
    \theoremstyle{TheoremForIntro}

\expandafter\chardef\csname pre amssym.def at\endcsname=\the\catcode`\@
\catcode`\@=11
\def\undefine#1{\let#1\undefined}
\def\newsymbol#1#2#3#4#5{\let\next@\relax
 \ifnum#2=\@ne\let\next@\msafam@\else
 \ifnum#2=\tw@\let\next@\msbfam@\fi\fi
 \mathchardef#1="#3\next@#4#5}
\def\mathhexbox@#1#2#3{\relax
 \ifmmode\mathpalette{}{\m@th\mathchar"#1#2#3}%
 \else\leavevmode\hbox{$\m@th\mathchar"#1#2#3$}\fi}
\def\hexnumber@#1{\ifcase#1 0\or 1\or 2\or 3\or 4\or 5\or 6\or 7\or 8\or
 9\or A\or B\or C\or D\or E\or F\fi}

\font\teneufm=eufm10
\font\seveneufm=eufm7
\font\fiveeufm=eufm5
\newfam\eufmfam
\textfont\eufmfam=\teneufm
\scriptfont\eufmfam=\seveneufm
\scriptscriptfont\eufmfam=\fiveeufm

\catcode`\@=\csname pre amssym.def at\endcsname

\newcommand{\fgl}{{\mathfrak {gl}}}
\newcommand{\fso}{{\mathfrak {so}}}

\newcommand{\fsp}{{\mathfrak {sp}}}

\newcommand{\fg}{{\mathfrak g}}
\newcommand{\fh}{{\mathfrak h}}

\newcommand{\fs}{{\mathfrak s}}
\newcommand{\fc}{{\mathfrak c}}

\newcommand{\ft}{{\mathfrak t}}

\newcommand{\fm}{{\mathfrak m}}

\newcommand{\fu}{{\mathfrak u}}

\newcommand{\sSU}{{\mathsf {SU}}}
\newcommand{\sSL}{{\mathsf {SL}}}
\newcommand{\sSO}{{\mathsf {SO}}}
\newcommand{\sO}{{\mathsf {O}}}
\newcommand{\sSp}{{\mathsf {Sp}}}
\newcommand{\sGL}{{\mathsf {GL}}}
\newcommand{\sSpin}{{\mathsf {Spin}}}
\newcommand{\sPin}{{\mathsf {Pin}}}

\newcommand{\sG}{{\mathsf G}}

\newcommand{\sH}{{\mathsf H}}
\newcommand{\sS}{{\mathsf S}}
\newcommand{\sC}{{\mathsf C}}

\newcommand{\sT}{{\mathsf T}}

\newcommand{\sP}{{\mathsf P}}

\newcommand{\sM}{{\mathsf M}}

\newcommand{\Mod}{{\mathsf {Mod}}}

\newcommand{\Teich}{{\mathsf{Teich}}}
\newcommand{\Prym}{{\mathsf{Prym}}}
\newcommand{\Pic}{{\mathsf{Pic}}}
\newcommand{\Fuch}{{\mathsf{Fuch}}}

\newcommand{\Hom}{{\mathsf{Hom}}}
\newcommand{\Diff}{{\mathsf{Diff}}}

\newcommand{\End}{{\mathsf{End}}}
\newcommand{\Hit}{{\mathsf{Hit}}}

\newcommand{\Aut}{{\mathsf{Aut}}}

\newcommand{\Aa}{{\mathcal A}}
\newcommand{\Bb}{{\mathcal B}}

\newcommand{\Xx}{{\mathcal X}}

\newcommand{\Ii}{{\mathcal I}}

\newcommand{\Jj}{{\mathcal J}}
\newcommand{\Ee}{{\mathcal E}}
\newcommand{\Ff}{{\mathcal F}}

\newcommand{\Gg}{{\mathcal G}}
\newcommand{\Hh}{{\mathcal H}}
\newcommand{\Kk}{{\mathcal K}}
\newcommand{\Mm}{{\mathcal M}}
\newcommand{\Nn}{{\mathcal N}}
\newcommand{\Oo}{{\mathcal O}}
\newcommand{\Pp}{{\mathcal P}}

\newcommand{\Qq}{{\mathcal Q}}
\newcommand{\Uu}{{\mathcal U}}
\newcommand{\Vv}{{\mathcal V}}

\newcommand{\Ww}{{\mathcal W}}
\newcommand{\Zz}{{\mathcal Z}}

\newcommand{\floor}[1]{\left\lfloor #1\right\rfloor}

\newcommand{\mtrx}[1]{\left (\begin{matrix}#1\end{matrix}\right)}
\newcommand{\smtrx}[1]{\left (\begin{smallmatrix}#1\end{smallmatrix}\right)}
\newcommand{\Thmtrx}[1]{\left (\vcenter{\xymatrix@=-.2em{#1}}\right)}

\newcommand{\Sym}{\text{Sym}}

\def    \C      {{\mathbb C}}
\def    \R      {{\mathbb R}}
\def    \Z      {{\mathbb Z}}
\def    \N      {{\mathbb N}}

\def    \Q      {{\mathbb Q}}

\def    \P    {{\mathbb P}}

\def    \ra     {{\rightarrow}}
\def    \lra     {{\longrightarrow}}
\def    \haf    {{\frac{1}{2}}}

\def    \p      {\partial}

\def    \rk     {\operatorname{rk}}

\def    \tr     {\operatorname{tr}}

\newcommand{\An}{\xymatrix{ *{\circ} \ar@{-}[r]|*\dir{ } & *{\circ}\ar@{-}[r]&{\cdots}&*{\circ}\ar@{-}[l]|*\dir{ }\ar@{-}[r]|*\dir{ }&*{\circ} }}
\newcommand{\Anlabel}{\xymatrix@R=.25em{ *{\circ}\ar@<-1ex>@{}[d]^{\alpha_{1}} \ar@{-}[r]|*\dir{ } & *{\circ}\ar@<-1ex>@{}[d]^{\alpha_{2}}\ar@{-}[r]&{\cdots}&*{\circ}\ar@{-}[l]|*\dir{ }\ar@{-}[r]|*\dir{ }\ar@<-2ex>@{}[d]^{\alpha_{n-1}}&*{\circ}\ar@<-1ex>@{}[d]^{\alpha_{n}}\\&&&& }}
\newcommand{\AnExtended}{\xymatrix@R=.25em{&&&&\\&& *{\circ}\ar@<-1ex>@{}[u]^{\ \ \ \ \ \ \alpha_0}\ar@{-}[dddll]\ar@{-}[dddrr]&& \\ &&&&\\&&&& \\*{\circ}\ar@<-1ex>@{}[d]^{} \ar@{-}[r]|*\dir{ } & *{\circ}\ar@<-1ex>@{}[d]^{}\ar@{-}[r]&{\cdots}&*{\circ}\ar@{-}[l]|*\dir{ }\ar@{-}[r]|*\dir{ }\ar@<-2ex>@{}[d]^{ }&*{\circ}\ar@<-1ex>@{}[d]^{ }\\&&&& }}
\newcommand{\AnExtendedlabel}{\xymatrix@R=.25em{&&&&\\&& *{\circ}\ar@<-1ex>@{}[u]^{1}\ar@{-}[dddll]\ar@{-}[dddrr]&& \\ &&&&\\&&&& \\*{\circ}\ar@<-1ex>@{}[d]^{1} \ar@{-}[r]|*\dir{ } & *{\circ}\ar@<-1ex>@{}[d]^{1}\ar@{-}[r]&{\cdots}&*{\circ}\ar@{-}[l]|*\dir{ }\ar@{-}[r]|*\dir{ }\ar@<-2ex>@{}[d]^{1}&*{\circ}\ar@<-1ex>@{}[d]^{1}\\&&&& }}
\newcommand{\Bn}{\xymatrix{ *{\circ} \ar@{-}[r]|*\dir{ } & *{\circ}\ar@{-}[r]&{\cdots}&*{\circ}\ar@{-}[l]|*\dir{ }\ar@{=}[r]|*\dir{>}&*{\circ} }}
\newcommand{\Bnlabel}{\xymatrix@R=.25em{ *{\circ}\ar@<-1ex>@{}[d]^{\alpha_{1}} \ar@{-}[r]|*\dir{ } & *{\circ}\ar@<-1ex>@{}[d]^{\alpha_{2}}\ar@{-}[r]&{\cdots}&*{\circ}\ar@{-}[l]|*\dir{ }\ar@{=}[r]|*\dir{>}\ar@<-2ex>@{}[d]^{\alpha_{n-1}}&*{\circ}\ar@<-1ex>@{}[d]^{\alpha_{n}}\\&&&& }}
 \newcommand{\BnExtended}{\xymatrix@R=.25em{*{\circ}\ar@<-1ex>@{}[d]^{\alpha_0}\ar@{-}[dr]&&&& \\  & *{\circ}\ar@<-1ex>@{}[d]^{ }\ar@{-}[r]&{\cdots}&*{\circ}\ar@{-}[l]|*\dir{ }\ar@{=}[r]|*\dir{>}|*\dir{ }\ar@<-1ex>@{}[d]^{ }&*{\circ}\ar@<-1ex>@{}[d]^{ } \\*{\circ}\ar@<-2ex>@{}[u]^(-.5){ }\ar@{-}[ur] &&&&}}

 \newcommand{\BnExtendedlabel}{\xymatrix@R=.25em{*{\circ}\ar@<-1ex>@{}[d]^{1}\ar@{-}[dr]&&&& \\  & *{\circ}\ar@<-1ex>@{}[d]^{2}\ar@{-}[r]&{\cdots}&*{\circ}\ar@{-}[l]|*\dir{ }\ar@{=}[r]|*\dir{>}|*\dir{ }\ar@<-1ex>@{}[d]^{2}&*{\circ}\ar@<-1ex>@{}[d]^{2} \\*{\circ}\ar@<-2ex>@{}[u]^(-.5){1}\ar@{-}[ur] &&&&}}

\newcommand{\Cn}{\xymatrix{ *{\circ} \ar@{-}[r]|*\dir{ } & *{\circ}\ar@{-}[r]&{\cdots}&*{\circ}\ar@{-}[l]|*\dir{ }\ar@{=}[r]|*\dir{<}&*{\circ} }}
\newcommand{\Cnlabel}{\xymatrix@R=.25em{ *{\circ}\ar@<-1ex>@{}[d]^{\alpha_{1}} \ar@{-}[r]|*\dir{ } & *{\circ}\ar@<-1ex>@{}[d]^{\alpha_{2}}\ar@{-}[r]&{\cdots}&*{\circ}\ar@{-}[l]|*\dir{ }\ar@{=}[r]|*\dir{<}\ar@<-2ex>@{}[d]^{\alpha_{n-1}}&*{\circ}\ar@<-1ex>@{}[d]^{\alpha_{n}}\\&&&& }}

\newcommand{\CnExtended}{\xymatrix@R=.25em{*{\circ}\ar@<-1ex>@{}[d]^{\alpha_0}\ar@{=}[r]|*\dir{>} &*{\circ}\ar@<-1ex>@{}[d]^{ } \ar@{-}[r]|*\dir{ } & *{\circ}\ar@<-1ex>@{}[d]^{ }\ar@{-}[r]&{\cdots}&*{\circ}\ar@{-}[l]|*\dir{ }\ar@{=}[r]|*\dir{<}\ar@<-2ex>@{}[d]^{ }&*{\circ}\ar@<-1ex>@{}[d]^{ }\\&&&&& }}

\newcommand{\CnExtendedlabel}{\xymatrix@R=.25em{*{\circ}\ar@<-1ex>@{}[d]^{1}\ar@{=}[r]|*\dir{>} &*{\circ}\ar@<-1ex>@{}[d]^{2} \ar@{-}[r]|*\dir{ } & *{\circ}\ar@<-1ex>@{}[d]^{2}\ar@{-}[r]&{\cdots}&*{\circ}\ar@{-}[l]|*\dir{ }\ar@{=}[r]|*\dir{<}\ar@<-2ex>@{}[d]^{2}&*{\circ}\ar@<-1ex>@{}[d]^{1}\\&&&&& }}

\newcommand{\Dn}{\xymatrix@R=.25em{&&&&*{\circ} \\ *{\circ} \ar@{-}[r]|*\dir{ } & *{\circ}\ar@{-}[r]&{\cdots}&*{\circ}\ar@{-}[l]|*\dir{ }\ar@{-}[ur]|*\dir{ }\ar@{-}[dr]|*\dir{ }& \\ &&&&*{\circ} }}
\newcommand{\Dnlabel}{\xymatrix@R=.25em{&&&&*{\circ}\ar@<-1ex>@{}[d]^{\alpha_{n-1}} \\ *{\circ}\ar@<-1ex>@{}[d]^{\alpha_{1}} \ar@{-}[r]|*\dir{ } & *{\circ}\ar@<-1ex>@{}[d]^{\alpha_{2}}\ar@{-}[r]&{\cdots}&*{\circ}\ar@{-}[l]|*\dir{ }\ar@{-}[ur]|*\dir{ }\ar@{-}[dr]|*\dir{ }\ar@<-4ex>@{}[d]^{\alpha_{n-2}}& \\ &&&&*{\circ}\ar@<-2ex>@{}[u]^(-.5){\alpha_{n}} }}
\newcommand{\DnExtended}{\xymatrix@R=.25em{*{\circ}\ar@<-1ex>@{}[d]^{\alpha_0}\ar@{-}[dr]&&&&*{\circ}\ar@<-1ex>@{}[d]^{ }\ar@{-}[dl] \\  & *{\circ}\ar@<-1ex>@{}[d]^{ }\ar@{-}[r]&{\cdots}&*{\circ}\ar@{-}[l]|*\dir{ }\ar@<-1ex>@{}[d]^{ }& \\*{\circ}\ar@<-2ex>@{}[u]^(-.5){ }\ar@{-}[ur] &&&&*{\circ}\ar@<-2ex>@{}[u]^(-.5){ }\ar@{-}[ul]}}
\newcommand{\DnExtendedlabel}{\xymatrix@R=.25em{*{\circ}\ar@<-1ex>@{}[d]^{1}\ar@{-}[dr]&&&&*{\circ}\ar@<-1ex>@{}[d]^{1}\ar@{-}[dl] \\  & *{\circ}\ar@<-1ex>@{}[d]^{2}\ar@{-}[r]&{\cdots}&*{\circ}\ar@{-}[l]|*\dir{ }\ar@<-1ex>@{}[d]^{2}& \\*{\circ}\ar@<-2ex>@{}[u]^(-.5){1}\ar@{-}[ur] &&&&*{\circ}\ar@<-2ex>@{}[u]^(-.5){1}\ar@{-}[ul]}}

\usepackage{color}
\definecolor{NoteColor}{rgb}{1,0,0}

\newcommand{\MCG}{\sM\sC\sG}
\newcommand{\PMCG}{\sP\sM\sC\sG}
\newcommand{\SMCG}{\sS\sM\sC\sG}

\begin{document}

\title[Components of maximal $\sP\sSp(4,\R)$ representations]{The geometry of maximal components of the $\sP\sSp(4,\R)$ character variety}

\author{Daniele Alessandrini}
\address{Daniele Alessandrini, Universitaet Heidelberg, Mathematisches Institut, INF 205, 69120, Heidelberg, Germany}
\email{daniele.alessandrini@gmail.com}

\author{Brian Collier}
\address{Brian Collier, University of Maryland, 4176 Campus Drive - William E. Kirwan Hall, College Park, MD 20742-4015}
\curraddr{}
\email{briancollier01@gmail.com}


\copyrightinfo{2017}{Daniele Alessandrini and Brian Collier}


\date{\today}

\dedicatory{}

\begin{abstract}

In this paper we describe the space of maximal components of the character variety of surface group representations into $\sP\sSp(4,\R)$ and $\sSp(4,\R)$.

For every real rank 2 Lie group of Hermitian type, we construct a mapping class group invariant complex structure on the maximal components.
For the groups $\sP\sSp(4,\R)$ and $\sSp(4,\R)$, we give a mapping class group invariant parameterization of each maximal component as an explicit holomorphic fiber bundle over Teichm\"uller space. 
Special attention is put on the connected components which are singular, we give a precise local description of the singularities and their geometric interpretation. We also describe the quotient of the maximal components of $\sP\sSp(4,\R)$ and $\sSp(4,\R)$ by the action of the mapping class group as a holomorphic submersion over the moduli space of curves.

These results are proven in two steps, first we use Higgs bundles to give a non-mapping class group equivariant parameterization, then we prove an analogue of Labourie's conjecture for maximal $\sP\sSp(4,\R)$ representations. 
\end{abstract}

\setlength{\smallskipamount}{6pt}
\setlength{\medskipamount}{10pt}
\setlength{\bigskipamount}{16pt}

\maketitle

\section{Introduction}

Let $\Gamma$ be the fundamental group of a closed orientable surface $S$ of genus $g\geq2$ and let $\sG$ be a connected real semi-simple algebraic Lie group. 
Our main object of interest is the character variety $\Xx(\Gamma,\sG)$ of representations of $\Gamma$ into $\sG$. It can be seen as the set of reductive representations $\Hom^+(\Gamma,\sG)$ up to conjugation:
\[\Xx(\Gamma,\sG) = \Hom^+(\Gamma,\sG)/\sG~. \] 
The mapping class group $\MCG(S) = \Diff^+(S)/\Diff_0(S)$ acts naturally on $\Xx(\Gamma,\sG)$.

Character varieties and their mapping class group symmetry are the main objects of study in Higher Teichm\"uller theory (see for example \cite{BurgerIozziWienhardSurvey} and  \cite{GoldmanMappingClassAction}). They also play an important role in other areas of geometry and theoretical physics.

The {\em natural} geometric structures on $\Xx(\Gamma,\sG)$ are the ones preserved by the action of the mapping class group. For instance, the Goldman symplectic form  defines a natural symplectic structure on $\Xx(\Gamma,\sG)$. 
When $\sG$ is a complex group, the complex structure on $\sG$ gives $\Xx(\Gamma,\sG)$ a natural complex structure. However, when $\sG$ is a real Lie group, there is no obvious natural complex structure on $\Xx(\Gamma,\sG)$. 

There is however a classical subspace of $\Xx(\Gamma,\sP\sSL(2,\R))$ which does admit a natural complex structure. 
This is the set of discrete and faithful representations or \emph{Fuchsian representations} $\Fuch(\Gamma)$; it is a union of two connected components of $\Xx(\Gamma,\sP\sSL(2,\R))$. 
The uniformization theorem defines a mapping class group equivariant diffeomorphism between $\Fuch(\Gamma)$ and a disjoint union of two copies of Teichm\"uller space $\Teich(S)$. 
The complex structure on Teichm\"uller space then induces a natural complex structure on $\Fuch(\Gamma).$  
In fact, the Goldman symplectic form on $\Fuch(\Gamma)$ is the K\"ahler form given by the Weil-Peterson metric. 
Moreover, $\MCG(S)$ acts properly discontinuously on $\Teich(S)$ and the quotient is the Riemann moduli space of curves. 

It is important to note that the presence of a natural complex structure on $\Fuch(\Gamma)$ comes from the uniformization theorem and the work of Teichm\"uller on the moduli of Riemann surfaces (see \cite{AthanaseSurvey2014}). This is a deep result of complex geometry, and there is no elementary way to describe this natural complex structure directly from the definition of $\Xx(\Gamma,\sP\sSL(2,\R))$.

The goal of {\em higher} Teichm\"uller theory is to generalize these classical features to the character varieties of higher rank Lie groups. 
The space of {\em Hitchin representations} into a split real Lie group and the space of {\em maximal representations} into a Lie group of Hermitian type both define particularly interesting components of the character variety. 
The group $\sP\sSL(2,\R)$ is both split and of Hermitian type, and the spaces of Hitchin and maximal representations into $\sP\sSL(2,\R)$ agree and are exactly the subspace $\Fuch(\Gamma)$. 
Hence, for $\sP\sSL(2,\R)$, Hitchin representations and maximal representations are in one to one correspondence with $\Teich(S)$. 

In general, both Hitchin representations and maximal representations have many interesting geometric and dynamical properties, most notably, they are Anosov representations \cite{AnosovFlowsLabourie,MaxRepsAnosov}. 
Consequently, both Hitchin representations and  maximal representations define connected components of discrete and faithful representations which are holonomies of geometric structures on closed manifolds \cite{GWDomainsofDiscont} and carry a properly discontinuous action of the mapping class group \cite{CrossRatioAnosoveProperEnergy,WienhardAction}.

While many of the interesting features of Teichm\"uller theory generalize to these higher rank Lie groups, the analog of the complex geometry of Teichm\"uller space has not yet been developed. 
Indeed, it is not clear that a natural complex structure exists on these generalizations of Teichm\"uller space. 
There are some results in this direction. 
Namely, Loftin \cite{AffSpheresConvexRPn} and Labourie \cite{LabourieCubic} independently constructed a natural complex structure on the $\sP\sSL(3,\R)$-Hitchin component. More recently, Labourie \cite{cyclicSurfacesRank2} constructed a natural complex structure on the Hitchin component for all real split Lie groups of rank two  (namely $\sP\sSL(3,\R)$, $\sP\sSp(4,\R)$ and $\sG_2$). 
This was done by constructing a mapping class group equivariant diffeomorphism between $\Hit(\sG)$ and a holomorphic vector bundle over $\Teich(S)$.

In this paper, we construct a {\em natural complex structure} on the space of maximal representations into any rank two real Lie group $\sG$ of Hermitian type.

\begin{MonThm}\label{Theorem intro MCG C structure}
Let $\Gamma$ be the fundamental group of a closed orientable surface of genus $g\geq 2$ and let $\sG$ be a real rank two semi-simple Lie group of Hermitian type. The space of conjugacy classes of maximal representations of $\Gamma$ into $\sG$ has a mapping class group invariant complex structure. 
\end{MonThm}

\begin{Remark}
It is natural to ask whether the Goldman symplectic form is compatible with this complex structure, as it would then define a mapping class group invariant K\"ahler metric on the space of maximal representations.
\end{Remark}

\begin{Remark}
The Lie group $\sP\sSp(4,\R)$ and its coverings are the only simple rank two groups which are both split and of Hermitian type. The Hitchin component is only one connected component of the many connected components of the set of maximal $\sP\sSp(4,\R)$-representations (see  \eqref{EQ connected components of max PSP}). For the $\sP\sSp(4,\R)$-Hitchin component, the above theorem was proved by Labourie in \cite{cyclicSurfacesRank2}. We note that unlike the Hitchin component, the connected components of maximal representations are neither smooth nor contractible in general. 
\end{Remark}

For the groups $\sP\sSp(4,\R)$ and $\sSp(4,\R)$ we will give a parameterization of each connected component of the space of maximal representations as an explicit holomorphic fiber bundle over $\Teich(S)$. These parameterizations allow us to describe the complex structure from Theorem \ref{Theorem intro MCG C structure} more explicitly, 
and determine the global topology of these spaces and their homotopy type. 

Special attention is placed on components which contain singularities. 
 To understand the singularities, we will describe the local topology around every singular point, and we will show how the type of the singularity is related with the Zariski closure of the associated representation.
We also describe the quotient of the maximal components of $\sP\sSp(4,\R)$ and $\sSp(4,\R)$ by the action of the mapping class group as a holomorphic submersion over the moduli space of curves. 
\begin{Remark}
    Many aspects of the components of maximal representations into $\sSp(4,\R)$ were studied in  \cite{MaximalSP4} using Higgs bundles and in \cite{TopInvariantsAnosov} using representation theory techniques. In this paper, we both extend these results to the group $\sP\sSp(4,\R)$ and provide a more detailed analysis of the components, giving a finer description of the space and its structure.
\end{Remark}

\subsection{Maximal $\sP\sSp(4,\R)$ representations}
The group $\sP\sSp(4,\R)$ is a real rank two Lie group of Hermitian type. 
The (disconnected) subspace of maximal representations of the $\sP\sSp(4,\R)$-character variety will be denoted by $\Xx^{\mathrm{max}}(\Gamma,\sP\sSp(4,\R)).$ 

The space $\Xx^\mathrm{max}(\Gamma,\sP\sSp(4,\R))$ has $2(2^{2g}-1)+4g-3$ connected components \cite{MaxRepsHermSymmSpace}. More precisely, it was shown that there is a bijective correspondence between the connected components of $\Xx^{\mathrm{max}}(\Gamma,\sP\sSp(4,\R))$ and the set
\[\{0,1,\cdots,4g-4\}\sqcup (H^1(S,\Z_2)\setminus\{0\})\times H^2(S,\Z_2)~.\]
For each $d\in\{0,\cdots,4g-4\}$ and each $(sw_1,sw_2)\in (H^1(S,\Z_2)\setminus\{0\})\times H^2(S,\Z_2)$, denote the associated connected component of $\Xx^{\mathrm{max}}(\Gamma,\sP\sSp(4,\R))$ by 
\begin{equation}
\label{EQ connected components of max PSP}\xymatrix{\Xx_{0,d}^{\mathrm{max}}(\Gamma,\sP\sSp(4,\R))&\text{and}&\Xx_{sw_1}^{\mathrm{max},sw_2}(\Gamma,\sP\sSp(4,\R))}~.
\end{equation}

The connected components $\Xx_{0,d}^\mathrm{max}(\Gamma,\sP\sSp(4,\R)),$ for $d\in(0,4g-4],$ are the easiest to parameterize since they are smooth. 
\begin{MonThm}\label{MonTHM Xd}
Let $\Gamma$ be the fundamental group of a closed oriented surface $S$ of genus $g\geq2.$ For each integer $d\in(0,4g-4]$, there is a mapping class group equivariant diffeomorphism between $\Xx_{0,d}^{\mathrm{max}}(\Gamma,\sP\sSp(4,\R))$ and a holomorphic fiber bundle 
\[\pi:\mathcal{Y}_d\to \Teich(S)~.\] Here, for each Riemann surface $\Sigma\in\Teich(S)$, the fiber $\pi^{-1}(\Sigma)$ is a rank $d+3g-3$ vector bundle over the $(4g-4-d)^{th}$ symmetric product of $\Sigma.$ The mapping class group acts on $\mathcal{Y}_d$ by pullback by the action on $\Teich(S).$  
\end{MonThm}
As a direct corollary we have the following.
\begin{MonCorollary}
    Let $\Gamma$ be the fundamental group of a closed surface $S$ of genus $g\geq2.$
     For each integer $d\in(0,4g-4]$, the connected component $\Xx_{0,d}^\mathrm{max}(\Gamma,\sP\sSp(4,\R))$ is smooth and deformation retracts onto the $(4g-4-d)^{th}$ symmetric product of $S.$
\end{MonCorollary}
Note that when $d=4g-4,$ the above theorem says that $\Xx_{0,4g-4}(\Gamma,\sP\sSp(4,\R))$ is diffeomorphic to a rank $7g-7$ vector bundle over $\Teich(S).$ 
In this case, the component $\Xx_{0,4g-4}^{\mathrm{max}}(\Gamma,\sP\sSp(4,\R))$ is the $\sP\sSp(4,\R)$-Hitchin component and the fiber bundle $\Ff_{4g-4}$ is the vector bundle of holomorphic quartic differentials on $\Teich(S).$ In particular, Theorem \ref{MonTHM Xd} recovers Labourie's mapping class group invariant parameterization of the $\Hit(\sP\sSp(4,\R))$ from \cite{cyclicSurfacesRank2}.

\begin{Remark}
Notably, the connected components $\Xx_{0,d}^\mathrm{max}(\Gamma,\sP\sSp(4,\R))$ behave very differently for $d=4g-4$ and $d\in(0,4g-4).$ In particular, they are contractible if and only if $d=4g-4.$ 
Moreover, for $d\in(0,4g-4)$ we show that every representation $\rho\in \Xx_{0,d}^\mathrm{max}(\Gamma,\sP\sSp(4,\R))$ is Zariski dense. This generalizes similar results of \cite{MaximalSP4} and \cite{TopInvariantsAnosov} for certain components of maximal $\sSp(4,\R)$ representations (see Sections \ref{PSp4R} and \ref{Sp4R}).
\end{Remark} 

The connected component $\Xx_{0,0}^{\mathrm{max}}(\Gamma,\sP\sSp(4,\R))$ is the most singular, and thus the hardest to parameterize. We describe it here briefly. For each Riemann surface $\Sigma\in\Teich(S)$, let $\Pic^0(\Sigma)$ denote the abelian variety of degree zero line bundles on $\Sigma$. 
Denote the tautological holomorphic line bundle over $\C\P^{3g-4}$ by $\mathcal{O}_{\C\P^{3g-4}}(-1)$, and let $\mathcal{U}_{3g-3}$ denote the quotient of the total space of the direct sum of $3g-3$ copies of $\mathcal{O}_{\C\P^{n-1}}(-1)$ by the equivalence relation that collapses the zero section to a point. Note that $\Uu_{3g-3}$ is a singular space. 

\begin{MonThm}\label{MonTHM X0}
    Let $\Gamma$ be the fundamental group of a closed oriented surface $S$ of genus $g\geq2.$ There is a mapping class group equivariant homeomorphism between the component $\Xx_{0,0}^{\mathrm{max}}(\Gamma,\sP\sSp(4,\R))$ and a holomorphic fiber bundle
     \[\pi:\mathcal{Y}_0\to \Teich(S)~.\] 
     Here, for each $\Sigma\in\Teich(S),$ $\pi^{-1}(\Sigma)$ is a $\Z_2$ quotient of a holomorphic fiber bundle $\Aa\to\Pic^0(\Sigma)$ with fiber $\mathcal{U}_{3g-3}$ and $\Z_2$ acts by pullback to $\Aa$ of the inversion map $L\mapsto L^{-1}$ on $\Pic^0(\Sigma).$ The mapping class group acts on $\mathcal{Y}_0$ by pullback by the action on $\Teich(S).$   
\end{MonThm}

In Lemma \ref{Un contractible}, we show that the space $\Uu_{3g-3}$ is contractible, thus we have:
\begin{MonCorollary}\label{MonCor X0 cohomology}
The connected component $\Xx_{0,0}^\mathrm{max}(\Gamma,\sP\sSp(4,\R))$ deformation retracts to the quotient of $(S^1)^{2g}$ by the inversion map $x\mapsto x^{-1}$. In particular, its rational cohomology is:
\[H^j(\Xx_{0,0}^\mathrm{max}(\Gamma,\sP\sSp(4,\R)),\Q)\cong \begin{cases}
             H^{j}((S^1)^{2g},\Q) & \text{if\ j\ is even,}\\
             0&\text{otherwise.}
             \end{cases}\]
 \end{MonCorollary} 

Recall that a non-zero cohomology class $sw_1\in H^1(S,\Z_2)$ is equivalent to the data of a connected double covering $\pi: S_{sw_1}\to S.$ For each Riemann surface $\Sigma\in\Teich(S),$ denote the pullback of the complex structure to $S_{sw_1}$ by $\Sigma_{sw_1}$. If $\iota$ denotes the covering involution on $\Sigma_{sw_1},$ consider the following space:
\[\Prym(\Sigma_{sw_1})=\{L\in\Pic^0(\Sigma_{sw_1})\ |\ \iota^*L=L^{-1} \}~.\]
The space $\Prym(\Sigma_{sw_1})$ has two isomorphic connected components $\Prym^0(\Sigma_{sw_1})$ and $\Prym^1(\Sigma_{sw_1})$. The connected component of the identity $\Prym^0(\Sigma_{sw_1})$ is an abelian variety of complex dimension $g-1$ called the Prym variety of the covering.

\begin{MonThm}\label{MonTHM Xsw1}
Let $\Gamma$ be the fundamental group of a closed oriented surface $S$ of genus $g\geq2.$ For each $(sw_1,sw_2)\in H^1(S,\Z_2)\setminus\{0\}\times H^2(S,\Z_2),$ there is a mapping class group equivariant homeomorphism between the component $\Xx_{sw_1}^{\mathrm{max},sw_2}(\Gamma,\sP\sSp(4,\R))$ and a holomorphic fiber bundle
 \[\pi:\mathcal{Y}_{sw_1}^{sw_2}\to \Teich(S)~.\] Here, for each $\Sigma\in\Teich(S)$, $\pi^{-1}(\Sigma)$ is a $\Z_2$ quotient of an explicit holomorphic bundle over $\Prym^{sw_2}(\Sigma_{sw_1})$. The mapping class group acts on $\mathcal{Y}_{sw_1}^{sw_2}$ by pullback by the action on $\Teich(S).$  
\end{MonThm}

The spaces $\Xx_{sw_1}^{\mathrm{max},sw_2}(\Gamma,\sP\sSp(4,\R))$ are singular, but the singularities consist only of $\Z_2$ and $\Z_2\oplus\Z_2 $-orbifold points. The space $\Hh'/\Z_2$ also has an orbifold structure.  
The homeomorphism above is an orbifold isomorphism, in particular, it is smooth away from the singular set.

\begin{MonCorollary}
    Each space $\Xx_{sw_1}^{\mathrm{max},sw_2}(\Gamma,\sP\sSp(4,\R))$ deformation retracts onto the quotient of $(S^1)^{2g-2}$ by inversion. In particular, its rational cohomology is:
    \[H^j(\Xx^{\mathrm{max},sw_2}_{sw_1}(\Gamma,\sP\sSp(4,\R)),\Q)\cong \begin{cases}
             H^{j}((S^1)^{2g-2},\Q) & \text{if\ j\ is even,}\\
             0&\text{otherwise.}
             \end{cases}\]
\end{MonCorollary}

\begin{Remark}
In Theorems \ref{THM SP4 d>0}, \ref{THM Sp4 d=0 } and \ref{THM Higgs Sp4 sw1not0} of Section \ref{sec: param sp4}, we also find analogous descriptions for the components of the character variety $\Xx^\mathrm{max}(\Gamma,\sSp(4,\R))$. 
While every component is a covering of a component of the character variety of $\sP\sSp(4,\R)$, the order of this cover depends on the topological invariants of the component. 
\end{Remark}
\subsection{Higgs bundles and Labourie's conjecture}
For a real semi-simple Lie group $\sG,$ a $\sG$-Higgs bundle consists of a certain holomorphic bundle on a Riemann surface $\Sigma$ together with a section of an associated bundle. 
The remarkable theorem of Hitchin \cite{selfduality} for $\sSL(2,\C)$ and Simpson \cite{SimpsonVHS} for $\sG$ complex semi-simple is that the moduli space $\Mm(\Sigma,\sG)$ of poly-stable $\sG$-Higgs bundles on $\Sigma$ is homeomorphic to the character variety $\Xx(\Gamma,\sG)$. 
This correspondence, usually called the nonabelian Hodge correspondence, also holds for real reductive groups $\sG$ \cite{HiggsPairsSTABILITY}. 

The moduli space of $\sG$-Higgs bundles has more structure than the character variety. For example, there is a Hamiltonian circle action on $\Mm(\Sigma,\sG)$, and, when $\Mm(\Sigma,\sG)$ is smooth, the associated moment map is a perfect Morse-Bott function. Thus, Higgs bundles provide useful tools to study the topology of the character variety. 
When $\Mm(\Sigma,\sG)$ is not smooth, the moment map only provides enough structure to determine bounds on the connected components of the moduli space. 

In special cases, one can explicitly parameterize a connected component of the Higgs bundle moduli space. The only previous examples of this are for the connected components of $\Mm(\Sigma,\sP\sSL(2,\R))$  with non-zero Euler class \cite{selfduality} and the Hitchin component of $\Mm(\Sigma,\sG)$ when $\sG$ is a real split Lie group \cite{liegroupsteichmuller}. 
More precisely, Hitchin proved \cite{liegroupsteichmuller} that, for each Riemann surface $\Sigma\in\Teich(S)$, the Hitchin component $\Hit(\sG)\subset\Xx(\Gamma,\sG)$ is homeomorphic to the vector space 
\begin{equation}\label{INto THM Hitchin comp}
\Hit(\sG)\cong\bigoplus\limits_{j=1}^{\rk(\sG)}H^0(\Sigma, K^{m_j+1}) ~,
\end{equation} where $K$ is the canonical bundle of $\Sigma$, $m_1=1$ and the integers $\{m_j\}$ are the exponents of $\sG.$

\begin{Remark}\label{Intro Remark Higgs param} We note that the nonzero Euler class components of $\Xx(\Gamma,\sP\sSL(2,\R))$ and the Hitchin component are smooth. In Section \ref{PSp4R}, we will explicitly parameterize the connected components of maximal $\sP\sSp(4,\R)$-Higgs bundle moduli space as the product of the fiber from Theorems \ref{MonTHM Xd}, \ref{MonTHM X0} and \ref{MonTHM Xsw1} with the vector space $H^0(\Sigma,K^2)$ of holomorphic quadratic differentials. 
    These parameterizations are the first description of a singular connected component of the Higgs bundle moduli space.   
\end{Remark}

One drawback of the non-abelian Hodge correspondence is that it requires fixing a Riemann surface $\Sigma\in\Teich(S)$ and thus breaks the mapping class group symmetry of $\Xx(\Gamma,\sG)$. In particular, the mapping class group does not act on the parameterization of the Hitchin component from \eqref{INto THM Hitchin comp}. 
To obtain a mapping class group invariant parameterization of the Hitchin component, Labourie suggested the following method of associating a preferred Riemann surface to each $\rho\in\Hit(\sG)$: 

For each $\rho\in\Xx(\Gamma,\sG)$ one can define an energy function
\begin{equation}
    \label{enegyfunction EQ}\Ee_\rho:\Teich(S)\to \R
\end{equation}
by defining $\Ee_\rho(\Sigma)$ to be the energy of a $\rho$-equivariant harmonic map from the universal cover of $\Sigma $ to the symmetric space of $\sG.$ The existence of such harmonic maps is guaranteed by Corlette's Theorem \cite{canonicalmetrics}. 
The critical points of $\Ee_\rho$ are given by those harmonic maps which are weakly conformal, or equivalently, whose image is a branched minimal immersion \cite{SchoenYauMinimalSurfEnergy,MinImmofRiemannSurf}. 
In \cite{CrossRatioAnosoveProperEnergy}, Labourie showed that, for each Anosov representation $\rho\in\Xx(\Gamma,\sG),$ the energy function $\Ee_\rho$ is smooth and proper and thus admits a critical point. He then conjectured that for Hitchin representations, the critical point was unique.

\begin{MonConjecture}{(Labourie \cite{CrossRatioAnosoveProperEnergy})} Let $\Gamma$ be the fundamental group of a closed oriented surface $S$ of genus at least two and let $\sG$ be a semi-simple split real Lie group. If $\rho\in\Xx(\Gamma,\sG)$ is a Hitchin representation, then there is a unique Riemann surface structure $\Sigma\in\Teich(S)$ which is a critical point of the energy function $\Ee_\rho$ from \eqref{enegyfunction EQ}.
\end{MonConjecture}
For a Hitchin representation $\rho,$ a Riemann surface $\Sigma\in\Teich(S)$ is a critical point of the energy function $\Ee_\rho$ if and only the tuple of holomorphic differentials $(q_2,q_{m_2+1},\cdots,q_{m_{\rk\sG}+1})$ associated to $\rho$ via \eqref{INto THM Hitchin comp} has $q_2=0$.
 Thus, a consequence of Labourie's conjecture would be that there is a mapping class group equivariant diffeomorphism between $\Hit(\sG)$ and the vector bundle 
 \[ \xymatrix{\pi:\Vv\to\Teich(S)&\text{with}&\pi^{-1}(\Sigma)=\bigoplus\limits_{j=2}^{\rk(\sG)} H^0(K_X^{m_j})}.\]
Since $\Vv$ is naturally a holomorphic vector bundle, a positive answer to Labourie's conjecture would provide the Hitchin component with a mapping class group invariant complex structure. For this reason, Labourie's conjecture is arguably among the most important conjectures in the field of higher Teichm\"uller theory. 

Labourie's conjecture has been proven when the rank of $\sG$ is two. This was done independently by Loftin \cite{AffSpheresConvexRPn} and Labourie \cite{LabourieCubic} for $\sG=\sP\sSL(3,\R)$ and by Labourie \cite{cyclicSurfacesRank2} in the general case. 
Little is known when $\rk(\sG)>2.$

To go from our Higgs bundle parameterization of the components of maximal $\sP\sSp(4,\R)$ representations to a mapping class group invariant parameterization, we prove an analog of Labourie's conjecture for maximal representations. 
\begin{MonThm}\label{Into Thm Lab conj}
Let $\Gamma$ be the fundamental group of a closed oriented surface with genus at least two. If $\rho\in\Xx^{\mathrm{max}}(\Gamma,\sP\sSp(4,\R))$ is a maximal representation, then there is a unique critical point of the energy function $\Ee_\rho$ from \eqref{enegyfunction EQ}.
\end{MonThm}
 
\begin{Remark}
For $d\in(0,4g-4],$ the above theorem was proven for the connected components $\Xx_{0,d}^{\mathrm{max}}$ by the second author in \cite{MySp4Gothen} by generalizing Labourie's techniques for the Hitchin component. The proof of Theorem \ref{Into Thm Lab conj} is also along these lines. In a more recent work and using completely different methods, the second author with Tholozan and Toulisse \cite{CollierTholozanToulisse} extended Theorem \ref{Into Thm Lab conj} to maximal representations into any real rank two Lie group of Hermitian type. 
\end{Remark}
 We conjecture that the extension of Labourie's conjecture to all maximal representations holds. 
\begin{MonConjecture}\label{Conj MaxLabourie}
Let $\Gamma$ be the fundamental group of a closed oriented surface $S$ of genus at least two and let $\sG$ be a real Lie group of Hermitian type. If $\rho\in\Xx(\Gamma,\sG)$ is a maximal representation, then there is a unique Riemann surface structure $\Sigma\in\Teich(S)$ which is a critical point of the energy function $\Ee_\rho$ from \eqref{enegyfunction EQ}.
\end{MonConjecture}
Now, putting our Higgs bundle parameterization of $\Xx^{\mathrm{max}}(\Gamma,\sP\sSp(4,\R))$ from Remark \ref{Intro Remark Higgs param} together with Theorem \ref{Into Thm Lab conj} we obtain a homeomorphism between each component of $\Xx^\mathrm{max}(\Gamma,\sP\sSp(4,\R))$ and the fibrations over Teichm\"uller space from Theorems \ref{MonTHM Xd}, \ref{MonTHM X0} and \ref{MonTHM Xsw1}. 
At this point it is neither clear that these homeomorphisms are equivariant with respect to the mapping class group action nor is it clear that the fibrations are holomorphic. 
To solve these issues, we construct a {\em universal Higgs bundle moduli space}.

\begin{MonThm}     \label{Intro thm:universal}
Given an algebraic Lie group of Hermitian type $\sG$, there is a complex analytic space $\Mm^\mathrm{max}(\Uu,\sG)$ with a holomorphic map $\pi:\Mm^\mathrm{max}(\Uu,\sG) \to \Teich(S)$ such that 
\begin{enumerate}
\item for every $\Sigma \in \Teich(S)$, $\pi^{-1}(\Sigma)$ is biholomorphic to $\Mm^\mathrm{max}(\Sigma,\sG)$,
\item $\pi$ is a trivial smooth fiber bundle.
\item the pullback operation on Higgs bundles gives a natural action of $\MCG(S)$ on $\Mm^\mathrm{max}(\Uu,\sG)$ by holomorphic maps that lifts the action on $\Teich(S)$. 
\end{enumerate}
\end{MonThm}
\begin{Remark}
The proof of Theorem \ref{Intro thm:universal} relies on Simpson's construction of the moduli space of Higgs bundles over Riemann surfaces over schemes of finite type over $\C$. More specifically, we will use \cite[Corollary 6.7]{SimpsonModuli2}. We note that $\Teich(S)$ is not a scheme of finite type over $\C$.
\end{Remark}
Theorem \ref{Into Thm Lab conj} defines a map 
\begin{equation}\label{EQ map intro}
    \Psi:\xymatrix@R=0em{\Xx^\mathrm{max}(\Gamma,\sP\sSp(4,\R))\ar[r]&\Mm^\mathrm{max}(\Uu,\sP\sSp(4,\R))\\\rho\ar@{|->}[r]&(\Sigma_\rho,E,\Phi)}~,
\end{equation}
where $\Sigma_\rho\in\Teich(S)$ is the unique critical point of $\Ee_\rho$ and $(E,\Phi)$ is the Higgs bundle associated to $\rho$ on the Riemann surface $\Sigma_\rho.$ The natural complex structure on $\Xx^\mathrm{max}(\Gamma,\sP\sSp(4,\R))$ is given by the following corollary.
\begin{MonCorollary}
The map $\Psi$ is equivariant with respect to the mapping class group action and its image is a complex analytic subspace of $\Mm^\mathrm{max}(\Uu,\sP\sSp(4,\R)).$
\end{MonCorollary}
\begin{Remark}
Let $\sG$ be a real semi-simple Lie group of Hermitian type. A positive answer to Conjecture \ref{Conj MaxLabourie} would define a map analogous to \eqref{EQ map intro}
\[\Psi:\xymatrix@R=0em{\Xx^\mathrm{max}(\Gamma,\sG)\ar[r]&\Mm^\mathrm{max}(\Uu,\sG)\\\rho\ar@{|->}[r]&(E,\Phi,\Sigma_\rho)}~.\]
In this case also the map $\Psi$ is equivariant with respect to the mapping class group and its image is a complex analytic subspace. This would give a natural complex structure on $\Xx^{\mathrm{max}}(\Gamma,\sG).$ 
In particular, the extension of Theorem \ref{Into Thm Lab conj} to all rank two Hermitian Lie groups of \cite{CollierTholozanToulisse} implies Theorem \ref{Theorem intro MCG C structure}.
\end{Remark}

\subsection{Organization of the paper}

In Section \ref{section_preliminaries}, we introduce character varieties, Higgs bundles and some Lie theory for the groups $\sP\sSp(4,\R)$ and $\sSp(4,\R)$. In Section \ref{Orthogonal_bundles}, we describe holomorphic orthogonal bundles, with special attention to the description of the moduli space of holomorphic $\sO(2,\C)$-bundles. This is a necessary tool which we will use repeatedly. In Section \ref{PSp4R}, Higgs bundles over a fixed Riemann surface $\Sigma$ are used to describe the topology of $\Xx^\mathrm{max}(\Gamma,\sP\sSp(4,\R))$, special attention is placed on the singular components. In Section \ref{Sp4R}, we prove analogous results for $\Xx^\mathrm{max}(\Gamma,\sSp(4,\R))$. In Section \ref{minimal}, we  prove Labourie's conjecture concerning uniqueness of minimal surfaces. In Section \ref{sec:mcg_inv_cmplx_str}, the universal Higgs bundle moduli space is constructed and in Section \ref{equivariant} we will put everything together and describe the action of $\MCG(S)$ on  $\Xx^\mathrm{max}(\Gamma,\sP\sSp(4,\R))$ and $\Xx^\mathrm{max}(\Gamma,\sSp(4,\R))$.

\subsection*{Acknowledgments} We would like to thank Steve Bradlow, Camilla Felisetti, Ian McIntosh, Gabriele Mondello and Anna Wienhard for very helpful conversations concerning this work. We would also like to thank Olivier Guichard for many useful suggestions and comments. 
The authors gratefully acknowledge support from the NSF grants DMS-1107452, 1107263 and 1107367 “RNMS: GEometric structures And Representation varieties” (the GEAR Network) and the hospitality of the Mathematical
Sciences Research Institute in Berkeley where some of this research was carried out. B. Collier's research is supported by the National Science Foundation under Award No. 1604263.

\section{Character varieties and Higgs bundles}   \label{section_preliminaries}
In this section we recall general facts about character varieties and Higgs bundles. 

\subsection{Character varieties}

Let $\Gamma$ be the fundamental group of an orientable closed surface $S$ with genus $g\geq2$ and let $\sG$ be a connected real semi-simple algebraic Lie group. 
Denote the fundamental group of $S$ by $\Gamma.$
The set of representations of $\Gamma$ into $\sG$ is defined to be the set of group homomorphisms $\Hom(\Gamma,\sG).$ 
Since $\sG$ is algebraic and $\Gamma$ is finitely generated, $\Hom(\Gamma,\sG)$ can be given the structure of an algebraic variety.
A representation $\rho\in\Hom(\Gamma,\sG)$ is called {\em reductive} if the Zariski closure of $\rho(\Gamma)$ is a reductive subgroup of $\sG.$ Denote the space of reductive representations by $\Hom^+(\Gamma,\sG).$
\begin{Definition}
    The $\sG$-character variety $\Xx(\Gamma,\sG)$ is the quotient space $\Xx(\Gamma,\sG)=\Hom^+(\Gamma,\sG)/\sG$
    where $\sG$ acts by conjugation.
    \end{Definition}
The $\sG$-character variety $\Xx(\Gamma,\sG)$ is a real semi-algebraic set of dimension $(2g-2)\mathrm{dim}(\sG)$ \cite{SymplecticNatureofFund} which carries a natural action of the  mapping class group of $S$
\[\MCG(S)=\Diff^+(S)/\Diff_0(S)~.\] 
An element
$\phi\in\MCG(S)$ acts on $\Xx(\Gamma,\sG)$ by precomposition: $\phi\cdot\rho=\rho\circ\phi_*,$
 \[\xymatrix{\Gamma\ar[r]^{\phi_*}&\Gamma\ar[r]^{\rho}&\sG}~.\]

\begin{Example}
The set of Fuchsian representations $\Fuch(\Gamma)\subset\Xx(\Gamma,\sP\sSL(2,\R))$ is defined to be the subset of conjugacy classes of {\em faithful} representations with {\em discrete image}.  
The space $\Fuch(\Gamma)$ consists of two isomorphic connected components of $\Xx(\Gamma,\sP\sSL(2,\R)))$ \cite{TopologicalComponents}. 
Each of these components is in one to one correspondence with the \emph{Teichm\"uller space} $\Teich(S)$ of isotopy classes of Riemann surface structures on the surface $S.$
Furthermore, the mapping class group acts properly discontinuously on $\Fuch(\Gamma)$.

When $\sG$ is a split real group of adjoint type, such as $\sP\sSL(n,\R)$ or $\sP\sSp(2n,\R)$, the unique (up to conjugation) irreducible representation $\iota:\sP\sSL(2,\R)\to\sG$ defines a map \[\iota:\Xx(\Gamma,\sP\sSL(2,\R))\to\Xx(\Gamma,\sG)~,\]
and allows one to try to deform Fuchsian representations into $\Xx(\Gamma,\sG)$. 
The space of \emph{Hitchin representations} $\Hit(\sG)\subset\Xx(\Gamma,\sG)$ is defined to be the union of the connected components containing $\iota(\Fuch(\Gamma)).$ A connected component of $\Hit(\sG)$ is called a {\em Hitchin component}. 

For each Riemann surface structure $\Sigma$ on $S$, Hitchin parameterized each Hitchin component by a vector space of holomorphic differentials \cite{liegroupsteichmuller}; 
\begin{equation}
    \label{EQ Hitchin comp param}\Hit(\sG)\cong \bigoplus\limits_{j=1}^{\mathrm{\rk}(\sG)}H^0(\Sigma,K^{m_j+1})~,
\end{equation} 
where $m_1=1$ and $\{m_j\}$ are the so called exponents of $\sG.$ For $\sG=\sP\sSp(4,\R),$ $m_1=1$ and $m_2=3.$
The mapping class group $\MCG(S)$ acts properly discontinuously on $\Hit(\sG)$ \cite{CrossRatioAnosoveProperEnergy}. Note however that $\MCG(S)$ does not act naturally on the parameterization \eqref{EQ Hitchin comp param} because of the choice of complex structure.
\end{Example}
Associated to a representation $\rho:\Gamma\to\sG$ there is a flat principal $\sG$-bundle $\widetilde S\times_\rho\sG$ on $S.$ In fact, there is a homeomorphism
\[\Xx(\Gamma,\sG)\cong \{\text{Reductive flat } \sG\text{-bundles on } S\}/\text{isomorphism}.\]
We will usually blur the distinction between $\rho\in\Xx(\Gamma,\sG)$ and the corresponding isomorphism class of the flat $\sG$-bundle $\widetilde S\times_\rho\sG$.

\subsection{Higgs bundles}

As above, let $\sG$ be a connected real semi-simple algebraic Lie group and $\sH$ be a maximal compact subgroup. Fix a Cartan involution $\theta:\fg\to\fg$ with Cartan decomposition $\fg=\fh\oplus\fm$; the complexified splitting  $\fg_\C=\fh_\C\oplus\fm_\C$ is $Ad_{H_\C}$-invariant. We will mostly deal with simple Lie groups $\sG.$

Let $\Sigma$ be a compact Riemann surface of genus $g\geq2$ with canonical bundle $K$. 

\begin{Definition} \label{DEF Higgs bundle}
 A {\em $\sG$-Higgs bundle} on $\Sigma$ is a pair $(\Pp_{\sH_\C},\varphi)$ where $\Pp_{\sH_\C}$ is a holomorphic principal $\sH_\C$-bundle on $\Sigma$ and $\varphi$ is a holomorphic $(1,0)$-form valued in the associated bundle with fiber $\fm_\C,$ i.e., $\varphi\in H^0(\Sigma, \Pp_{\sH_\C}[\fm_\C]\otimes K)$. The section $\varphi$ is called the {\em Higgs field}.
\end{Definition}
 Two Higgs bundles $(\Pp,\varphi)$ and $(\Pp',\varphi')$ are \emph{isomorphic} if there exists an isomorphism of the underlying smooth bundles $f:P_{\sH_\C}\to P_{\sH_\C}'$ such that $f^*\Pp'=\Pp$ and $f^*\varphi'=\varphi.$
 We will usually think of the underlying smooth bundle $P_{\sH_\C}$ as being fixed and define the gauge group $\Gg(P_{\sH_\C})$ as the group of smooth bundle automorphisms.

\begin{Example}
    For $\sG=\sGL(n,\C)$, we have $\fh=\fu(n)$, $\fm=i\fu(n)$ and $\sH_\C=\sGL(n,\C).$ 
    Thus a $\sGL(n,\C)$-Higgs bundle is a holomorphic principal $\sGL(n,\C)$-bundle $\Pp\to\Sigma$ and a holomorphic section $\varphi$ of the adjoint bundle of $\Pp$ twisted by $K$.  
    Using the standard representation of $\sGL(n,\C)$ on $\C^n,$ the data of a $\sGL(n,\C)$-Higgs bundle can be equivalently described by a pair $(\Ee,\Phi)$ where $\Ee=(E,\bar\p_E)$ is a rank $n$ holomorphic vector bundle on $\Sigma$ and $\Phi\in H^0(\Sigma,\End(\Ee)\otimes K)$ is a holomorphic endomorphism of $\Ee$ twisted by $K$.  
    Similarly, an $\sSL(n,\C)$-Higgs bundle is a pair $(\Ee,\Phi)$ where $\Ee$ is a rank $n$ holomorphic vector bundle with trivial determinant and $\Phi\in H^0(\End(\Ee)\otimes K)~$ is a {\em traceless} $K$-twisted endomorphism.
\end{Example}

\begin{Definition}\label{SL(n,C)stability}
A $\sGL(n,\C)$-Higgs bundle $(\Ee,\Phi)$ is called {\em stable} if for all $\Phi$-invariant subbundles $\Ff\subset\Ee$ we have $\frac{\deg(\Ff)}{\rk(\Ff)}<\frac{\deg(\Ee)}{\rk(\Ee)}~.$
An $\sSL(n,\C)$-Higgs bundle $(\Ee,\Phi)$ is 
\begin{itemize}
    \item {\em stable} if all $\Phi$-invariant subbundles $\Ff\subset\Ee$ satisfy $\deg(\Ff)<0,$
    \item {\em poly-stable} if $(\Ee,\Phi)=\bigoplus(\Ee_j,\Phi_j)$ where each $(\Ee_j,\Phi_j)$ is a stable $\sGL(n_j,\C)$-Higgs bundle with $\deg(\Ee_j)=0$ for all $j.$
\end{itemize}
 \end{Definition} 
We will also need the notion of stability for $\sSO(n,\C)$-Higgs bundles to simplify some proofs in Section \ref{PSp4R}. Let $Q$ be a nondegenerate symmetric bilinear form of $\C^n$ and define the group 
\[\sSO(n,\C)=\{A\in\sSL(n,\C)| A^TQA=Q\}~.\]
Using the standard representation of $\sSO(n,\C)$, a smooth principal bundle $\sSO(n,\C)$-bundle $P$ gives rise to a rank $n$ smooth vector bundle $E$ with trivial determinant bundle. Moreover, the nondegenerate symmetric form $Q$ defines an everywhere nondegenerate section $Q\in\Omega^0(S^2(E^*)).$ We will usually interpret the section $Q$ as a symmetric isomorphism $Q:E\to E^*.$ By nondegeneracy, $\det(Q):\Lambda^n E\to \Lambda^n E^*$ is an isomorphism, and defines a trivialization of the line bundle $(\Lambda^nE)^2.$

\begin{Proposition}
    An $\sSO(n,\C)$-Higgs bundle $(\Pp_{\sSO(n,\C)},\varphi)$ on $\Sigma$ is equivalent to a triple $(\Ee,Q,\Phi)$ where
\begin{itemize}
    \item $\Ee$ is a rank $n$ holomorphic vector bundle with $\Lambda^n\Ee\cong\Oo,$
    \item $Q\in H^0(\Sigma,S^2\Ee^*)$ is every nondegenerate,
    \item $\Phi\in H^0(\Sigma,\End(\Ee)\otimes K)$ and satisfies $\Phi^TQ+Q\Phi=0.$
\end{itemize}  
\end{Proposition}
Two $\sSO(n,\C)$-Higgs bundles $(\Ee,Q,\Phi)$ and $(\Ee',Q',\Phi')$ are isomorphic if there is a smooth bundle isomorphism $f:E\to E'$ with $f^TQ'f=Q$ and such that $f^*\bar\p_{E'}=\bar\p_E$ and $f^*\Phi'=\Phi.$ 
A holomorphic subbundle of $\Ff\subset\Ee$ is called \emph{isotropic} if $Q|_\Ff\equiv0.$ Note that an isotropic subbundle has rank at most $\floor{\frac{n}{2}}.$
\begin{Definition} \label{DEF SOn stability}
    An $\sSO(n,\C)$-Higgs bundle $(\Ee, Q, \Phi)$ is {\em stable} if for all $\Phi$-invariant {\em isotropic subbundles} $\Ff\subset\Ee$ we have $\deg(F)<0.$
\end{Definition}

In general, the notion of stability for a $\sG$-Higgs bundle is more subtle. When $\sG$ is a real form of a complex subgroup of $\sSL(n,\C)$ we have the following simplified notion of poly-stability \cite{HiggsPairsSTABILITY}. 
\begin{Definition}\label{DEF SLnC poly-stable is enough}
    Let $\sG$ be a real form of a semi-simple Lie subgroup of $\sSL(n,\C)$. A $\sG$-Higgs bundle $(\Pp_{\sH_\C},\varphi)$ is poly-stable if and only if the corresponding $\sSL(n,\C)$-Higgs bundle $(\Pp_{\sSL(n,\C)},\varphi)$ obtained via extension of structure group is poly-stable in the sense of Definition \ref{SL(n,C)stability}.
\end{Definition}
Although it is not immediately clear from the above definition, a stable $\sSO(n,\C)$-Higgs bundle is poly-stable in the sense of Definition \ref{DEF SLnC poly-stable is enough}.
Let $\Gg(P_{\sH_\C})$ be the gauge group of smooth bundle automorphisms of a principal $\sH_\C$-bundle $P_{\sH_\C}$.  
The group $\Gg(P_{\sH_\C})$ acts on the set of holomorphic structures on $P_{\sH_\C}$ and sections of $\Omega^{1,0}(\Sigma, P_{\sH_\C}[\fm_\C])$ by pullback. This action preserves the set of poly-stable $\sG$-Higgs bundles, and the orbits through poly-stable points are closed.

\begin{Definition} \label{DefModuliSpace}
Fix a smooth principal $\sH_\C$-bundle $P_{\sH_\C}$ on $\Sigma.$ The moduli space of $\sG$-Higgs bundle structures on $P_{\sH_\C}$ consists of isomorphism classes of poly-stable Higgs bundles with underlying smooth bundle $P_{\sH_\C}$ 
\[\Mm(\Sigma,P_{\sH_\C},\sG) = \{\text{poly-stable } \sG\text{-Higgs bundle structures on } P_{\sH_\C} \}/\Gg(P_{\sH_\C})~.\]
The union over the set of isomorphism classes of smooth principal $\sH_\C$-bundles on $\Sigma$ of the spaces $\Mm(\Sigma,P_{\sH_\C},\sG)$ will be referred to as the moduli space of $\sG$-Higgs bundles and denoted by $\Mm(\Sigma,\sG)$, or, when there is no confusion, by $\Mm(\sG)$. 
\end{Definition}

In fact, the space $\Mm(\Sigma,\sG)$ can be given the structure of a complex algebraic variety of complex dimension $(g-1)\dim(\sG)$ \cite{selfduality,SimpsonVHS,schmitt_2005}.
Moreover, we have the following fundamental result which allows one to go back and forth between statements about the Higgs bundle moduli space and the character variety. 
We will say more about how this correspondence works in Section \ref{minimal}.
\begin{Theorem}\label{THM: NAHC}
    Let $\Gamma$ be the fundamental group of a closed oriented surface $S$ and let $\sG$ be a real simple Lie group with maximal compact subgroup $\sH.$ For each choice of a Riemann surface structure $\Sigma\in\Teich(S),$ the moduli space $\Mm(\Sigma, \sG)$ of $\sG$-Higgs bundles on $\Sigma$ is homeomorphic to the $\sG$-character variety $\Xx(\Gamma,\sG)$. 
\end{Theorem}
\begin{Remark}
When $\sG$ is compact, Theorem \ref{THM: NAHC} was proven using the theory of stable holomorphic bundles by Narasimhan and Seshadri \cite{NarasimhanSeshadri} for $\sG=\sSU(n),$ and Ramanathan \cite{ramanathan_1975} in general. For $\sG$ complex, it was proven by Hitchin \cite{selfduality} and Donaldson \cite{harmoicmetric} for $\sG=\sSL(2,\C)$ and Simpson \cite{SimpsonVHS} and Corlette \cite{canonicalmetrics} in general using the theory of Higgs bundles and harmonic maps. For $\sG$ a general real reductive Lie group and appropriate notions of stability, Theorem \ref{THM: NAHC} was proven in \cite{HiggsPairsSTABILITY}.
\end{Remark}

It will be useful to know when the equivalence class of a $\sG$-Higgs bundle $(\Pp,\varphi)$ defines a smooth point or a singular point of the moduli space $\Mm(\sG).$ Moreover, it will be important to distinguish between ``mild'' singular points, i.e., orbifold points, and other singular points. Define the automorphism group of $(\Pp,\varphi)$ by
\[\Aut(\Pp,\varphi)=\{g\in\Gg(P_{\sH_\C})\ |\ g\cdot (\bar\p_\Pp,\varphi)=(\bar\p_\Pp,\varphi)\}~.\]
Note that the center $\Zz(\sG_\C)$ of $\sG_\C$ is equal to the intersection of the kernel of the representation $Ad: \sH_\C\to\sGL(\fm_\C)$ with the center of $\sH_\C.$
For every $\sG$-Higgs bundle $(\Pp,\varphi)$ we have $\Zz(\sG_\C)\subset\Aut(\Pp,\varphi).$  
The following proposition follows from Propositions 3.17 and 3.18 of \cite{HiggsPairsSTABILITY}.
\begin{Proposition}\label{Proposition: Orbifold points}
    For a simple Lie group $\sG,$ if $(\Pp,\varphi)$ is a poly-stable $\sG$-Higgs bundle which is stable as a $\sG_\C$-Higgs bundle, then $\Aut(\Pp,\varphi)$ is finite and the isomorphism class of $(\Pp,\varphi)$ in $\Mm(\sG)$ is an orbifold point of type $\Aut(\Pp,\varphi)\slash\Zz(\sG_\C).$ In particular, $(\Pp,\varphi)$ defines a smooth point of $\Mm(\sG)$ if and only if $\Aut(\Pp,\varphi)=\Zz(\sG_\C).$  
\end{Proposition}

We will use this proposition to determine when an poly-stable $\sSO_0(2,3)$-Higgs bundle defines a smooth point or orbifold point of the moduli space.

\subsection{The Lie groups $\sSO_0(2,3)$ and $\sSp(4,\R)$}\label{section the isomorphism}

Here we collect the necessary Lie theory for the groups of interest. In particular, we explain the isomorphism between the groups $\sP\sSp(4,\R)$ and $\sSO_0(2,3).$ 
\medskip

\noindent\textbf{The group $\sSp(4,\R)$:}  Consider the symplectic form $\Omega=\smtrx{0&\mathrm{Id}\\-\mathrm{Id}&0}$ on $\C^4.$ 
The symplectic group $\sSp(4,\C)$ consists of linear transformations $g\in\sGL(4,\C)$ such that $g^T\Omega g=\Omega.$ 
The Lie algebra $\fsp(4,\C)$ of $\sSp(4,\C)$ consists of matrices $X$ such that $X^T\Omega+\Omega X=0.$ Such an $X\in\fsp(4,\C)$ is given by $X=\smtrx{A&B\\C&-A^T}$ where $A,$ $B$ and $C$ are $2\times2$ complex matrices with $B$ and $C$ symmetric.

One way of defining the group $\sSp(4,\R)$ is as the subgroup of $\sSp(4,\C)$ consisting of matrices with real entries. 
However, when dealing with $\sSp(4,\R)$-Higgs bundles it will be useful to consider $\sSp(4,\R)$ as the fixed point set of a conjugation $\lambda$ which acts by 
\[\lambda(g)=\mtrx{0&\mathrm{Id}\\\mathrm{Id}&0}\overline g\mtrx{0&\mathrm{Id}\\\mathrm{Id}&0}.\]

The fixed points of the induced involution (also denoted by $\lambda$) on the Lie algebra $\fsp(4,\C)$ gives the Lie algebra $\fsp(4,\R)$ as the set of matrices $X=\smtrx{A&B\\C&-A^T}$ where $A, B$ and $C$ are $2\times2$ complex valued matrices with $A=-\overline A^T$ and $B=\overline{C}=B^T$. 
Since the conjugation $\lambda$ commutes with the compact conjugation $g\to \overline{g^{-1}}^T$ of $\sSp(4,\C)$, the composition defines the complexification of a Cartan involution $\theta:\fsp(4,\C)\to\fsp(4,\C)$ for $\fsp(4,\R)$. On the Lie algebra $\fsp(4,\C)$ the involution $\theta$ acts by 
\[\theta\left(\mtrx{A&B\\C&-A^T}\right)=\mtrx{A&-B\\-C&-A^T}.\] Thus, the complexification of the Cartan decomposition of $\fsp(4,\R)$ is given by
\begin{equation}
    \label{eq: complexified Cartan decomp Sp(4,R)}
\fsp(4,\C)=\fh_\C\oplus\fm_\C=\fgl(2,\C)\oplus S^2(V)\oplus S^2(V^*)
\end{equation}
where $S^2(V)$ denotes the symmetric product of the standard representation $V$ of $\sGL(2,\C)$. 
\medskip

\noindent\textbf{The group $\sSO_0(2,3)$:}
Fix positive definite quadratic forms $Q_2$ and $Q_3$ on $\R^2$ and $\R^3$ respectively and consider the signature $(2,3)$ form $Q=\smtrx{Q_2&\\&-Q_3}$ on $\R^5$. The group $\sSO(2,3)$ consists of matrices $g\in\sSL(5,\R)$ such that $g^TQg=Q.$ There are two connected components of $\sSO(2,3)$, and the connected component of the identity will be denoted by $\sSO_0(2,3).$ 

The Lie algebra $\fso(2,3)$ consists of matrices $X$ such that $X^TQ+QX=0.$ 
A matrix $X\in\fso(2,3)$ decomposes as 
$$\mtrx{A & Q_3^{-1}B^TQ_2\\B  & C},$$ 
where $B$ is a $3\times2$ matrix, $A$ is a $2\times 2$ matrix which satisfies $A^TQ_2+Q_2A=0$, and $C$ is a $3\times 3$ matrix which satisfies $C^TQ_3+Q_3C=0$. Thus, the Cartan decomposition is given by
\[\fso(2,3)=\fh\oplus\fm=(\fso(2)\oplus\fso(3))\oplus\Hom(\R^2,\R^3).\] 
Complexifying this gives a decomposition of $\sH_\C=\sSO(2,\C)\times\sSO(3,\C)$-modules
\begin{equation}
    \label{eq: complexified Cartan decomp SO(2,3)}
    \fso(5,\C)=\fh_\C\oplus\fm_\C=(\fso(2,\C)\times\fso(3,\C))\oplus\Hom(V,W)
\end{equation}
where $V$ and $W$ denote the standard representations of $\sSO(2,\C)$ and $\sSO(3,\C)$ on $\C^2$ and $\C^3$ respectively.  \medskip

\noindent\textbf{The isomorphism $\sP\sSp(4,\R)\cong\sSO_0(2,3)$:}
Let $U$ be a 4 dimensional real vector space with a symplectic form $\Omega\in\Lambda^2U^*.$ 
The $6$ dimensional vector space $\Lambda^2U^*$ has a natural orthogonal structure given by $\langle a,b\rangle=C$ where $a\wedge b=C\Omega\wedge\Omega$. Moreover, the signature of this orthogonal structure is $(3,3).$ 
Since $\Omega\in\Lambda^2U^*$ has norm $1$, the orthogonal complement of the subspace spanned by $\Omega$ defines a 5 dimensional orthogonal subspace with a signature $(2,3)$ inner product.
This defines a surjective map $\sSp(4,\R)\to\sSO_0(2,3)$ with kernel $\pm Id.$ Since $\pm Id$ is the center of $\sSp(4,\R)$, the group $\sSO_0(2,3)$ is isomorphic to the adjoint group $\sP\sSp(4,\R)$.

The universal cover of $\sSO(5,\C)$ is the spin group $\sSpin(5,\C).$ The split real form of $\sSpin(5,\C)$ will be denoted by $\sSpin(2,3)$. The isomorphism $\sP\sSp(4,\R)\cong\sSO_0(2,3)$ defines an isomorphism between $\sSp(4,\R)$ and the connected component of the identity of the spin group $\sSpin_0(2,3)$.

\section{Complex orthogonal bundles}   \label{Orthogonal_bundles}

Holomorphic $\sO(n,\C)$-bundles will be an important tool in the next sections. We describe their main properties and, for $n=2$, we describe their parameter spaces.

\subsection{General properties}
For $Q$ a symmetric nondegenerate form on $\C^n$, define 
\[\sO(n,\C)=\{A\in\sGL(n,\C)|A^TQA=Q\}~.\]
The standard representation $\sO(n,\C)$ on $\C^n$ allows one to describe a principal $\sO(n,\C)$-bundle in terms of a rank $n$ complex vector bundle $V$. 
On $V$ the form $Q$ defines a global section $Q\in \Omega^0(\Sigma,S^2(V^*))$ which is everywhere nondegenerate. 
A holomorphic structure on the orthogonal bundle $(V,Q)$ is a holomorphic structure $\bar\p_V$ on $V$ for which $Q$ is holomorphic. 
The determinant $\det(V)$ is a holomorphic line bundle, and non-degeneracy of $Q$ is equivalent to $\det(Q)$ never vanishing. 
Thus, $\det(V)^2\cong\Oo$ and $\det(V)$ is a holomorphic $\sO(1,\C)$-bundle. 
In particular, holomorphic $\sO(1,\C)$-bundles are exactly the $2^{2g}$ holomorphic line bundles $L$ with $L^2=\Oo$.

There are two main topological invariants of $\sO(n,\C)$-bundles on a Riemann surface: the first and second Stiefel-Whitney classes:
\[\xymatrix{sw_1(V,Q_V) \in H^1(\Sigma,\Z_2) = \Z_2^{2g}&\text{and}&sw_2(V,Q_V) \in H^2(\Sigma,\Z_2) = \Z_2}~.\]
The first Stiefel-Whitney class is the obstruction to reducing the structure group to $\sSO(n,\C)$. Hence, $sw_1(V,Q_V)=sw_1(\det(V),\det(Q_V))$ vanishes if and only if $\det(V)=\Oo.$
The class $sw_1\in H^1(\Sigma,\Z_2)$ can also be interpreted as a $\Z_2$-bundle $\pi:\Sigma_{sw_1}\ra\Sigma$, i.e. an unramified double cover, which is connected if and only if $sw_1 \neq 0.$ The cover $\Sigma_{sw_1}$ inherits the complex structure from $\Sigma$ by pullback, and it is a Riemann surface of genus $g'=2g-1$. 
The second Stiefel-Whitney class is the obstruction to lifting the structure group from $\sO(n,\C)$ to $\sPin(n,\C)$. 

The following Whitney sum formula will help us to compute these invariants:  
\begin{equation} \label{WhitneySum}
\begin{array}{lcr}
    sw_1(V\oplus W) = sw_1(V) + sw_1(W)~,\\
sw_2(V\oplus W) = sw_2(V) + sw_2(W) + sw_1(V) \wedge sw_1(W)~.
\end{array}
\end{equation}

\subsection{Bundles of rank $2$ with vanishing first Stiefel-Whitney class}   \label{orthogonal_bundles_vanishing}

We will now recall Mumford's \cite{MumO2Bun} classification of the holomorphic $\sSO(2,\C)$ and $\sO(2,\C)$-bundles. These parameter spaces will be respectively denoted by $\Bb(\Sigma,\sSO(2,\C))$ and $\Bb(\Sigma,\sO(2,\C))$, and shortened to $\Bb(\sSO(2,\C))$ and $\Bb(\sO(2,\C))$ when possible.

We will write a holomorphic $\sSO(2,\C)$-bundle as $(V,Q_V,\omega)$, where $(V,Q_V)$ is a holomorphic $\sO(2,\C)$-bundle and $\omega$ is a holomorphic volume form compatible with $Q_V$. The form $\omega$ can be seen as a non-zero holomorphic section of $\Lambda^2V=\det(V) \cong \Oo$. Such an $\sSO(2,\C)$-bundle can be described explicitly since $\sSO(2,\C)\cong\C^*$:
\[\sSO(2,\C)\cong\left\{A\in\sSL(2,\C)\ \big|\ A^T\mtrx{0&1\\1&0}A=\mtrx{0&1\\1&0}\right\}=\left\{\mtrx{e^\lambda&0\\0&e^{-\lambda}}
\ \big |\ \lambda\in\C
\right\}~.\]

Thus, the isomorphism class of $(V,Q_V,\omega)$ is determined by a holomorphic line bundle $L$:
\begin{equation}\label{SO(2,C) holomorphic bundles}
    (V,Q_V,\omega)=\left(L\oplus L^{-1}~,~ \mtrx{0&1\\1&0}~,~ \mtrx{0&1\\-1&0}\right),
\end{equation}
here $\omega=\smtrx{0&1\\-1&0}$ is seen as a skew symmetric bilinear form on  $L\oplus L^{-1}.$

The degree of $L$ provides a topological invariant of $\sSO(2,\C)$-bundles whose reduction modulo $2$ is the second Stiefel-Whitney class. The parameter space $\Bb(\sSO(2,\C))$ is then the Picard group of holomorphic line bundles on $\Sigma$, 
\[\Bb(\sSO(2,\C)) =\Pic(\Sigma)~.\]
In particular, it is a disjoint union of countably many tori of complex dimension $g$.

The classification of $\sO(2,\C)$-bundles is more complicated and will depend on the values of the Stiefel-Whitney classes. Denote by $\Bb_{sw_1}(\sO(2,\C))$ and $\Bb_{sw_1}^{sw_2}(\sO(2,\C))$ the subsets containing bundles with fixed values of $sw_1$ or of $sw_1$ and $sw_2$. 

Orthogonal bundles in the subspace $\Bb_0(\sO(2,\C))$ admit two different $\sSO(2,\C)$-structures: 
\[\left(L\oplus L^{-1}, \mtrx{0&1\\1&0}, \mtrx{0&1\\-1&0}\right)\ \ \ \text{and}\ \ \  \left(L^{-1}\oplus L, \mtrx{0&1\\1&0}, \mtrx{0&1\\-1&0}\right)~.\]
This equivalence corresponds to an action of $\Z_2$ on $\Pic(\Sigma)$ given by $L\mapsto L^{-1}$. Thus
\[\Bb_0(\sO(2,\C)) = \Pic(\Sigma)/\Z_2~.\]
Since $sw_2(V,Q_V)=|\deg(L)|\ (\text{mod } 2)$, for an $\sO(2,\C)$-bundle with vanishing first Stiefel-Whitney class, the second Stiefel-Whitney class lifts to an $\N$-invariant. Denote by $\Bb_{0,d}(\sO(2,\C))$ the subspace containing $\sO(2,\C)$-bundles with $|\deg(L)|=d$. 

When $d>0$, we can choose $L\in\Pic^d(\Sigma)$ which represents a point in $\Bb_{0,d}(\sO(2,\C))$. Thus,
for $d>0$, $\Bb_{0,d}(\sO(2,\C))$ is identified with the torus $\Pic^d(\Sigma)$ of degree $d$ line bundles. 
When $d=0$, both $L$ and $L^{-1}$ have degree zero. Hence, $\Bb_{0,0}(\sO(2,\C)) = \Pic^0(\Sigma)/\Z_2$ is singular with $2^{2g}$ orbifold points corresponding to the line bundles $L\in\Pic^0(\Sigma)$ with $L\cong L^{-1}$ or, equivalently, $L^2\cong\Oo$.
\[\Bb_{0,d}(\sO(2,\C)) = \begin{dcases}
    \Pic^d(\Sigma)& \text{if } d>0,\\\Pic^0(\Sigma)\slash\Z_2& \text{if } d=0.
\end{dcases}\]

\subsection{Bundles of rank $2$ with non-vanishing first Stiefel-Whitney class}   \label{orthogonal_bundles_nonvanishing}

Now consider the spaces $\Bb_{sw_1}(\Sigma,\sO(2,\C))$ for a $sw_1\neq0.$ Note that it is nonempty.

\begin{Lemma}   \label{sw_2_non_empty}
For every $sw_1\in H^1(\Sigma, \Z_2) \setminus \{0\}$ and $sw_2 \in H^2(\Sigma, \Z_2)$, the space $\Bb_{sw_1}^{sw_2}(\Sigma,\sO(2,\C))$ is not empty.  
\end{Lemma}
\begin{proof}
Fix $sw_1\neq0.$ If $sw_2=0$, consider the bundle $(V,Q_V)=\left(L\oplus \Oo, \smtrx{1&0\\0&1}\right)$, with $L^2=\Oo$ and $sw_1(L)=sw_1$. By \eqref{WhitneySum}, $sw_1(V,Q_V) = sw_1$ and $sw_2(V,Q_V)=0$.

For $sw_2\neq0$, by non-degeneracy of the cup product there exists a $t\in H^1(\Sigma, \Z_2)$ such that $sw_1 \wedge t  \neq0$.
Now consider the bundle $(V,Q_V)=\left(L_1\oplus L_2, \smtrx{1&0\\0&1}\right)$ with $L_1^2=L_2^2=\Oo$ and $sw_1(L_1)=sw_1 + t$ and $sw_1(L_2)=t$. By \eqref{WhitneySum}, $sw_1(V,Q_V) = sw_1$ and $sw_2(V,Q_V)= sw_1(L_1)\wedge sw_1(L_2) = (sw_1+t) \wedge t =sw_1\wedge t \neq 0.$
\end{proof}

For $sw_1 \in H^1(\Sigma,\Z_2)\setminus \{0\}$, consider the double cover $\pi:\Sigma_{sw_1}\to\Sigma$ of genus $g'=2g-1$. We have the following pullback map
\[\pi^*: \Bb_{sw_1}(\Sigma,\sO(2,\C)) \ra  \Bb(\Sigma_{sw_1},\sO(2,\C))~.\]
 For $\iota$ the covering involution of $\Sigma_{sw_1},$ consider the following group homomorphism: 
\[Id\otimes\iota^*:\xymatrix@R=0em{\Pic(\Sigma_{sw_1})\ar[r] & \Pic(\Sigma_{sw_1})\\M\ar@{|->}[r]& M\otimes\iota^*M ~.}\]
We will denote the kernel $Id\otimes\iota^*$ by $\Prym(\Sigma_{sw_1}) = \ker(Id\otimes\iota^*)$ because it is closely related to the Prym variety of the covering (see below). 
Note that $\iota$ induces a $\Z_2$ action on $\Prym(\Sigma_{sw_1})$ by $\iota(M)=M^{-1}$. Thus, there is a natural injective map 
\[\xymatrix{\Prym(\Sigma_{sw_1})\slash\Z_2\ar[r]&\Bb_0(\Sigma_{sw_1},\sO(2,\C))}~.\]

\begin{Proposition} (see \cite{MumO2Bun})
The map $\pi^*$ maps bijectively onto $\Prym(\Sigma_{sw_1})\slash\Z_2$:
\[\pi^*(\Bb_{sw_1}(\Sigma,\sO(2,\C)))=\Prym(\Sigma_{sw_1})\slash\Z_2\subset\Bb_0(\Sigma_{sw_1},\sO(2,\C))~.\]
\end{Proposition}
We provide a proof of this and the following propositions because the details are important for the next section. 
\begin{proof}
Let $(V,Q_V)\in\Bb_{sw_1}(\Sigma,\sO(2,\C))$. By the geometric interpretation of $sw_1$ as a double cover, $(\pi^*V,\pi^*Q_V) \in \Bb_0(\Sigma_{sw_1},\sO(2,\C))$. In particular, there is a line bundle $M \in \Pic(\Sigma_{sw_1})$ such that
\[(\pi^*V,\pi^*Q_V)=\left(M\oplus M^{-1}, \mtrx{0&1\\1&0}\right)~.\]
Since $M\oplus M^{-1}$ is a pullback, it is isomorphic to $\iota^*(M)\oplus \iota^*(M^{-1})$. Thus either $M=\iota^*(M)$ or $M=\iota^*(M^{-1})$.
But, if $\iota^*M=M$, then $(V,Q_V)$ would have $sw_1=0$. 

Every line bundle $M$ satisfying $M=\iota^*(M^{-1})$ can be obtained in this way, since we can construct an $\sO(2,\C)$-bundle $(V,Q_V)=(\pi_* M,\pi_* \iota^*)$ by pushforward. Since $\Sigma_{sw_1}\ra\Sigma$ is unramified, $\pi^*\pi_*(M)=M\oplus \iota^*M$.
\end{proof}

We now only need to understand the subspace $\Prym(\Sigma_{sw_1})$. 

\begin{Lemma} (see \cite{MumO2Bun}).
Every $M \in \Prym(\Sigma_{sw_1})$ admits a meromorphic section $s$ such that $s \otimes \iota^*s = 1$.
\end{Lemma}
\begin{proof}
By Tsen's theorem (see \cite{Lang52}), for every meromorphic function $f$ on $\Sigma$, there exists a meromorphic function $g$ on $\Sigma_{sw_1}$ such that $g \iota^*g = \pi^*f$.  
We start with a meromorphic section $t$ of $M$. Then $t \otimes \iota^*t$ is a meromorphic function on $\Sigma_{sw_1}$ that is the pullback of a function on $\Sigma$. So, we can find a meromorphic function $g$ on $\Sigma_{sw_1}$ such that $t \otimes \iota^*t = g\iota^*g$. Define $s = t g^{-1}$, this is again a meromorphic section of $M$ and $s \otimes \iota^*(s) = 1$.
\end{proof}

Consider the group homomorphism:
\[\Psi:\xymatrix@R=0em{\Pic(\Sigma_{sw_1}) \ar[r]& \Pic^0(\Sigma_{sw_1})\\
 L \ar[r]& L\otimes \iota^* L^{-1}}~.\]

\begin{Lemma}
There is an exact sequence:
\[\xymatrix{ 0 \ar[r] &\Z_2 \ar[r] &\Pic(\Sigma)  \ar[r]^{\pi^*}  & \Pic(\Sigma_{sw_1}) \ar[r]^{\Psi}& \Pic^0(\Sigma_{sw_1}) \ar[r]^{Id\otimes\iota^*} & \Pic(\Sigma_{sw_1})}~.\]
\end{Lemma}
\begin{proof}
The image of $\Psi$ is in $\ker(Id\otimes\iota^*)$ since $(L\otimes \iota^* L^{-1})\otimes \iota^*(L\otimes \iota^* L^{-1})=\Oo $. Moreover, if $M \in \ker(Id\otimes\iota^*)$, we can find a meromorphic section $s$ of $M$ such that $s\otimes \iota^*s =1$. Let $D$ be the divisor of the zeros of $s$ (but not the poles, so that $D(s) = D - \iota(D)$). Now, since the line bundle $L(D)$ has the property that $L(D)\iota^*L(D)^{-1} = M$, $M$ is in the image of $\Psi$.
The kernel of $\Psi$ can also be computed explicitly: if $L\otimes \iota^* L^{-1} = \Oo$, then $L = \iota^* L$, hence $L$ is the pullback of a line bundle on $\Sigma$, and vice versa. Hence $\ker(\Psi) = \pi^*(\Pic(\Sigma))$. 
\end{proof}

\begin{Proposition}\label{Prop Prym disconnected}
The group $\Prym(\Sigma_{sw_1})$ is the disjoint union of two homeomorphic connected components $\Prym^0(\Sigma_{sw_1})$ and $\Prym^1(\Sigma_{sw_1})$ so that 
\[\Psi(\Pic^i(\Sigma_{sw_1})) = \Prym^{i (\text{mod } 2)}(\Sigma_{sw_1})~.\]
Moreover, $\Prym^0(\Sigma_{sw_1})$ is an abelian variety of dimension $g-1$. 
\end{Proposition}
\begin{proof}
The kernel of $\Psi$ consists of line bundles which are pullbacks, in particular, they all have even degree. Thus, the components $\Pic^{2k}(\Sigma_{sw_1})$ all have the same image. Similarly, the components $\Pic^{2k+1}(\Sigma_{sw_1})$ all have the same image which is disjoint from the image of the even components.  
The space $\Prym^0(\Sigma_{sw_1})$ is the quotient of the abelian variety $\Pic(\Sigma_{sw_1})$ by the abelian subvariety $\pi^*(\Pic(\Sigma))$, so it is an abelian variety of dimension $g'-g = g-1$.
\end{proof}

The abelian variety  $\Prym^0(\Sigma_{sw_1})$ is usually called the \emph{Prym variety} of the cover $\pi:\Sigma_{sw_1}\to\Sigma$~.

\begin{Proposition}
If $(V,Q_V) \in \Bb_{sw_1}^{sw_2}(\Sigma,\sO(2,\C))$, then its pullback to the double cover $\Sigma_{sw_1}$ defines a point in $\Prym^{sw_2}(\Sigma_{sw_1})$. This gives a bijection
\[\pi^*: \Bb_{sw_1}^{sw_2}(\Sigma,\sO(2,\C)) \ra \Prym^{sw_2}(\Sigma_{sw_1})/\Z_2~.\]
\end{Proposition}
\begin{proof}
Consider the bundle $(V,Q_V)=\left(L\oplus \Oo, \smtrx{1&0\\0&1}\right)$, with $L^2=\Oo$ and $sw_1(L)=sw_1$. By the proof of Lemma \ref{sw_2_non_empty}, $(V,Q_V)$ is in $\Bb_{sw_1}^{0}(\Sigma,\sO(2,\C))$ and we have $(\pi^*V,\pi^*Q_V) = \left(\Oo\oplus \Oo, \smtrx{1&0\\0&1}\right)$. Hence, $(V,Q_V)$ is in $\Prym^0(\Sigma_{sw_1})$. Since $sw_2$ is constant on connected components, all the points of $\Prym^0(\Sigma_{sw_1})$ have $sw_2=0$. 
By Lemma \ref{sw_2_non_empty}, there exists a bundle in $\Bb_{sw_1}^{1}(\Sigma,\sO(2,\C))$. This bundle must pullback to $\Prym^1(\Sigma_{sw_1})$, so all the points in $\Prym^1(\Sigma_{sw_1})$ must have $sw_2=1$.
\end{proof}

The space $\Prym^{sw_2}(\Sigma_{sw_1})/\Z_2$ is singular, it has $2^{2g-2}$ orbifold points corresponding to the fixed points of the $\Z_2$-action. They correspond to poly-stable $\sO(2,\C)$-bundles, who split orthogonally as a direct sum of two distinct $\sO(1,\C)$-bundles: 
$$(V,Q_V) = \left(L_1\oplus L_2, \mtrx{1&0\\0&1}\right)~, $$
with $L_1^2 = L_2^2 = \Oo$.

Summarizing, the space $\Bb(\sO(2,\C))$ splits in the following connected components:
\[
\bigsqcup\limits_{d\in\N}
\Bb_{0,d}(\sO(2,\C)) \ \sqcup \ 
\bigsqcup\limits_{\substack{sw_1\neq0\\ sw_2}}
\Bb_{sw_1}^{sw_2}(\sO(2,\C)) ~, \]
and every one of these pieces can be described explicitly. In the next section the parameter spaces of maximal $\sP\sSp(4,\R)$-Higgs bundles are described and we will see that its connected components are indexed by a finite subset of $\pi_0(\Bb(\sO(2,\C)))$. 

Topologically, a connected component of $\Bb(\sO(2,\C))$ is a torus or the quotient of a torus by the inversion involution ($x\mapsto x^{-1}$). Their rational cohomology is given by the following.
\begin{Proposition}
     \label{Prop: cohomology of Torus mod inversion} The cohomology of each component of $\Bb(\Sigma,\sO(2,\C))$ is 
     \begin{itemize}
         \item if $d\neq0$, then $H^*(\Bb_{0,d}(\Sigma,\sO(2,\C)))=H^*((S^1)^{2g},\Q)$,
         \item $H^j(\Bb_{0,0}(\Sigma,\sO(2,\C)))=\begin{cases}
             H^{j}((S^1)^{2g},\Q) & \text{if\ j\ is even,}\\
             0&\text{otherwise,}
         \end{cases}$
         \item if $sw_1\neq0,$ then $H^j(\Bb_{sw_1}^{sw_2}(\Sigma,\sO(2,\C)))=\begin{cases}
             H^{j}((S^1)^{2g-2},\Q) & \text{if\ j\ is even,}\\
             0&\text{otherwise.}
         \end{cases}$
     \end{itemize}
 \end{Proposition} 
 \begin{proof}
 For $d\neq0,$ the space $\Bb_{0,d}(\Sigma,\sO(2,\C))$ is given by $\Pic^d(\Sigma)$ and hence it is a torus of dimension $2g.$ The component $\Bb_{0,0}(\Sigma,\sO(2,\C))$ is the quotient of a $2g$-dimensional torus by inversion and the components $\Bb_{sw_1}^{sw_2}(\Sigma,\sO(2,\C))$ are quotients of a $(2g-2)$-dimensional torus by inversion. 

Given a CW-complex $X$ with an action of a finite group $\Delta,$ there is an isomorphism between the cohomology of the quotient $X/\Delta$ and the $\Delta$-invariant cohomology of $X$ (see \cite[section 2]{SymmetricProductsofAlgebraicCurves}, where the author summarizes chapter 5 of \cite {GrothendickTohoku})
\[H^*(X,\Q)^\Delta= H^*(X/\Delta,\Q)~.\]
Since, the $\Z_2$ action by inversion on a torus $(S^1)^{2m}$ acts on the cohomology group $H^j((S^1)^{2m},\Q)$ by $(-1)^j,$ the result follows.
 \end{proof}

\begin{Proposition}\label{Prop: no isotropic subbundles}
    A holomorphic $\sO(2,\C)$-bundle on $\Sigma$ with nonzero first Stiefel-Whitney class has no holomorphic isotropic line sub-bundles. 
\end{Proposition}
\begin{proof}
Let $(V, Q_V)$ be a holomorphic $\sO(2,\C)$-bundle on $\Sigma$, and suppose $L\subset V$ is an isotropic line sub-bundle. In every fiber $V_z$ there are two isotropic lines. Denote by $m_z$ the isotropic line that is not in $L$. The union of all the $m_z$ is a holomorphic line sub-bundle $M$, and $V = L \oplus M$. Hence $\det(V) = L M$. Since $L$ and $M$ are nowhere perpendicular, $Q_V$ defines a holomorphic isomorphism between $M$ and $L^* = L^{-1}$. This implies $\det(V) = \Oo$ and $sw_1(V,Q_V) = 0$.   
\end{proof}

\subsection{Gauge transformations}  \label{orth_gauge_transf}
We can now easily determine the group of holomorphic gauge transformations $\mathcal{H}_{\sO(2,\C)}(\Vv,Q_V)$ of an $\sO(2,\C)$-bundle.

In the case when $sw_1(V,Q_V)=0$, for $\deg(L)\geq 0$, we can write 
\[(\Vv,Q_V)=\left(L\oplus L^{-1}, \mtrx{0&1\\1&0}\right)~.\]
 In this splitting every holomorphic gauge transformation can be written as a matrix
$ \smtrx{a&b\\c&d} $
where $a,d\in \C$, $b\in H^0(\Sigma,L^{2})$, $c\in H^0(\Sigma,L^{-2})$. There are two cases:
\begin{itemize}
    \item If $L^2=\Oo$, then $L\cong L^{-1}$ and $b,c\in \C$. In this case, $\Hh_{\sO(2,\C)}\cong\sO(2,\C)$:
\begin{equation}\label{EQ HO2 sw1=0 M=M^-1}
    \mathcal{H}_{\sO(2,\C)}(\Vv,Q_V)= \left\{\mtrx{a&0\\0&a^{-1}} \ |\ a\in \C^*\right\} \cup \left\{\mtrx{0&b\\b^{-1}&0}\ |\ b\in \C^*\right\}~.
\end{equation}
\item If $L^2\neq\Oo$, then $b = 0$. The condition that the matrix preserves $Q_V$ implies that $c=0$, and $a d = 1$. In this case, 
\begin{equation}
    \label{EQ HO2 sw1=0}\mathcal{H}_{\sO(2,\C)}(\Vv,Q_V) = \C^* \cong
\sSO(2,\C) = \left\{\mtrx{a&0\\0&a^{-1}} \ |\ a\in \C^*\right\}~.
\end{equation}
 \end{itemize}

 When $sw_1(V,Q_V) \neq 0$, we can pullback $(\Vv,Q_V)$ to the double cover $\Sigma_{sw_1}$:
\[(\pi^*\Vv,\pi^*Q_V) = \left(M\oplus M^{-1}, \mtrx{0&1\\1&0}\right)~. \] 
Every gauge transformation of $(\Vv,Q_V)$ induces a $\iota$-invariant gauge transformation of $(\pi^*\Vv,\pi^*Q_V)$. Written in matrix form, the condition of $\iota$-invariance becomes
\[\mtrx{a&b\\c&d} = \mtrx{0&1\\1&0}  \mtrx{a&b\\c&d}   \mtrx{0&1\\1&0} =  \mtrx{d&c\\b&a} ~.\]
This implies $a=d$ and $b=c$. Together with the condition of preserving $Q$, this gives only 4 possible elements:
\[\mtrx{1&0\\0&1}~~,~~~~ \mtrx{-1&0\\0&-1}~~,~~~~ \mtrx{0&1\\1&0}~~~~ \text{and}~~~~\mtrx{0&-1\\-1&0} ~.\]

If $M^2 = \Oo$, we have $\mathcal{H}_{\sO(2,\C)}(\Vv,Q_V) = \Z_2 \oplus\Z_2$. In this case, $(\Vv,Q_V)$ splits as an orthogonal direct sum $L_1\oplus L_2$ with $L_1^2 = L_2^2 = \Oo$ and $L_1\neq L_2$.
When $M^2 \neq \Oo$, only diagonal elements are possible, thus, in this case $\mathcal{H}_{\sO(2,\C)}(\Vv,Q_V) = \Z_2$.

\section{$\sP\sSp(4,\R)$-Higgs bundles}     \label{PSp4R}

In this section, we describe the moduli space of $\sP\sSp(4,\R)$-Higgs bundles. To do this, we use the isomorphism of $\sP\sSp(4,\R)$ with $\sSO_0(2,3)$ described in Section \ref{section the isomorphism}. After some set up, we prove Theorems \ref{THM d>0}, \ref{thm:zero_component} and \ref{THM HiggsParamsw1sw2orbifold} parameterizing all connected components of maximal $\sP\sSp(4,\R)$-Higgs bundles.

\subsection{General definition}

Using the complexified Cartan decomposition \eqref{eq: complexified Cartan decomp SO(2,3)} and Definition \ref{DEF Higgs bundle}, an $\sSO_0(2,3)$-Higgs bundle is a pair $(\Pp,\varphi)$, where $\Pp$ is a holomorphic principal $\sSO(2,\C)\times \sSO(3,\C)$-bundle and $\varphi\in H^0(\Sigma, \Pp[\Hom(\C^2,\C^3)]\otimes K)$. 
The vector bundle $\Ee = \Pp[\C^2 \oplus \C^3]$ associated to the standard representations of $\sSO(2,\C)\times\sSO(3,\C)$  splits as a direct sum $\Ee = \Vv \oplus \Ww$, where $\Vv$ and $\Ww$ are respectively holomorphic vector bundles of rank $2$ and $3$, with holomorphic orthogonal structures $Q_V$ and $Q_W$ and with trivial determinants $\det(\Vv) = \det(\Ww) = \Oo$. 
In this notation, $\varphi\in H^0(\Sigma, \Hom(\Vv,\Ww)\otimes K)$.

The $\sSO(2,\C)$-bundle $(\Vv,Q_V)$ splits holomorphically as a direct sum:
\begin{equation}
    (\Vv,Q_V)=\left(L\oplus L^{-1}, \mtrx{0&1\\1&0}\right)~.
\end{equation}
The splitting of $\Vv$ allows us to split the Higgs field:  $\varphi = (\gamma, \beta)$, where
\[\xymatrix{\gamma\in H^0(\Sigma,L^{-1}\otimes \Ww\otimes K)&\text{and}&\beta\in H^0(\Sigma, L\otimes \Ww\otimes K)}~.\]

\begin{Definition}\label{DEF of SO(2,3) Higgs bundle}
    An $\sSO_0(2,3)$-Higgs bundle is a tuple $(L,(\Ww,Q_W),\beta,\gamma)$ 
    where
    \begin{itemize}
        \item $L$ is a holomorphic line bundle 
        and $(\Ww,Q_W)$ is a holomorphic rank three holomorphic orthogonal vector bundle with $\det(\Ww)=\Oo.$ 
        \item $\gamma\in H^0(\Sigma,L^{-1}\otimes \Ww\otimes K)$ and $\beta\in H^0(\Sigma, L\otimes \Ww\otimes K).$
    \end{itemize}
\end{Definition}

The $\sSL(5,\C)$-Higgs bundle associated to an $\sSO_0(2,3)$-Higgs bundle determined by  $(L,(\Ww,Q_W),\beta,\gamma)$ is 
\begin{equation}\label{SL(5,C) Higgs of SO(2,3) Higgs}
    (\Ee,\Phi)=\left(L\oplus \Ww\oplus L^{-1}, \mtrx{0&\beta^T&0\\\gamma&0&\beta\\0&\gamma^T&0}\right).
\end{equation}
Here $\beta$ and $\gamma$ are interpreted as holomorphic bundle maps 
\[\xymatrix{\beta:L^{-1}\to \Ww\otimes K &\text{and}&\gamma:L\to \Ww\otimes K~,}\] and $\beta^T=\beta^*\circ Q_W:\Ww\to LK$ and $\gamma^T=\gamma^*\circ Q_W:\Ww\to L^{-1}K$.
\begin{Remark}
Since $\Ee=L\oplus \Ww\oplus L^{-1}$ has a holomorphic orthogonal structure \[Q=\mtrx{0&0&1\\0&-Q_W&0\\1&0&0}\]with respect to which $\Phi^TQ+Q\Phi=0,$ $(\Ee,Q,\Phi)$ is an $\sSO(5,\C)$-Higgs bundle. 
\end{Remark}

To construct the moduli space we need to restrict our attention to the tuples $(L,(\Ww,Q_W),\beta,\gamma)$ that give rise to poly-stable Higgs bundles, we will call them \emph{poly-stable tuples}. Recall $(L,(\Ww,Q_W),\beta,\gamma)$ is a poly-stable tuple if the $\sSL(5,\C)$-Higgs bundle \eqref{SL(5,C) Higgs of SO(2,3) Higgs} is poly-stable in the sense of Definition \ref{SL(n,C)stability}. 

The moduli space of $\sSO_0(2,3)$-Higgs bundles on $\Sigma$ can be described as
\[\Mm(\sSO_0(2,3)) = \{ (L,(\Ww,Q_W),\beta,\gamma) \ |\ \text{ poly-stable tuples }\} / \sim~.\]

For $\sSO_0(2,3)$-Higgs bundles, there are two topological invariants, the {\em Toledo number} $\tau = \deg L \in \Z$ and the second Stiefel-Whitney class $ sw_2(\Ww,Q_W)$.

\begin{Lemma} \label{LemmaGamma}
Let $(L,(\Ww,Q_W),\beta,\gamma)$ be a poly-stable tuple. If $\deg(L)>0$, then $\gamma \neq 0$. Moreover, if $\deg(L)>g-1$, then $\gamma^T \circ \gamma \neq 0$. 
\end{Lemma}
\begin{proof}
We will interpret $\gamma$ as a map $\gamma:L K^{-1}\to \Ww$, so that $\gamma^T \circ \gamma:L K^{-1} \to L^{-1} K$.
First note that if $\gamma=0$ and $\deg(L) > 0$, then $L$ would be an invariant subbundle of positive degree. This contradicts poly-stability. 

The image $\gamma(L K^{-1})$ is contained in a line subbundle $M\subset W$, with $\deg(M) \geq \deg(L) - 2g + 2$. The kernel of $\gamma^T$ is orthogonal to the image of $\gamma$. So, outside the zeros of $\gamma$, $\ker(\gamma^T) = M^\perp$. Let's assume now that $\gamma^T \circ \gamma = 0$, hence $M \subset M^\perp$, i.e. $M$ is isotropic. In this case, $M^\perp / M$ has degree zero since it is an $\sO(1,\C)$-bundle. This implies $\deg(M^\perp) = \deg(M)$. Now $M^\perp \oplus L$ is a $\varphi$-invariant subbundle and $\deg(M^\perp \oplus L) \geq 2\deg(L)-2g+2$. By poly-stability, we have $\deg(L)\leq g-1$.
\end{proof}

\begin{Proposition}\label{Prop MW inequality}
    For poly-stable $\sSO_0(2,3)$-Higgs bundles, the Toledo number satisfies the Milnor-Wood inequality $|\tau|\leq 2g-2.$
\end{Proposition}  

\begin{proof}
Consider a poly-stable $\sSO_0(2,3)$-Higgs bundle with $\tau>g-1$. By Lemma \ref{LemmaGamma},  poly-stability forces $\gamma^T\circ\gamma\in H^0(L^{-2}K^2)\setminus\{0\},$ and hence $\deg(L)\leq 2g-2.$ For $\tau<1-g,$ similar considerations imply $\deg(L)\geq-2g+2.$
\end{proof}

For every $\tau \in \Z$, let $\Mm^{\tau}(\sSO_0(2,3))\subset\Mm(\sSO_0(2,3))$ denote the subspace containing Higgs bundles with Toledo number $\tau$. The Milnor-Wood inequality gives a decomposition of the moduli space as:
\[\Mm(\sSO_0(2,3))=\bigsqcup\limits_{|\tau|\leq 2g-2}\Mm^{\tau}(\sSO_0(2,3)) ~.\]

We can subdivide these subspaces further. Namely, for every $sw_2 \in H^2(\Sigma,\Z_2)$, let $\Mm^{\tau,sw_2}(\sSO_0(2,3))$ denote the subspace of $\Mm^{\tau}(\sSO_0(2,3))$ containing Higgs bundles with second Stiefel-Whitney class $sw_2.$
As these discrete invariants vary continuously, each $\Mm^{\tau,sw_2}(\sSO_0(2,3))$ is a union of connected components of $\Mm(\sSO_0(2,3))$.

\begin{Remark} \label{only Consider Positive Toledo Remark}
    The map $\xymatrix{(L,(\Ww,Q_W),\beta,\gamma)\ar@{<->}[r]&(L^{-1},(\Ww,Q_W),\gamma,\beta)}$ defines an isomorphism $\Mm^{\tau,sw_2}(\sSO_0(2,3))\cong\Mm^{-\tau,sw_2}(\sSO_0(2,3)).$ Thus, it suffices to restrict our analysis to $0\leq\tau\leq 2g-2.$
\end{Remark}

The $\sSO_0(2,3)$-Higgs bundles with $|\tau|=2g-2$ are called \emph{maximal $\sSO_0(2,3)$-Higgs bundles}, these Higgs bundles will be the focus of the remainder of the paper. By Remark \ref{only Consider Positive Toledo Remark}, we can restrict our attention to $\tau = 2g-2$. We will use the notation 
\[\Mm^{\mathrm{max}}(\sSO_0(2,3)) = \Mm^{2g-2}(\sSO_0(2,3))~,\]
\[\Mm^{\mathrm{max}, sw_2}(\sSO_0(2,3)) = \Mm^{2g-2, sw_2}(\sSO_0(2,3))~. \]

\subsection{Maximal $\sSO_0(2,3)$-Higgs bundles}

In this subsection, we will describe the Higgs bundles in $\Mm^{\mathrm{max}}(\sSO_0(2,3))$.

Maximal $\sSO_0(2,3)$-Higgs bundles satisfy the following important property.

\begin{Proposition}\label{Prop gamma not zero}
    Let $(L,(\Ww,Q_W),\beta,\gamma)$ be a maximal $\sSO_0(2,3)$-Higgs bundle, then the map $\gamma:L\to \Ww\otimes K$ is nowhere vanishing, nowhere isotropic and $(L^{-1}K)^2=\Oo$.  
\end{Proposition}

\begin{proof}
By Lemma \ref{LemmaGamma}, $\gamma^T\circ\gamma$ is a non-zero section of $(L^{-1}K)^2.$ For $\deg(L)=2g-2$, this implies $(L^{-1}K)^2=\Oo$, and thus $\gamma$ is nowhere vanishing and nowhere isotropic. 
\end{proof}

Since $(L K^{-1})^2=\Oo$,  $L K^{-1}$ is a holomorphic $\sO(1,\C)$-bundle. Hence, for maximal $\sSO_0(2,3)$-Higgs bundles, the first Stiefel-Whitney class of $LK^{-1}$ gives an additional topological invariant: 
\[sw_1 = sw_1(L K^{-1}) \in H^1(\Sigma, \Z_2) \simeq \Z_2^{2g}~.\]

\begin{Proposition}\label{splittingofWProp}
If $(L,(\Ww,Q_W),\beta,\gamma)$ is a maximal $\sSO_0(2,3)$-Higgs bundle, then $\gamma(L^{-1}K)\subset\Ww$ is an orthogonal subbundle. 
    If $\Ff = \gamma(L K^{-1})^\perp$, then $\Ww$ splits holomorphically as $\Ww=L K^{-1}\oplus \Ff$ and  $\det(\Ff)=L^{-1}K=L K^{-1}$. 
\end{Proposition}
\begin{proof}
By Proposition \ref{Prop gamma not zero}, $\gamma(L K^{-1})$ is never isotropic, and hence the restriction of $Q_W$ there is non-degenerate. Thus $\Ww=\gamma(L K^{-1})\oplus\gamma(L K^{-1})^\perp$, and since $\det(L K^{-1}\oplus \Ff)=\Oo$ we have $\det(\Ff)=L^{-1}K.$
\end{proof}

The $\sO(2,\C)$-bundle $(\Ff,Q_F)$ determines the bundles $L$ and $(\Ww,Q_W)$.

\begin{Proposition}\label{Prop invariants}
The two topological invariants $sw_1, sw_2$ of the maximal $\sSO_0(2,3)$-Higgs bundle agree with the topological invariants of $(\Ff,Q_F)$:
$$sw_1(\Ff,Q_F) = sw_1(L K^{-1}) = sw_1, \hspace{0.5cm} sw_2(\Ff,Q_F) = sw_2(\Ww,Q_W) = sw_2.$$
\end{Proposition}
\begin{proof}
This follows from the Whitney sum formula \eqref{WhitneySum}. Since $(\Ww,Q_W)$ is an $\sSO(3,\C)$-bundle, we have $sw_1(\Ww,Q_W) = 0$, and hence \[sw_1(\Ff,Q_F) =  sw_1(L K^{-1})~.\] 
Since $sw_1(\Ff,Q_F) \wedge sw_1(L K^{-1}) = 0$, \eqref{WhitneySum} implies $sw_2(\Ww,Q_W) = sw_2(\Ff,Q_F)$.
\end{proof}

Let $\Mm_{sw_1}^{\mathrm{max}}(\sSO_0(2,3))$ be the subset of $\Mm^{\mathrm{max}}(\sSO_0(2,3))$ containing Higgs bundles with $sw_1(L K^{-1})=sw_1$, and $\Mm_{sw_1}^{\mathrm{max},sw_2}(\sSO_0(2,3))$ be the subset of $\Mm_{sw_1}^{\mathrm{max}}(\sSO_0(2,3))$ containing the Higgs bundles with $sw_2(W,Q_W)=sw_2$.

The orthogonal bundle $(\Ff,Q_F)$ defines a map to the space of $\sO(2,\C)$-bundles:
\[\xymatrix{\Mm^{\mathrm{max}}(\sSO_0(2,3)) \ar[r]& \Bb(\sO(2,\C))}~.\]
By Proposition \ref{Prop invariants}, this map sends $\Mm_{sw_1}^{\mathrm{max},sw_2}(\sSO_0(2,3))$ to $\Bb_{sw_1}^{sw_2}(\sO(2,\C))$.

For a poly-stable maximal $\sSO_0(2,3)$-Higgs bundle $(L,(\Ww,Q_W),\beta,\gamma)$, we can write the maps $\beta$ and $\gamma$ in terms of the decomposition $\Ww = L K^{-1} \oplus \Ff$. Since the holomorphic splitting of $\Ww$ was determined by the image of $\gamma,$ we can take $\gamma = \smtrx{1\\0}$. The map $\beta$ will be written as $\beta = \smtrx{q_2\\ \delta}$, where $q_2:L^{-1}\to L K^{-1} \otimes K$, is a quadratic differential and $\delta\in H^0(\Ff\otimes LK)=H^0(\Ff\otimes \det(\Ff)K^2)$. We have thus proven:

\begin{Proposition}
    A poly-stable maximal $\sSO_0(2,3)$-Higgs bundle is determined by the triple $((\Ff,Q_F),q_2,\delta)$, where $(\Ff,Q_F)$ is a holomorphic orthogonal bundle, $q_2 \in H^0(\Sigma, K^2)$ is a quadratic differential and $\delta\in H^0(\Ff \otimes \det(\Ff)K^2)$.
\end{Proposition}

\begin{Remark}\label{SL5 Higgs bundle for sw_1}
    The $\sSL(5,\C)$-Higgs bundle associated to a maximal $\sSO_0(2,3)$-Higgs bundle determined by a triple $((\Ff,Q_F),q_2,\delta)$ is

\[(\Ee, \Phi) = \left( \det(\Ff)K \oplus \det(\Ff)\oplus \Ff\oplus\det(\Ff)K^{-1}, \mtrx{0&q_2&\delta^T&0\\1&0&0&q_2\\0&0&0&\delta\\0&1&0&0} \right)~. \]
We find it helpful to think of such an object schematically as
    \begin{equation}\label{EQ schematic sw1 not 0}
        \xymatrix@R=0em{\det(\Ff)K\ar[r]_{\ 1}&\det(\Ff)\ar[r]_{1\ \ \ }\ar@/_1pc/[l]_{q_2}&\det(\Ff)K^{-1}\ar@/_1pc/[l]_{q_2}\ar[ddl]^{\delta}\\&\oplus&\\&\Ff\ar[uul]^{\delta^T}&}~~.
    \end{equation}
\end{Remark}

We first need to understand when two triples $((\Ff,Q_F),q_2,\delta)$, $((\Ff',Q_F),q_2',\delta')$ give rise to isomorphic Higgs bundles. In this notation, $\Ff$ and $\Ff'$ denote holomorphic structures $\bar\p_F$ and $\bar\p_{F}'$ on an underlying smooth orthogonal bundle $(F,Q_F)$ such that $Q_F$ is holomorphic with respect to both $\bar\p_F$ and $\bar\p_{F}'$. 
\begin{Proposition}\label{Prop gauge equivalent triples}
Let $\Ff$ and $\Ff'$ be two holomorphic structures on a smooth rank $2$ orthogonal bundle $(F,Q)$. Two triples $(\Ff,q_2,\delta)$, $(\Ff',q_2',\delta')$ give rise to $\sSO_0(2,3)$-Higgs bundles which are isomorphic if and only if $q_2=q_2'$ and there is a smooth gauge transformation $g\in\Gg_{\sO(2,\C)}(F,Q)$ such that $g\cdot \Ff\cdot g^{-1}=\Ff'$ and $\delta'=\det(g)\cdot g\cdot\delta.$
\end{Proposition}
\begin{proof}
Let $I$ be the smooth bundle underlying $\Ii=\det(\Ff),$ and consider the two smooth gauge transformations 
\[g_1\in\Gg_{\sSO(2,\C)}(IK\oplus IK^{-1})\ \ \ \text{and}\ \ \ g_2\in \Gg_{\sSO(3,\C)}\left(\Lambda^2 F\oplus  F,\mtrx{1&\\&Q_F}\right)~.\]
If the $\sSO_0(2,3)$-Higgs bundles associated to $(\Ff,q_2,\delta)$ and $(\Ff',q_2',\delta')$ are isomorphic,
    \[g_2\cdot\mtrx{1&q_2\\0&\delta}g_1^{-1}=\mtrx{1&q_2'\\0&\delta'}~.\] 
    Write $g_1=\smtrx{\lambda&\\&\lambda^{-1}}$ and $g_2=\smtrx{a&b\\c&d}:\Lambda^2 F\oplus F\to \Lambda^2 F\oplus  F,$ and note that $g_2^T\smtrx{1&\\&Q_F}g_2=\smtrx{1&\\&Q_F}$.
    With the this decomposition we compute
    \[\mtrx{a&b\\c&d}\mtrx{1&q_2\\0&\delta}\mtrx{\lambda^{-1}&\\&\lambda}=\mtrx{\lambda^{-1} a&\lambda  aq_2+b\lambda\\c\lambda^{-1}&c\lambda q_2+d\lambda\delta}~.\]
    Thus, we have $c=0$ and $\lambda=a$. The orthogonality of $g_2$ implies $b=0$, $d^TQ_Fd=Q_F$ and $a=\det(d)=\pm1.$ Thus, we have $d\in\Gg_{\sO(2,\C)}(F,Q_F)$, $\lambda=\det(d)$ and
    \[g_2\cdot\mtrx{1&q_2\\0&\delta}\cdot g_1^{-1}=\mtrx{\lambda&\\&d}\mtrx{1&q_2\\0&\delta}\mtrx{\lambda^{-1}&\\&\lambda}=\mtrx{1&q_2\\0&\lambda d\delta}~.\] 
\end{proof}

\begin{Remark}
    When a triple $(\Ff,q_2,\delta)$ defines a poly-stable Higgs bundle, we will call it a \emph{poly-stable triple}. The moduli space of maximal $\sSO_0(2,3)$-Higgs bundles on $\Sigma$ can be described as
\[\Mm^{\mathrm{max}}(\sSO_0(2,3)) = \{ ((\Ff,Q_F),q_2,\delta) \ |\ \text{ poly-stable triples }\} / \sim~,\]
where two triples are equivalent if and only if the associated $\sSO_0(2,3)$-Higgs bundles are isomorphic.
The space $\Mm^{\mathrm{max}}(\sSO_0(2,3))$ can be further subdivided in the pieces $\Mm^{\mathrm{max},sw_2}_{sw_1}(\sSO_0(2,3))$ according to the Stiefel-Whitney classes of $(\Ff,Q_F)$. 
\end{Remark}

\subsection{The case $sw_1(\Ff,Q_F)=0$} \label{HB_vanishing_sw_1}
To determine when a triple $(\Ff,q_2,\delta)$ is poly-stable we start with the case $sw_1=0.$ For this case, $(\Ff,Q_F)$ reduces to an $\sSO(2,\C)$-bundle, hence $\det(\Ff)=L^{-1}K=\Oo$. 
As in Section \ref{orthogonal_bundles_nonvanishing}, there is a holomorphic line bundle $M\in\Pic(\Sigma)$ so that 
\[(\Ff,Q_F)=\left(M\oplus M^{-1},\mtrx{0&1\\1&0}\right).\]
The splitting of $\Ff$ gives a decomposition of the map $\delta$ 
\[\delta:=\mtrx{\nu\\\mu}:K^{-1}\to (M\oplus M^{-1})\otimes K~,\]
where $\nu\in H^0(\Sigma, M K^2)$ and $\mu\in H^0(\Sigma, M^{-1}K^2).$ 

A poly-stable maximal $\sSO_0(2,3)$-Higgs bundle with vanishing $sw_1$ is then determined by the tuple $(M,q_2,\mu,\nu)$, where $M\in\Pic(\Sigma)$ is a holomorphic line bundle, $q_2 \in H^0(\Sigma, K^2)$, $\nu\in H^0(\Sigma, M K^2)$ and $\mu\in H^0(\Sigma,  M^{-1}K^2).$
The $\sSO(5,\C)$-Higgs bundle $(\Ee, Q, \Phi)$ associated to a maximal $\sSO_0(2,3)$-Higgs bundle determined by a tuple $(M,q_2,\mu,\nu)$ is
\begin{equation} \label{SL5_vanishing}
 \left( M \oplus K \oplus \Oo \oplus K^{-1} \oplus M^{-1}, \mtrx{&&&&-1\\&&&1&\\&&-1&&\\&1&&&\\-1&&&&} ,
\mtrx{0&0&0&\nu&0\\\mu&0&q_2&0&\nu\\0&1&0&q_2&0\\0&0&1&0&0\\0&0&0&\mu&0} \right)~.
\end{equation}

\begin{Proposition}\label{Prop sw1=0 stability}
The $\sSO(5,\C)$-Higgs bundle $(\Ee,Q,\Phi)$ associated to $(M,\mu,\nu,q_2)$ is stable if and only if one of the following holds
\begin{enumerate}
\item $0<\deg(M)\leq 4g-4$ and $\mu \neq 0$,
\item $4-4g\leq\deg(M)< 0$ and $\nu \neq 0$,
\item $\deg(M) = 0$, $\mu\neq 0$ and $ \nu\neq0$~.
\end{enumerate}
The $\sSO_0(2,3)$-Higgs bundle determined by $(M,\mu,\nu,q_2)$ is poly-stable if and only if  the associated $\sSO(5,\C)$-Higgs bundle is stable or $\deg(M)=0$, $\mu=0$ and $\nu=0.$
\end{Proposition}
\begin{proof}
The $\sSO(5,\C)$-Higgs bundle $(\Ee,Q,\Phi)$ associated to a tuple $(M,\mu,\nu,q_2)$ is given by \eqref{SL5_vanishing}. 
Recall from Definition \ref{DEF SOn stability} that  an $\sSO(5,\C)$-Higgs bundle is stable if and only if there are no nonnegative degree isotropic subbundles which are left invariant by the Higgs field. 
Suppose $\Vv\subset \Ee$ is an invariant isotropic subbundle with nonnegative degree. Denote the inclusion map by 
\[\smtrx{a\\b\\c\\d\\e}:\Vv\to M\oplus K\oplus \Oo\oplus K^{-1}\oplus M^{-1}~.\]
Note that since $\Vv$ is isotropic we have $-ae+bd-c^2=0$ and invariance is given by 
\[\Phi\smtrx{a\\b\\c\\d\\e}=\smtrx{\nu d\\\mu a+\nu e\\b+q_2c\\c\\\mu d}~.\]

First suppose $\rk(\Vv)=1.$ 
Since $\deg(\Vv)\geq 0$ we have $d=0$, and hence, by invariance, we must have $b=0$ and $c=0.$ Now, since $\Vv$ is isotropic, either $a=0$ or $e=0.$ 
Therefore, if $\deg(M)>0$ the $\sSO(5,\C)$-Higgs bundle is stable only if $\mu\neq0\in H^0(M^{-1}K^{2}).$ 
In particular, this gives a bound $0\leq \deg(M)\leq 4g-4.$ Similarly, if $\deg(M)<0$ the $\sSO(5,\C)$-Higgs bundle is stable if and only if $\nu\neq0$ and $4-4g\leq \deg(M)\leq 0.$ 
Finally, if $\deg(M)=0$ and either $\mu=0$ or $\nu=0,$ then the Higgs bundle is not stable since $M$ or $M^{-1}$ is an invariant isotropic line bundle of nonnegative degree. However, if $\mu\neq0$ and $\nu\neq0$, $(\Ee,Q,\Phi)$ is stable. 

Now suppose $\rk(\Vv)=2$ and $\Ee$ has no invariant positive degree invariant line subbundle.  
If $\Vv$ has a line subbundle $L$ of positive degree, then the restriction of the above map has $d=c=0$ and $e=0$ if     $\deg(M)\geq 0$ and $a=0$ if $\deg(M)\leq0$. 
Since $L$ is assumed to not be invariant we have $b\neq0$. But $b\neq0$ contradicts the fact that $\Phi(L)\subset\Vv$ is isotropic. 
Finally, suppose $\Vv$ has no positive degree line subbundles. In this case, $\Vv$ is a semi-stable vector bundle. Hence, $\Vv^*K^{-1}$ is a semi-stable vector bundle with nonpositive degree. This implies $H^0(\Vv^*\otimes K^{-1})=0,$ and thus $d=0.$ 
By invariance, we have $c=0$ and thus $b=0$. 
Since $\Vv$ is isotropic, $a=0$ or $e=0$. Thus, $(\Ee,Q,\Phi)$ has no non-negative degree rank two invariant isotropic subbundles.

To complete the proof, note that if $(\Ee,Q,\Phi)$ is a stable $\sSO(5,\C)$-Higgs bundle, then $(M,\mu,\nu,q_2)$ defines a poly-stable $\sSO_0(2,3)$-Higgs bundle. 
If $\deg(M)>0$ and $\mu=0$ or $\deg(M)<0$ and $\nu=0,$ then the associated $\sSL(5,\C)$-Higgs bundle is not poly-stable. Finally, if $\deg(M)=0$ and $\mu=0$ and $\nu\neq0$, then the associated $\sSL(5,\C)$-Higgs bundle is not poly-stable since $M$ defines a degree zero invariant subbundle with no invariant complement. Similarly, if $\nu=0$, then $M^{-1}$ is a degree zero invariant subbundle with no invariant complement. 
Thus, we conclude that the $\sSO_0(2,3)$-Higgs bundle determined by $(M,\mu,\nu,q_2)$ is poly-stable if and only if the associated $\sSO(5,\C)$-Higgs bundle is stable or $\deg(M)=0$, $\mu=0$ and $\nu=0.$
\end{proof}

By Proposition \ref{Prop gauge equivalent triples}, the $\sSO_0(2,3)$-Higgs bundles defined by $(M,\mu,\nu,q_2)$ and $(M',\mu',\nu',q_2)$ are isomorphic if and only if there is a smooth gauge transformation $g\in\Gg_{\sO(2,\C)}(M\oplus M^{-1},\smtrx{0&1\\1&0})$ such that $g\cdot M=M\in\Pic(\Sigma)$ or $g\cdot M=M^{-1}\in\Pic(\Sigma)$ and $g\cdot\smtrx{\nu\\\mu}=\smtrx{\nu'\\\mu'}.$ 
Thus, the number 
\[d = |\deg(M)|\in\Z \cap [0,4g-4]\] gives a new invariant to the maximal poly-stable $\sSO_0(2,3)$-Higgs bundles with $sw_1=0$. Let 
\[\Mm^{\mathrm{max}}_{0,d}(\sSO_0(2,3))\subset\Mm^{\mathrm{max}}_0(\sSO_0(2,3))\] denote the subspace of Higgs bundles determined by tuples $(M,\mu,\nu,q_2)$ such that $|\deg(M)|=d$.
This new invariant, only depends on $(\Ff,Q_F)\in\Bb(\sO(2,\C))$ and refines the second Stiefel-Whitney class: 
\[sw_2=d\ \ (\text{mod } \ 2)~.\]
We will see that all of these subspaces define connected components of $\Mm(\sSO_0(2,3))$. 
The orbifold points of $\Mm^\mathrm{max}_{0,d}(\sSO_0(2,3))$ are determined as follows.
\begin{Proposition}
    \label{Prop: smooth/orbifold in Md}
For $d>0$, the space $\Mm_{0,d}^{\max}(\sSO_0(2,3))$ is smooth. When $d=0$ the isomorphism class of the poly-stable $\sSO_0(2,3)$-Higgs bundle associated to a tuple $(M,\mu,\nu,q_2)$ is a 
\begin{itemize}
\item non-orbifold singularity if and only if $\mu=\nu=0$,
\item  $\Z_2$-orbifold singularity if and only $M=M^{-1}$, $\mu\neq0$ and $\mu=\lambda\nu$ for some $\lambda\in\C^*$,
\item  smooth point otherwise.
\end{itemize}
\end{Proposition}
\begin{proof}
By Proposition \ref{Prop sw1=0 stability}, the $\sSO(5,\C)$-Higgs bundle given by a tuple $(M,\mu,\nu,q_2)$ is stable if and only if $d\neq0$ or $d=0$ and $\mu\neq0$ and $\nu\neq0.$ 
Thus, by Proposition \ref{Proposition: Orbifold points} the isomorphism class of such tuples define smooth and orbifold points of $\Mm_{0,d}^{\max}(\sSO_0(2,3))$. To determine the type of orbifold point we need to compute the automorphism group $\Aut(\Vv,\Ww,\eta)$ of the associated $\sSO_0(2,3)$-Higgs bundle.

By Proposition \ref{Prop gauge equivalent triples}, we need only consider how the holomorphic automorphism group $\Hh_{\sO(2,\C)}(M\oplus M^{-1})$ acts on the sections $\mu$ and $\nu.$ Recall from \eqref{EQ HO2 sw1=0} that, if $M\neq M^{-1}$, then the holomorphic gauge transformations  are given by 
\[g=\mtrx{\lambda&\\&\lambda^{-1}}: M\oplus M^{-1}\to M\oplus M^{-1}\]
for $\lambda\in\C^*$. We have $g\cdot\mtrx{\nu\\\mu}=\mtrx{\lambda&\\&\lambda^{-1}}\mtrx{\nu\\\mu}=\mtrx{\lambda\nu\\\lambda^{-1}\mu}~.$ Thus, by Proposition \ref{Prop gauge equivalent triples}, for $M\neq M^{-1}$, the automorphism group of the $\sSO_0(2,3)$-Higgs bundle associated to a tuple $(M,\mu,\nu,q_2)$ is trivial for $\mu\neq0$ or $\nu\neq0.$ In particular, for $d>0,$ the space $\Mm_{0,d}(\sSO_0(2,3))$ is smooth.

For $d=0$ and $M=M^{-1}\in\Pic^0(\Sigma)$, recall from \eqref{EQ HO2 sw1=0 M=M^-1} that $\Hh_{\sO(2,\C)}(M\oplus M^{-1})\cong\sO(2,\C)$ and we need to also consider holomorphic gauge transformation of the form 
\[g=\mtrx{&\lambda\\\lambda^{-1}&}: M\oplus M^{-1}\to M\oplus M^{-1}\]
for $\lambda\in\C^*.$ We have $g\cdot\mtrx{\nu\\\mu}=-\mtrx{&\lambda\\\lambda^{-1}&}\mtrx{\nu\\\mu}=\mtrx{-\lambda\mu\\-\lambda^{-1}\nu}.$ Thus, by Proposition \ref{Prop gauge equivalent triples}, if $\mu\neq0$ and $\nu\neq0$, the Higgs bundles associated to $(M,\mu,\nu,q_2)$ and $(M,\mu',\nu',q_2)$ are isomorphic if and only if $\nu=-\lambda\mu.$ In this case the automorphism group of the associated Higgs bundle is $\Z_2.$
On the other hand, if $\mu=\nu=0$, then the automorphism group of the associated Higgs bundle is not discrete, and thus a tuple $(M,0,0,q_2)$ defines a non-orbifold singularity.
\end{proof}

\begin{Remark}
For a geometric interpretation of the singular points of the subspace $\Mm_{0,0}(\sSO_0(2,3))$, see Proposition \ref{Prop Higgs reductions M0d}.
\end{Remark}

 \subsection{The case $sw_1(\Ff,Q_F)\neq0$}   \label{HB_non_vanishing_sw_1}
When $sw_1\neq0,$ the associated $\sSO(5,\C)$-Higgs bundle is always stable.
\begin{Proposition}\label{Prop sw1 not 0 stability}
    The $\sSO(5,\C)$-Higgs bundle associated to a triple $((\Ff,Q_F),q_2,\delta)$ with $sw_1(\Ff,Q_F)\neq0$ is stable.
\end{Proposition}
\begin{proof}
    The proof is very similar to Proposition \ref{Prop sw1=0 stability}. Recall that the $\sSO(5,\C)$-Higgs bundle associated to a triple $((\Ff,Q_F),q_2,\delta)$ is given by
    \[(\Ee,Q,\Phi)=\left(K\Ii\oplus \Ii\oplus \Ff\oplus K^{-1}\Ii,\ \mtrx{0&0&0&1\\0&-1&0&0\\0&0&-Q_F&0\\1&0&0&0},\mtrx{0&q_2&\delta^T&0\\1&0&0&q_2\\0&0&0&\delta\\0&1&0&0} \right)~.\] 
    where $\Ii=\det(\Ff).$ We will show that $\Ee$ has no $\Phi$-invariant isotropic subbundles with non-negative degree. 

Suppose $L$ is an isotropic invariant line subbundle with non-negative degree. As in the proof of Proposition \ref{Prop sw1=0 stability}, $L$ must be an isotropic line subbundle of $(\Ff,Q_F).$ However, by Proposition \ref{Prop: no isotropic subbundles}, since $sw_1\neq0,$ $\Ff$ has no isotropic line subbundles. 
Again, as in the proof of Proposition \ref{Prop sw1=0 stability}, if $\Vv\subset\Ee$ is an isotropic rank two bundle with non-negative degree, then $\Vv= \Ff.$ But $\Ff$ is not isotropic, and we conclude that the $\sSO(5,\C)$-Higgs bundle $(\Ee,Q,\Phi)$ is stable. 
\end{proof}

As in Section \ref{Orthogonal_bundles}, for $sw_1\in H^1(\Sigma,\Z_2)\setminus \{0\}$ let  $\pi:\Sigma_{sw_1} \to \Sigma$ be the associated double cover and denote the covering involution by $\iota$. Let $K_{sw_1}$ be the canonical bundle of $\Sigma_{sw_1}$; since the covering is unramified, we have $\pi^*K=K_{sw_1}.$ Recall that 
\[\Prym(\Sigma_{sw_1})=\{M\in\Pic^0(\Sigma_{sw_1})~|~\iota^*M\cong M^{-1}\}~.\]

\begin{Proposition} 
\label{Prop: pullbackHiggsBundles} A poly-stable maximal $\sSO_0(2,3)$-Higgs bundle with non-vanishing $sw_1$ is determined by the tuple $(M,f,\mu,q_2)$, where
\[\xymatrix@R=0em@C=1em{M \in \Prym(\Sigma_{sw_1})~,&\mu\in H^0(\Sigma, M^{-1}K_{sw_1}^2)~,& q_2 \in H^0(\Sigma, K^2)~,\\\text{and}& f:M \to \iota^*M^{-1} \text{ is an isomorphism}~.}\]
Moreover, the covering map $\pi:\Sigma_{sw_1}\to\Sigma$ induces a pullback map
\begin{equation}\label{pullbackHiggsBundles}
    \pi^*: \xymatrix{\Mm^{\mathrm{max}}_{sw_1}(\Sigma,\sSO_0(2,3)) \ar[r]& \Mm^{\mathrm{max}}_{0,0}(\Sigma_{sw_1},\sSO_0(2,3))}~.
\end{equation}
\end{Proposition}
\begin{proof}
Let $(\Ff,Q_F)$ be an orthogonal rank two bundle with nonzero first Stiefel-Whitney class $sw_1,$ and denote $\Ii=\det(\Ff).$ Then $\pi^*\Ii=\Oo$ and there is $M\in\Prym(\Sigma_{sw_1})$ such that 
\[(\pi^*\Ff,\pi^*Q_F)\cong \left(M\oplus M^{-1},\mtrx{0&1\\1&0}\right)~.\]
Note that $\pi^*\Ff$ is $\iota^*$-invariant: $\iota^*\pi^*\Ff \cong \pi^*\Ff$. Moreover, the natural projection $\pi^*\Ff \ra \Ff$ gives a choice of an isomorphism $\pi^*\Ff \ra \iota^*\pi^*\Ff$. When this isomorphism is restricted to $M$, we get an isomorphism $f:M \ra \iota^*M^{-1}$, and when it is restricted to $M^{-1}$, the isomorphism is $\iota^* f^{-1}: M^{-1} \ra \iota^*M$.

Recall that the maximal $\sSO_0(2,3)$-Higgs bundle associated to $((\Ff,Q_F),q_2,\delta)$ is given by 
\[(\Vv,Q_V,\Ww,Q_W,\eta)=\left(K\Ii\oplus K^{-1}\Ii,\mtrx{0&1\\1&0},\Ii\oplus \Ff,\mtrx{1&0\\0&Q_F},\mtrx{1&q_2\\0&\delta}\right)~.\]
We have 
        \[\pi^*(\Vv,Q_V,\Ww,Q_W)=\left(K_{sw_1}\oplus K_{sw_1}^{-1},\mtrx{0&1\\1&0},M\oplus \Oo\oplus M^{-1},\mtrx{0&0&1\\0&1&0\\1&0&0}\right)~.\]
Moreover, $\pi^*q_2\in H^0(\Sigma_{sw_1},K^2_{sw_1})$ and the decomposition of $\pi^*\Ff$ splits $\pi^*\delta$ as 
\[\pi^*\delta:=\mtrx{\nu\\\mu}:\xymatrix{K_{sw_1}^{-1}\ar[r]& (M\oplus M^{-1})\otimes K_{sw_1}}~,\]
where $\nu\in H^0(\Sigma_{sw_1}, M K_{sw_1}^2)$ and $\mu\in H^0(\Sigma_{sw_1}, M^{-1}K_{sw_1}^2).$ Also, since the pulled back objects are invariant under the covering involution and $\iota^*f:\iota^*M \to M^{-1},$ we have $\iota^*\nu\circ f  =\mu$. 
Thus, we have $\mu=0$ if and only if $\nu=0.$ 
By Proposition \ref{Prop sw1=0 stability}, the pulled back $\sSO_0(2,3)$-Higgs bundle is a maximal poly-stable $\sSO_0(2,3)$-Higgs bundle whose isomorphism class defines a point in $\Mm^{\mathrm{max}}_{0,0}(\Sigma_{sw_1},\sSO_0(2,3))$.
\end{proof}

\begin{Proposition}     \label{Prop: equivalence_nonvanishing}  
Two tuples $(M,f,\mu,q_2)$, $(M',f',\mu',q_2')$ give rise to isomorphic $\sSO_0(2,3)$-Higgs bundles if and only if one of the following holds:
\begin{enumerate}
\item $M'=M$, $q_2' = q_2$, $f'=\lambda^{-2} f$ and $\mu' = \lambda^{-1} \mu$ for $\lambda\in\C^*$, 
\item $M'= M^{-1}$, $q_2' = q_2$, $f'=\lambda^{-2} \iota^*f^{-1}$ and $\mu' = -\lambda^{-1} f^{-1}\circ \iota^*\mu$ for $\lambda\in \C^*$.
\end{enumerate}
\end{Proposition}

\begin{proof}
The two Higgs bundles on $\Sigma$ are isomorphic if and only if their pullbacks to $\Sigma_{sw_1}$ are isomorphic via a gauge transformation which is invariant under the covering involution. Thus, we can apply Proposition \ref{Prop: smooth/orbifold in Md}, and compute how the gauge transformations act on $\mu$ and $f$.
\end{proof}

By Proposition \ref{Prop sw1 not 0 stability}, all tuples $(M,f,\mu,q_2)$ from Proposition \ref{Prop: pullbackHiggsBundles} define poly-stable $\sSO_0(2,3)$-Higgs bundles on $\Sigma$ whose associated $\sSO(5,\C)$-Higgs bundle is stable. Thus, all points of $\Mm_{sw_1}^{sw_2,\mathrm{max}}(\sSO_0(2,3))$ are smooth or orbifold points.
Using Proposition \ref{Proposition: Orbifold points}, we have the following.
\begin{Proposition}
\label{Prop: smooth/orbifold in M sw1not0} 
The singularities of  $\Mm^{sw_2,\mathrm{max}}_{sw_1}(\sSO_0(2,3))$ are all orbifold singularities. Moreover,  the poly-stable $\sSO_0(2,3)$-Higgs bundle associated to a tuple $(M,f,\mu,q_2)$ defines a
    \begin{itemize}
        \item $\Z_2\oplus\Z_2$ orbifold point if $M=M^{-1}$ and $\mu=0,$ 
        \item $\Z_2$ orbifold point if $M=M^{-1}$, $\mu\neq0,$ $f=\lambda^{-2}\iota^* f^{-1}$ and $\mu = -\lambda \iota^*(f\mu)$, for some $\lambda \in \C^*$. 
        \item $\Z_2$ orbifold point if $M\neq M^{-1}$ and $\mu=0,$
        \item  smooth point otherwise.
    \end{itemize} 
\end{Proposition}
\begin{proof}
We need to check which of the gauge transformations described in Proposition \ref{Prop: smooth/orbifold in Md} act trivially on the Higgs bundle described by a tuple $(M,f,\mu,q_2)$. When $M= M^{-1}$, the first two points follow from item $(2)$ of Proposition \ref{Prop: equivalence_nonvanishing}. When $M\neq M^{-1}$ the last two points follow from item $(1)$ of Proposition \ref{Prop: equivalence_nonvanishing}.

\end{proof}

\begin{Remark}
For a geometric interpretation of the singular points of the subspace $\Mm^{sw_2,\max}_{sw_1}(\sSO_0(2,3))$, see Proposition \ref{Prop Higgs reduction sw1not0}.
\end{Remark}

\subsection{Parameterizing the components $\Mm^{\textmd{max}}_{0,d}(\sSO_0(2,3))$}  \label{parameterizing_vanishing}
We start by parameterizing the components $\Mm^{\textmd{max}}_{0,d}(\sSO_0(2,3))$ for $d>0.$

\begin{Theorem}\label{THM d>0}
    For $0<d\leq 4g-4$, the space $\Mm^{\mathrm{max}}_{0,d}(\sSO_0(2,3))$ is diffeomorphic to the product $\Ff_d \times H^0(\Sigma, K^2)$, where $H^0(\Sigma, K^2)$ is the space of holomorphic quadratic differentials on $\Sigma$ and $\Ff_d$ is a rank $3g-3+d$ holomorphic vector bundle over the $(4g-4-d)^{th}$-symmetric product $\Sym^{4g-4-d}(\Sigma)$ of $\Sigma$. 
\end{Theorem}
\begin{proof}
    By Proposition \ref{Prop: smooth/orbifold in Md}, when $0<d\leq 4g-4,$  the space $\Mm_{0,d}^{\mathrm{max}}(\sSO_0(2,3))$ is smooth.
    Define the space 
    \[\widehat\Ff_d=\{(M,\mu,\nu)~|~ M\in\Pic^d(\Sigma),\ \mu\in H^0(M^{-1}K^2)\setminus 0,\ \nu\in H^0(M K^2)\}~.\]
    In Section \ref{HB_vanishing_sw_1} we described a surjective map 
    \[\widehat\Psi:\xymatrix{\widehat\Ff_d\times H^0(K^2)\ar[r]&\Mm_{0,d}^{\mathrm{max}}(\sSO_0(2,3))}~.\] 
    There is an action of $\C^*$ on $\widehat\Ff_d$ given by  $\lambda \cdot (M,\mu,\nu) =  (M,\lambda\mu,\lambda^{-1}\nu)$. 
    Moreover, by the proof of Proposition \ref{Prop: smooth/orbifold in Md},
    $\widehat\Psi(M,\mu,\nu,q_2)=\widehat\Psi(M',\mu',\nu',q_2')$
    if and only if $(M,\mu,\nu,q_2)$ and $(M',\mu',\nu',q_2')$ are in the same $\C^*$ orbit. Thus, if $\Ff_d=\widehat\Ff_d/\C^*$, then $\Ff_d\times H^0(\Sigma, K^2)$ is diffeomorphic to $\Mm_{0,d}^{\mathrm{max}}(\sSO_0(2,3)).$ 
    
 Given an $\C^*$-equivalence class $[(M,\mu,\nu)]$, the projective class of the nonzero section $\mu$ defines an effective divisor on $\Sigma$ of degree $-d+4g-4.$ This defines a projection $\pi:\Ff_d\to\Sym^{-d+4g-4}(\Sigma)$. 
  If $D\in\Sym^{-d+4g-4}(\Sigma)$ and $\Oo(D)$ is the holomorphic line bundle associated to $D$, then $\pi^{-1}(D)$ is identified (noncanonically) with the vector space $H^0(\Oo(D)^{-1}K^4)\cong\C^{d+3g-3}$.
\end{proof}
\begin{Corollary}
    For $d\neq0,$ the connected component $\Mm_{0,d}^\mathrm{max}(\sSO_0(2,3))$ is homotopically equivalent to the space $\Sym^{4g-4-d}(\Sigma).$ 
\end{Corollary}
The cohomology of the symmetric product of a surface was computed in \cite{SymmetricProductsofAlgebraicCurves}.
\begin{Remark}
    When $d=4g-4,$ the space $\Mm_{0,4g-4}^{\mathrm{max}}(\sSO_0(2,3))$ is the Hitchin component from \eqref{EQ Hitchin comp param} and the parametrization was given by Hitchin in \cite{liegroupsteichmuller}.
\end{Remark}

The component $\Mm^{\mathrm{max}}_{0,0}(\sSO_0(2,3))$ (for $d=0$) is the hardest to describe because of the presence of singularities. We  will describe it in Theorem \ref{thm:zero_component}, but we first introduce some notation and prove some preliminary lemmas.

Let $\mathcal{O}_{\C\P^{n-1}}(-1)$ denote the tautological holomorphic line bundle over $\C\P^{n-1}$. Let $T_n$ denote the rank $n$ holomorphic vector bundle over $\C\P^{n-1}$ obtained by taking the direct sum of $\mathcal{O}_{\C\P^{n-1}}(-1)$ with itself $n$ times: 
\[T_n = \mathcal{O}_{\C\P^{n-1}}(-1)\underbrace{\oplus \dots \oplus}_{n \text{-times}} \mathcal{O}_{\C\P^{n-1}}(-1)~.\]
Let $\mathcal{U}_n$ be the quotient of the total space of $T_n$, by the equivalence relation that collapses the zero section of $T_n$ to a point:
\[\mathcal{U}_n = T_n / \{\text{zero section}\}~.\]

\begin{Lemma}\label{Un contractible}
The topological space $\mathcal{U}_{n}$ is contractible.
\end{Lemma}
\begin{proof}
Since $T_n$ is a vector bundle, its total space can be retracted to the zero section. When the same retraction is applied to $\mathcal{U}_{n}$, it retracts the latter space to its singular point. Hence $\mathcal{U}_{n}$ is contractible.   
\end{proof}

\begin{Lemma}   \label{lemma:tautological_bundle}
Consider the action of $\C^*$ on $\C^n \times \C^n$ given by $\lambda \cdot (v,w) = (\lambda v, \lambda^{-1}w)$. If $\widehat\Uu_n$ is the $\C^*$-invariant subspace 
\[\widehat\Uu_n = (\C^n\setminus \{0\})\times(\C^n\setminus \{0\}) \cup \{(0,0)\} \subset \C^n \times \C^n~,\] then the quotient $\widehat{\mathcal{U}}_n / \C^*$ is homeomorphic to $\mathcal{U}_{n}$.
\end{Lemma}
\begin{proof}
Consider the map: 
\[\hat\phi:\xymatrix@R=.2em{(\C^n\setminus \{0\})\times(\C^n\setminus \{0\})\ar[r]&\C\P^{n-1} \times (\C^n\oplus\dots\oplus\C^n)\\(v,w) \ar@{|->}[r] & ([v],(w_1 v, w_2 v, \dots, w_n v) )}~.\]
The image of this map is exactly the vector bundle $T_n$ minus the zero section, and the map is $\C^*$-invariant. This map induces a homeomorphism 
$$\phi: (\C^n\setminus \{0\})\times(\C^n\setminus \{0\})/\C^* \to T_n \setminus \{\text{zero section}\}~. $$  
We can extend $\phi$ to a map 
$$\phi': \widehat{\mathcal{U}}_n/\C^* \to \mathcal{U}_{n}$$ 
by defining it as $\phi$ on $(\C^n\setminus \{0\})\times(\C^n\setminus \{0\})/\C^*$, and by mapping the point $(0,0)$ to the point of $\mathcal{U}_{n}$ corresponding to the zero section of $T_n$. 
To check that $\phi'$ is a homeomorphism, we just need to verify the following elementary fact: given a sequence $(v_m)$ in $\C^n$ and $(x_m)$ in $\C$, then $x_m v_m \to 0$ if and only if there exists a sequence $(\lambda_m)$ in $\C^*$ such that $\lambda_m^{-1} x_m \to 0$ and $\lambda_m\to 0$.  
\end{proof}

\begin{Theorem}   \label{thm:zero_component}
The component $\Mm_{0,0}^{\mathrm{max}}(\sSO_0(2,3))$ is homeomorphic to 
\[(\Aa/\Z_2) \times H^0(\Sigma, K^2)~,\] where $H^0(\Sigma, K^2)$ the space of holomorphic quadratic differentials on $\Sigma$, $\Aa$ is a holomorphic fiber bundle over $\Pic^0(\Sigma)$ with fiber $\mathcal{U}_{3g-3}$ and $\Z_2$ acts on $\Aa$ by pullback by inversion on $\Pic^0(\Sigma).$
In particular, $\Mm_{0,0}^{\mathrm{max}}(\sSO_0(2,3))$ is homotopically equivalent to the quotient $\Pic^0(\Sigma)/\Z_2$.
\end{Theorem}
\begin{proof}

Define the spaces
\[\widetilde\Aa=\{(M,\mu,\nu)| M\in\Pic^0(\Sigma),\ \mu\in H^0(\Sigma, M^{-1}K^2),\ \nu\in H^0(\Sigma, M K^2)\}~,\]
\[\widehat\Aa=\{(M,\mu,\nu)\in\widetilde\Aa\ |\ \mu=0\ \text{if\ and\ only\ if}\ \nu=0\}~.\]
In Section \ref{HB_vanishing_sw_1}, we constructed a surjective map from $\widehat\Aa\times H^0(\Sigma, K^2)$ to the space $\Mm_{0,0}^{\mathrm{max}}(\sSO_0(2,3)).$ 
By the proof of Proposition \ref{Prop: smooth/orbifold in Md}, $(M,\mu,\nu,q_2)$ and $(M',\mu'\nu',q_2')$ define the same point in $\Mm_{0,0}^{\max}(\sSO_0(2,3))$ if and only if, for $\lambda\in\C^*$
\[(M',\mu',\nu',q_2')=(M,\lambda\mu,\lambda^{-1}\nu,q_2)\ \ \ \ or\ \ \ \ (M',\mu',\nu',q_2')=(M^{-1},\lambda\nu,\lambda^{-1}\mu,q_2)~. \]

Let $\Aa$ be the quotient of $\widehat\Aa$ by the $\C^*$ action $\lambda \cdot (M,\mu,\nu) = (M,\lambda \mu,\lambda^{-1}\nu)$. 
We claim that the map $\Aa\to\Pic^0(\Sigma)$ defined by sending an equivalence class $[M,\mu,\nu]$ to $M$ is a holomorphic bundle over $\Pic^0(\Sigma)$ with fiber $\mathcal{U}_{3g-3}$. In particular, $\Aa$ is homotopically equivalent to $\Pic^0(\Sigma) \simeq (\mathbb{S}^1)^{2g}$.
Indeed, the fiber of this map over the point $M\in \Pic^0(\Sigma)$ is given by 
\[((H^0(\Sigma, M K^2)\setminus\{0\} \times H^0(\Sigma, M^{-1}K^2)\setminus\{0\})\cup\{(0,0)\})\slash \C^*~.\] We have $\dim H^0(\Sigma, M K^2)=\dim H^0(\Sigma, M^{-1} K^2) = 3g-3$, hence, by Lemma \ref{lemma:tautological_bundle}, the fiber is the space $\mathcal{U}_{3g-3}$.  

The action of $\Z_2$ on $\Pic^0(\Sigma)$ by inversion ($M \to M^{-1}$) lifts to an action on $\Aa$ by sending $(M,\mu,\nu)$ to $(M^{-1},\nu,\mu)$. 
We conclude that the component $\Mm_{0,0}^{\mathrm{max}}(\sSO_0(2,3))$ is homeomorphic to $\Aa/\Z_2 \times H^0(\Sigma, K^2)$. 
\end{proof}
\begin{Corollary}
    The component $\Mm_{0,0}^\mathrm{max}(\sSO_0(2,3))$ is homotopically equivalent to the space $\Bb_{0,0}(\sO(2,\C)).$ Its rational cohomology is given by Proposition \ref{Prop: cohomology of Torus mod inversion}.
\end{Corollary}

\subsection{Parameterizing the components $\Mm^{\textmd{max},sw_2}_{sw_1}(\sSO_0(2,3))$} \label{parameterizing_nonvanishing}

Fix a pair of cohomology classes $(sw_1,sw_2)\in H^1(\Sigma,\Z_2)\times H^2(\Sigma,\Z_2)$ with $sw_1\neq 0.$ We will use the notation of Section \ref{orthogonal_bundles_nonvanishing}. Let $\pi:\Sigma_{sw_1}\to\Sigma$ be the genus $g'=2g-1$ double cover associated to $sw_1$, $\iota:\Sigma_{sw_1}\to \Sigma_{sw_1}$ be the corresponding covering involution and $\Prym^{sw_2}(\Sigma_{sw_1})$ be a connected component of $\ker(Id \otimes \iota^*)$.

\begin{Proposition} \label{bundleoverPrym} 
There is a holomorphic vector bundle $\Ee\to \Prym^{sw_2}(\Sigma_{sw_1})$ of rank $6g-6$ such that for every $M \in \Prym^{sw_2}(\Sigma_{sw_1})$, the fiber $\Ee|_{\{M\}}$ is the space $H^0(\Sigma_{sw_1},M^{-1} K_{sw_1}^2)$.  
\end{Proposition}
\begin{proof}
Consider the Poincar\'e line bundle $\Pp \to \Pic^0(\Sigma_{sw_1})\times \Sigma_{sw_1}$. This is the universal bundle of the fine moduli space $\Pic^0(\Sigma_{sw_1})$; it has the property that for every $M\in \Pic^0(\Sigma_{sw_1})$, the restriction $\Pp|_{\{M\} \times \Sigma_{sw_1}}$ is a line bundle on  $\Sigma_{sw_1}$ isomorphic to $M$.
Let $\pi_{\Pic^0}$ and $\pi_{\Sigma_{sw_1}}$ be the projections from $\Pic^0(\Sigma_{sw_1})\times \Sigma_{sw_1}$ to the two respective factors. Now $\Pp \otimes \pi_{\Sigma_{sw_1}}^*K_{sw_1}^2$ is a line bundle over $\Pic^0(\Sigma_{sw_1})\times \Sigma_{sw_1}$ with the property that its restriction to every $M \in \Pic^0(\Sigma_{sw_1})$ is a line bundle over $\Sigma_{sw_1}$ isomorphic to $M K_{sw_1}^2$. 
The push forward $\Ee' = (\pi_{\Pic^0})_*(\Pp \otimes \pi_{\Sigma_{sw_1}}^*K_{sw_1}^2)$ is a vector bundle over $\Pic^0(\Sigma_{sw_1})$ 
whose fiber over every point $M \in \Pic^0(\Sigma_{sw_1})$ is the vector space $H^0(\Sigma_{sw_1},M K_{sw_1}^2)$. In particular, it has dimension $3g'-3= 6g-6$. The bundle $\Ee$ is the pull back of $\Ee'$ via the map 
\[\xymatrix@R=0em{\Prym^{sw_2}(\Sigma_{sw_1})\ar[r]&\Pic^0(\Sigma_{sw_1})~.\\M\ar@{|->}[r]& M^{-1}}\]
\end{proof}

\begin{Proposition} \label{bundleoverPrym2}  
There is a holomorphic line bundle $\Jj\to \Prym^{sw_2}(\Sigma_{sw_1})$ such that for every $M \in \Prym^{sw_2}(\Sigma_{sw_1})$, the fiber $\Jj|_{\{M\}}$ is the space $\mathrm{End}(M, \iota^* M^{-1})$.  
\end{Proposition}
\begin{proof}
Similar to the proof of the previous proposition.
\end{proof}

We will consider the direct sum $\Ee \oplus \Jj$ as a vector bundle of rank $6g-5$ over $\Prym^{sw_2}(\Sigma_{sw_1})$ whose total space parametrizes the tuples $(M,f,\mu)$ where $M\in \Prym^{sw_2}(\Sigma_{sw_1})$, $f\in \mathrm{End}(M, \iota^* M^{-1})$ and $\mu \in H^0(\Sigma_{sw_1},M^{-1} K_{sw_1}^2)$. We will denote by $\Hh$ the open subset:
\begin{equation}
    \label{eq H definition}\Hh=\{ (M,f,\mu) \in \Ee \oplus \Jj \mid f\neq 0 \} 
\end{equation}
parameterizing the tuples $(M,f,\mu)$ where $f$ is an isomorphism. 
We define an action of $\C^*$ on the total space of $\Hh$ via the following formula:
\begin{equation}
    \label{eq C* SO23}\lambda \cdot (M,f,\mu) = (M, \lambda^2 f, \lambda \mu)~.
\end{equation}

The quotient $\Hh'=\Hh/\C^*$ is the space parameterizing gauge equivalence classes of the triples $(M,f,\mu)$, with $f$ an isomorphism. This space is a bundle over $\Prym^{sw_2}(\Sigma_{sw_1})$ whose fiber over the point $M$ is isomorphic to
$$H^0(\Sigma_{sw_1},M^{-1} K_{sw_1}^2) / \pm 1. $$ 
The space $\Hh'$ is an orbifold which has one orbifold point in each fiber with orbifold group $\Z_2$. This point is defined by the class $[(M,f,0)]$. 
On the space $\Hh'$ we have a $\Z_2$-action given by 
\begin{equation}
    \label{eq Z2 action So23}\tau \cdot [(M,f,\mu)] = [(\iota^*M, \iota^*f, \iota^*\mu)]~.
\end{equation}

We can describe the quotient space by this action.

\begin{Proposition}   \label{prop:description-orbifold}
The quotient space $\Hh' / \Z_2$ is an orbifold where:
\begin{enumerate}
\item The image of the $2^{2g-2}$ points $[(M,f,0)]$ where $M=\iota^* M$ define orbifold points with orbifold group $\Z_2\oplus\Z_2$.
\item The image of the points $[(M,f,0)]$ with $M\neq\iota^* M$ form a (non-closed) submanifold of orbifold points with orbifold group $\Z_2$.
\item The image of the points $[(M,f,\mu)]$ with $M=\iota^* M$, $\mu=\iota^*\mu$ and $\mu\neq 0$ form a (non-closed) submanifold of orbifold points with orbifold group $\Z_2$. 
\item All the other points are smooth.
\end{enumerate}  
The image of the points $[(M,f,0)]$ form a closed subspace which is orbifold isomorphic to $\Prym^{sw_2}(\Sigma_{sw_1})/\Z_2$. 
Moreover, the quotient space $\Hh' / \Z_2$ is homotopically equivalent to $\Prym^{sw_2}(\Sigma_{sw_1})/\Z_2$.
\end{Proposition}
\begin{proof}
The action of the group $\Z_2$ is not free, so the quotient is an orbifold. To understand the orbifold points we just need to compute the stabilizer of every point. Since $\Hh'$ is a bundle whose fiber is contractible, it can be retracted to its zero section. The retraction can be made in a $\Z_2$ equivariant way, so this passes to the quotient. 
\end{proof}

The $\Z_2$-action on $\Hh'$ can be extended trivially to an action on $\Hh'\times H^0(\Sigma,K^2)$.

\begin{Proposition}  \label{prop:non-zero-sw1}
There is a $\Z_2$-invariant surjective map 
\[\widehat\Psi:\xymatrix{\Hh'\times H^0(\Sigma,K^2)\ar[r]& \Mm^{\mathrm{max},sw_2}_{sw_1}(\sSO_0(2,3))}~.\] 
This induces a bijective map on the quotient
\[\Psi: \left(\Hh' / \Z_2\right)\times H^0(\Sigma,K^2)\lra \Mm^{\mathrm{max},sw_2}_{sw_1}(\sSO_0(2,3)) ~.\]
\end{Proposition}
\begin{proof}
In Section \ref{HB_non_vanishing_sw_1}, we described a surjective map from $\Hh'\times H^0(K^2)$ to the space $\Mm_{sw_1}^{\mathrm{max},sw_2}(\sSO_0(2,3))$. By Proposition \ref{Prop: equivalence_nonvanishing} this map is $\Z_2$-invariant and injective on the quotient by $\Z_2$.
\end{proof}

\begin{Theorem}\label{THM HiggsParamsw1sw2orbifold}
Let $(sw_1,sw_2)\in H^1(\Sigma,\Z_2)\times H^2(\Sigma,\Z_2)$ be a pair of cohomology classes with $sw_1\neq 0$. Let $\Mm^{\mathrm{max},sw_2}_{sw_1}(\sSO_0(2,3))$ be the corresponding component of moduli space of maximal $\sSO_0(2,3)$-Higgs bundles from \eqref{maximalSO(2,3)decomp}, and let $\Hh'\to\Prym^{sw_2}(X_{sw_1},\Sigma)$ be the bundle defined above. 
There is an orbifold isomorphism between $\Mm^{\mathrm{max},sw_2}_{sw_1}(\sSO_0(2,3))$ and the space 
$$\left(\Hh' / \Z_2\right)\times H^0(\Sigma,K^2)~.$$
\end{Theorem}
\begin{proof} The isomorphism is given by the map described in Proposition \ref{prop:non-zero-sw1}. 
This map is an orbifold isomorphism by Proposition \ref{Prop: smooth/orbifold in M sw1not0} and Proposition \ref{prop:description-orbifold}.  
\end{proof}
\begin{Corollary}
    The component $\Mm_{sw_1}^{\mathrm{max},sw_2}(\sSO_0(2,3))$ is homotopically equivalent to the space $\Bb_{sw_1}^{sw_2}(\sO(2,\C)).$ Its rational cohomology is given by Proposition \ref{Prop: cohomology of Torus mod inversion}.
\end{Corollary}

\subsection{Zariski closures of maximal $\sP\sSp(4,\R)$-representations}

In this section we will use the parameterizations from the previous section to compute the Zariski closure of a maximal representation. This will play a key role in Section \ref{minimal}. 

Let $\Xx^{\mathrm{max}}(\sSO_0(2,3))$ denote the subset of the $\sSO_0(2,3)$-character variety which corresponds to $\Mm^{\mathrm{max}}(\sSO_0(2,3)),$ we will call $\rho\in\Xx^{\max}(\sSO_0(2,3))$ a maximal representation.
Using the correspondence between Higgs bundles and representations, for each integer $d\in[0,4g-4]$ we will suggestively denote the connected component of $\Xx^\mathrm{max}(\Gamma,\sSO_0(2,3))$ corresponding to $\Mm_{0,d}^\mathrm{max}(\sSO_0(2,3))$ by $\Xx^\mathrm{max}_{0,d}(\Gamma,\sSO_0(2,3))$. Similarly, for each $(sw_1,sw_2)\in H^1(S,\Z_2)\setminus\{0\}\times H^2(S,\Z_2)$ we will denote the connected component of $\Xx^\mathrm{max}(\Gamma,\sSO_0(2,3))$ corresponding to $\Mm_{sw_1}^{\mathrm{max},sw_2}(\sSO_0(2,3))$ by $\Xx^{\mathrm{max},sw_2}_{sw_1}(\Gamma,\sSO_0(2,3)).$ 
Thus $\Xx^\mathrm{max}(\Gamma,\sSO_0(2,3))$ decomposes as 
\[\bigsqcup\limits_{d\in[0,4]}\Xx^\mathrm{max}_{0,d}(\Gamma,\sSO_0(2,3))\sqcup\bigsqcup\limits_{\substack{(sw_1,sw_2)\in\\ H^1(S,\Z_2)\setminus\{0\}\times H^2(S,\Z_2)}}\Xx^{\mathrm{max},sw_2}_{sw_1}(\Gamma,\sSO_0(2,3))~.\]

To determine when a maximal $\sSO_0(2,3)$-Higgs bundle gives rise to a maximal representation $\rho$ with smaller Zariski closure we need the following definition of reduction of structure group for a Higgs bundle.
\begin{Definition}\label{DEF:Higgs bundle Reduction}
    Let $\sG$ and $\sG'$ be reductive Lie groups with maximal compact subgroups $\sH$ and $\sH'$ and Cartan decompositions $\fg=\fh\oplus\fm$ and $\fg=\fh'\oplus\fm'$. Given $i:\sG'\ra\sG$, we can always assume, up to changing $i$ by a conjugation, that $i(\sH')\subset\sH$ and $di(\fm')\subset\fm.$ A $\sG$-Higgs bundle $(\Pp,\varphi)$ {\em reduces to a $\sG'$-Higgs bundle} $(\Pp',\varphi)$ if the holomorphic $\sH_\C$-bundle $\Pp$ admits a holomorphic reduction of structure group to an $\sH_\C'$-bundle $\Pp'$ and, with respect to this reduction, $\varphi\in H^0(\Pp'[\fm_\C']\otimes K)\subset H^0(\Pp[\fm_\C]\otimes K)$. 
\end{Definition}

Using the non-abelian Hodge correspondence, this definition can be interpreted as a property of the corresponding representation of $\Gamma = \pi_1(\Sigma)$. 

\begin{Proposition}
    Let $\sG'$ be a reductive Lie subgroup of a reductive Lie group $\sG.$ The Zariski closure of a reductive representation $\rho:\Gamma\ra\sG$ is contained in $\sG'$ up to conjugation if and only if the corresponding poly-stable $\sG$-Higgs bundle reduces to a $\sG'$-Higgs bundle. 
\end{Proposition}

 In \cite{BIWmaximalToledoAnnals}, it is shown that if the Zariski closure of a maximal representation $\rho$ is a proper subgroup $\sG'\subset\sSO_0(2,3),$ then $\sG'$ is a group of Hermitian type and the the inclusion map $\sG'\to\sSO_0(2,3)$ is a \emph{tight embedding}. Moreover, the associated representation $\rho:\Gamma\to \sG'$ is also maximal. 

The list of tightly embedded subgroups of $\sO(2,3)$ is as follows \cite{TightHomomorphismClassification}
\begin{itemize}
    \item $\sO(2,1)\times\sO(2)$ where the inclusion is induced by the isometric embedding of $\R^{2,1}\to\R^{2,3}$ which sends $(x_1,x_2,x_3)\to(x_1,x_2,x_3,0,0).$ 
    \item $\sO(2,2)\times \sO(1)$ where the inclusion is induced by the isometric embedding of $\R^{2,2}\to\R^{2,3}$ which sends $(x_1,x_2,x_3,x_4)\to(x_1,x_2,x_3,x_4,0).$ 
    \item $\sO(2,1)$ where the inclusion is induced by the irreducible five dimensional representation of $\sO(2,1).$
\end{itemize} 
Denote the subgroup of $\sO(2,1)\times\sO(2)$ contained in $\sSO_0(2,3)$ by $\sS_0(\sO(2,1) \times \sO(2))$ and the subgroup of $\sO(2,2)\times\sO(1)$ contained in $\sSO_0(2,3)$ by $\sS_0(\sO(2,2)\times\sO(1))$. Both of these groups have two connected component.
The identity component of $\sS_0(\sO(2,1)\times\sO(2))$ is $\sSO_0(2,1)\times\sSO(2)$ and the identity component of $\sS_0(\sO(2,2)\times\sO(1))$ is $\sSO_0(2,2).$ The subgroup of $\sO(2,1)$ contained in $\sSO_0(2,3)$ is $\sSO_0(2,1).$ 

\begin{Proposition}\label{Prop S(O(2,1)xO(2)) Higgs bundle}
    An $\sS_0(\sO(2,1)\times\sO(2))$-Higgs bundle is determined by a tuple $(L,\Ww,\gamma,\beta)$ where $L\in\Pic(\Sigma)$, $\Ww$ is a rank two holomorphic orthogonal bundle, $\gamma\in H^0(L^{-1}\det(\Ww)\otimes K)$ and $\beta\in H^0(L\otimes\det(\Ww)\otimes K).$ Moreover, such a Higgs bundle reduces to an $\sSO_0(2,1)\times\sSO(2)$-Higgs bundle if and only if $\det(\Ww)=\Oo.$
\end{Proposition}
\begin{proof}
    The maximal compact subgroup of $\sO(2,1)\times\sO(2)$ is $\sH=\sO(2)\times\sO(1)\times\sO(2).$ A triple $(A,B,C)\in\sH$ belongs to $\sS_0(\sO(2,1)\times\sO(2))$ if and only if $A\in\sSO(2),$ and $B=\det(C)$. 
    Thus, an $\sS_0(\sO(2,1)\times\sO(2))$-Higgs bundle is given by a holomorphic $\sO(2,\C)$-bundle $\Ww$, a holomorphic $\sO(1,\C)$-bundle given by $\det(\Ww)$, an $\sSO(2,\C)$-bundle $(L\oplus L^{-1})$, and a holomorphic map \[\eta=(\gamma,\beta): L\oplus L^{-1} \to \det(\Ww)\otimes K~.\] 
Such a Higgs bundle reduces to $\sSO_0(2,1)$ if and only if $\det(\Ww)=\Oo$.
\end{proof}

\begin{Proposition}\label{Prop S(O(2,2)xO(1)) Higgs bundle}
    An $\sS_0(\sO(2,2)\times\sO(1))$-Higgs bundle is determined by a tuple $(L,\Ww,\gamma,\beta)$ where $L\in\Pic(\Sigma)$, $\Ww$ is a rank two holomorphic orthogonal bundle, $\gamma\in H^0(L^{-1}\otimes\Ww\otimes K)$ and $\beta\in H^0(L\otimes\Ww\otimes K).$ Moreover, such a Higgs bundle reduces to an $\sSO_0(2,2)$-Higgs bundle if and only if $\det(\Ww)=\Oo.$
\end{Proposition}
\begin{proof}
    The maximal compact subgroup of $\sO(2,2)\times\sO(1)$ is $\sH=\sO(2)\times\sO(2)\times\sO(1).$ A triple $(A,B,C)\in\sH$ belongs to $\sS_0(\sO(2,2)\times\sO(1))$ if and only if $A\in\sSO(2),$ and $\det(B)=C$. 
    Thus, an $\sS_0(\sO(2,2)\times\sO(2))$-Higgs bundle is given by a holomorphic $\sO(2,\C)$-bundle $\Ww$, a holomorphic $\sO(1,\C)$-bundle given by $\det(\Ww)$, an $\sSO(2,\C)$-bundle $(L\oplus L^{-1})$, and a holomorphic map \[\eta=(\gamma,\beta): L\oplus L^{-1} \to \Ww\otimes K~.\] 
Such a Higgs bundle reduces to $\sSO_0(2,2)$ if and only if $\det(\Ww)=\Oo$.
\end{proof}

We have the following characterization of when a maximal $\sSO_0(2,3)$-Higgs bundles reduces to one of the tightly embedded subgroups listed above. 
\begin{Proposition}\label{Prop Higgs reductions M0d}
    A maximal $\sSO_0(2,3)$-Higgs bundle in $\Mm_{0}^{\mathrm{max}}(\sSO_0(2,3))$ determined by a poly-stable tuple $(M,\mu,\nu,q_2)$ reduces to an
    \begin{enumerate}
        \item $\sSO_0(2,1)$-Higgs bundle (irreducibly embedded) if and only if $d=4g-4$ and $\nu=0.$
        \item $\sSO_0(2,1)\times\sSO(2)$-Higgs bundle if and only if $d=0$ and $\mu=\nu=0.$
        \item $\sS_0(\sO(2,2) \times \sO(1))$-Higgs bundle if and only if $d=0$ and $M^2=\Oo$ and $\nu=\lambda\mu$
         for some $\lambda\in\C^*.$
        \item $\sSO_0(2,2)$-Higgs bundle if and only if $d=0$ and $M=\Oo$ and $\nu=\lambda\mu$.
    \end{enumerate}
\end{Proposition}
\begin{proof}
The first part of the proposition follows from the definition of the Hitchin component, and the second part follows directly from Proposition \ref{Prop S(O(2,1)xO(2)) Higgs bundle}.

For the third and fourth parts, consider a maximal $\sSO_0(2,3)$-Higgs bundle determined by $(M,\mu,\nu,q_2)$ with $M^2=\Oo$. 
In this case, the $\sSO(2,\C)$-bundle $(M\oplus M,\smtrx{0&1\\1&0})$ has two holomorphic line subbundles $M_1$ and $M_2$ which are orthogonal and isomorphic to $M$. They are given by 
\[\xymatrix@R0em{M_1\ar[r]& M\oplus M&&\text{and}&&M_2\ar[r]&M\oplus M\\x\ar@{|->}[r]&(x,x)&&&&x\ar@{|->}[r]&(x,-x)}~.\]
In the splitting $M_1\oplus M_2,$ the map $\smtrx{\nu\\\mu}:K^{-1}\to MK\oplus M^{-1}K$ is given by 
\[\mtrx{\mu+\nu\\\mu-\nu}:K^{-1}\to M_1K\oplus M_2K~.\]
If $\mu=\lambda^{-1}\nu$ for some $\lambda\in\C^*,$ then, by the proof of Proposition \ref{Prop: smooth/orbifold in Md}, such a Higgs bundle is isomorphic to the one determined by $(M,\lambda^\haf\mu,\lambda^\haf\mu,q_2).$ In the splitting $M_1\oplus M_2,$ the map $\smtrx{\nu\\\mu}:K^{-1}\to MK\oplus M^{-1}K$ is given by 
\[\smtrx{2\lambda^\haf\mu\\0}:K^{-1}\to M_1K\oplus M_2K~.\] 
Since $\Oo\oplus M_1$ is a holomorphic $\sO(2,\C)$-bundle and $M_1\otimes M_2\otimes\Oo=\Oo,$ by Proposition \ref{Prop S(O(2,2)xO(1)) Higgs bundle}, such a Higgs bundle reduces to an $\sS_0(\sO(2,2)\times\sO(1))$-Higgs bundle. Moreover, this Higgs bundle reduces to $\sSO_0(2,2)$ if and only if $M=\Oo.$

Note that the $\sSO_0(2,3)$-Higgs bundle $(L,\Ww,\beta,\gamma)$ associated to an $\sS_0(\sO(2,2)\times\sO(1))$-Higgs bundle $(L',\Ww',\beta',\gamma')$ is given by $L=L',$ $\Ww=\Ww'\oplus \det(\Ww)$ and
\[\beta=\mtrx{\beta'\\0}:L^{-1}\to \Ww\oplus\det(\Ww)\ \ \ \ \ \text{and}\ \ \ \ \ \ \gamma=\mtrx{\gamma'\\0}:L\to\Ww\oplus\det(\Ww)~.\]
Thus, the non-trivial holomorphic $\sSO_0(2,3)$-gauge transformation given by
\[-Id:L\oplus L^{-1}\to L\oplus L^{-1}\ \ \ \ \ \text{and}\ \ \ \ \ \mtrx{-Id&0\\0&1}:\Ww\oplus\det(\Ww)\to\Ww\oplus\det(\Ww)\]
defines a non-trivial automorphism of the Higgs bundle. By Proposition \ref{Prop: smooth/orbifold in Md}, all other Higgs bundle $\Mm_{0}^\mathrm{max}(\sSO_0(2,3))$ have trivial automorphism groups. Thus, no other Higgs bundles in $\Mm_{0}^\mathrm{max}(\sSO_0(2,3))$ reduce to $\sS_0(\sO(2,2)\times\sO(1)).$ 
\end{proof}
For the components $\Mm_{sw_1}^{\mathrm{max},sw_2}(\sSO_0(2,3))$ with $sw_1\neq0$, we have the following classification of Higgs bundle reductions. 
Recall from Proposition \ref{Prop: pullbackHiggsBundles} that a Higgs bundle in the component $\Mm_{sw_1}^{\mathrm{max},sw_2}(\sSO_0(2,3))$ is determined by a tuple $(M,f,\mu,q_2)$ where $M\in\Prym^{sw_2}(\Sigma_{sw_1})$, $f:M \ra \iota^*M^{-1}$ an isomorphism, $\iota$ is the covering involution of the double cover $\Sigma_{sw_1}$, $\mu\in H^0(M^{-1}K_{sw_1}^2)$ and $q_2\in H^0(K^2).$

\begin{Proposition}\label{Prop Higgs reduction sw1not0}
    Fix $sw_1\in H^1(\Sigma,\Z_2)\neq 0,$ a maximal $\sSO_0(2,3)$-Higgs bundle in $\Mm_{sw_1}^{\mathrm{max},sw_2}(\sSO_0(2,3))$ determined by a poly-stable tuple $(M,f,\mu,q_2)$ 
    \begin{enumerate}
        \item reduces to an $\sS_0(\sO(2,1)\times\sO(2))$-Higgs bundle if and only if $\mu=0,$
        \item reduces to an $\sS_0(\sO(2,2) \times \sO(1))$-Higgs bundle if and only if $M=\iota^* M$, $f = \lambda^{-2} \iota^*f^{-1}$ and $\mu=-\lambda\iota^*(f\mu)$, for some $\lambda \in \C^*$. 
    \end{enumerate}
    If both conditions are met, the Higgs bundle reduces to an $\sS_0(\sO(2,1)\times\sO(1)\times \sO(1))$-Higgs bundle.
\end{Proposition}
\begin{proof}
Let $\pi:\Sigma_{sw_1}\to\Sigma$ be the connected double cover of a nonzero $sw_1\in H^1(\Sigma,\Z_2).$ An $\sSO_0(2,3)$-Higgs bundle on $\Sigma$ reduces to a subgroup of $\sSO_0(2,3)$ if and only if its pullback to $\Sigma_{sw_1}$ reduces. 
Recall that the pullback of a Higgs bundle  in $\Mm_{sw_1}^{\mathrm{max},sw_2}(\Sigma,\sSO_0(2,3))$ determined by a tuple $(M,f,\mu,q_2)$ defines a Higgs bundle in $\Mm_{0,0}^{\mathrm{max}}(\Sigma_{sw_1},\sSO_0(2,3))$ determined by $(M,\mu,\iota^*(\mu f^{-1}),\pi^*(q_2))$.

The result now follows from Proposition \ref{Prop Higgs reductions M0d}.
\end{proof}

Putting together the above propositions we have the following:

\begin{Theorem}
    Let $\Gamma$ be the fundamental group of a closed oriented surface of genus $g\geq2$. If $\rho:\Gamma\to\sSO_0(2,3)$ is a maximal representation which is not in the Hitchin component, then $\rho$ defines a smooth point of the character variety $\Xx^{\mathrm{max}}(\Gamma,\sSO_0(2,3))$ if and only if the image of $\rho$ is Zariski dense. In particular, for $0<d<4g-4$ every representation in $\Xx_{0,d}^{\mathrm{max}}(\Gamma,\sSO_0(2,3))$ is Zariski dense. 
\end{Theorem}

\subsection{Other comments}

The extra invariants for maximal $\sSO_0(2,3)$-Higgs bundles give the following decomposition of $\Mm^{\mathrm{max}}(\sSO_0(2,3))$ as:
\begin{equation}
    \label{maximalSO(2,3)decomp}
    \bigsqcup\limits_{0\leq d\leq 4g-4}\Mm^{\mathrm{max}}_{0,d}(\sSO_0(2,3))  \ \sqcup\    \bigsqcup\limits_{\substack{sw_1\neq0\\ sw_2}}\Mm_{sw_1}^{\mathrm{max},sw_2}(\sSO_0(2,3))~.
\end{equation}

\begin{Remark}
\label{TotalSO(23)ConnectedComponents}
We have shown that every one of the spaces in \eqref{maximalSO(2,3)decomp} is non-empty and connected, so we have $2(2^{2g}-1)+4g-3$ connected components of $\Mm^{\mathrm{max}}(\sSO_0(2,3))$.  In \cite{AndreQuadraticPairs}, it is proven that $\Mm^{\tau,sw_2}(\sSO_0(2,3))$ is connected for $|\tau|<2g-2.$ This gives $2(2(2^{2g}-1)+4g-3) + 4(2g-3)+2= 2^{2g+2}+16g-20$ connected components of $\Mm(\sSO_0(2,3)).$
\end{Remark}

\section{$\sSp(4,\R)$-Higgs bundles}     \label{Sp4R}

In this section, we describe the $\sG$-Higgs bundles in the case when $\sG$ is the group $\sSp(4,\R)=\sSpin_0(2,3).$ 
Recall from Section \ref{section the isomorphism}
that, for $\sSp(4,\R)$, the complexification of the maximal compact subgroup is $\sH_\C=\sGL(2,\C)$ and the complexified Cartan decomposition is given by 
\[\fsp(4,\C)=\fgl(2,\C)\oplus (S^2(V)\oplus S^2(V^*))\]
where $V$ is the standard representation of $\sGL(2,\C)$ and $S^2(V)$ is the symmetric tensor product.

\begin{Definition}\label{DEF of  Sp(4,R) Higgs bundle}
    An $\sSp(4,\R)$-Higgs bundle over $\Sigma$ is given by a triple $(\Vv,\beta,\gamma)$ where $\Vv\to\Sigma$ is a holomorphic rank $2$ vector bundle, $\beta\in H^0(S^2(\Vv)\otimes K)$ and $\gamma\in H^0(S^2(\Vv^*)\otimes K)$. 
    \end{Definition}

The $\sSL(4,\C)$-Higgs bundle associated to an $\sSp(4,\R)$-Higgs bundle $(\Vv,\beta,\gamma)$ is 
\begin{equation}\label{SL(4,C) Higgs of Sp(4,R) Higgs}
    (\Ee,\Phi)=\left(\Vv\oplus \Vv^*, \mtrx{0&\beta\\\gamma&0}\right).
\end{equation}
Here $\beta$ and $\gamma$ are symmetric holomorphic maps 
$\beta:\Vv^*\to \Vv\otimes K$ and $\gamma:\Vv\to \Vv^*\otimes K$.
Since $\Ee=\Vv\oplus \Vv^*$ has a holomorphic symplectic structure $\Omega=\mtrx{0&Id\\-Id&0}$ with respect to which $\phi^T\Omega+\Omega\phi=0,$ this is an $\sSp(4,\C)$-Higgs bundle. 
\begin{Proposition}\label{SO(23)bundleofSp4Rbundle}
    Given an $\sSp(4,\R)$-Higgs bundle $(\Vv,\beta,\gamma)$ the associated $\sSO_0(2,3)$-Higgs bundle is 
    $(L,\Ww,\beta,\gamma)=(\Lambda^2\Vv,S^2(\Vv^*)\otimes\Lambda^2\Vv,\beta,\gamma).$
\end{Proposition}
\begin{proof}
Given an $\sSp(4,\R)$-Higgs bundle $(\Vv,\beta,\gamma),$ the corresponding $\sSO_0(2,3)$-Higgs bundle is determined by the map $\sSp(4,\C)\to\sSO(5,\C)$ described in Section \ref{section the isomorphism}.
For the bundle, one takes the second exterior product
\[  \Lambda^2(\Vv\oplus \Vv^*)\cong\Lambda^2(\Vv)\oplus \Vv\otimes \Vv^*\oplus \Lambda^2(\Vv^*)=\Lambda^2\Vv\oplus\Lambda^2(\Vv^*)\oplus \Hom(\Vv,\Vv)~.\]
The orthogonal structure on this bundle is given by $\mtrx{0&1\\1&0}$ on $\Lambda^2(\Vv)\oplus\Lambda^2(\Vv^*)$ and the Killing form on $\Hom(\Vv,\Vv)$ (i.e. $\langle A,B\rangle=\tr(AB)$).
The symplectic structure $\Omega=\mtrx{0&Id\\-Id&0}\in\Lambda^2(\Vv^*\oplus \Vv)$ corresponds to $Id\in \Hom(\Vv,\Vv)$. If $\Hom_0(\Vv,\Vv)$ is the space of traceless homomorphisms, then
\[\langle\Omega\rangle^\perp= \Lambda^2 \Vv\oplus \Hom_0(\Vv,\Vv)\oplus \Lambda^2\Vv^*.\]  
If $V$ is the standard representation of $\sGL(2,\C)$, then $\Hom_0(V,V)$ is the representation $S^2(V)\otimes \Lambda^2V^*\cong S^2(V)^*\otimes \Lambda^2V.$ 
Thus, 
\[\Hom_0(\Vv,\Vv)=S^2(\Vv)\otimes \Lambda^2\Vv^*\cong S^2(\Vv^*)\otimes\Lambda^2\Vv~.\] 
This gives $L=\Lambda^2\Vv$ and $\Ww=S^2(\Vv)\otimes \Lambda^2\Vv^*\cong S^2(\Vv^*)\otimes\Lambda^2\Vv$.
Finally, note that $\gamma\in H^0(\Sigma, S^2(\Vv^*)\otimes K)=H^0(\Sigma, L^{-1}\otimes \Ww\otimes K)$ and $\beta\in H^0(\Sigma, S^2(\Vv)\otimes K)=H^0(\Sigma, L\otimes \Ww\otimes K)$.
\end{proof}

For an $\sSp(4,\R)$-Higgs bundle $(\Vv,\beta,\gamma)$, the integer $\deg(\Vv) = \deg(\Lambda^2 \Vv)$ is a topological invariant called the \emph{Toledo number}.  This agrees with the Toledo number $\deg(L)$ we defined for the associated $\sSO_0(2,3)$-Higgs bundle.

The Milnor-Wood inequality for $\sSO_0(2,3)$ gives $|\deg(\Vv)|\leq 2g-2$  
(for the original proof of this fact, see \cite{sp4GothenConnComp}.)
 If $\Mm^\tau(\sSp(4,\R))$ is the moduli space of $\sSp(4,\R)$-Higgs bundles $(\Vv,\beta,\gamma)$ with $\deg(\Vv) = \tau$, then 
\[\Mm(\sSp(4,\R))=\bigsqcup\limits_{|\tau|\leq 2g-2}\Mm^\tau(\sSp(4,\R))~.\]

\begin{Proposition}\label{liftingtoSp(4R)Prop}
The image of the map $\pi:\Mm^\tau(\sSp(4,\R))\to \Mm^\tau(\sSO_0(2,3))$ is $\Mm^{\tau,1}(\sSO_0(2,3))$ when $\tau$ is odd and  $\Mm^{\tau,0}(\sSO_0(2,3))$ when $\tau$ is even.
\end{Proposition}
\begin{proof}
If $(L,(\Ww,Q_W),\beta,\gamma)$ is an $\sSO_0(2,3)$-Higgs bundle, then it can be lifted to a Higgs bundle for $\sSp(4,\R) = \sSpin_0(2,3)$ if and only if the structure group of the $\sSO(5,\C)$-bundle $(\Ee,Q_\Ee)$ lifts to $\sSpin(5,\C).$ 
This happens if and only if the second Stiefel-Whitney class $sw_2(\Ee,Q_\Ee)=(\deg(L)\ mod\ 2) + sw_2\left(W,Q_W\right)$ vanishes. 
\end{proof}

Let $(L,(\Ww,Q_W),\beta,\gamma)$ be the $\sSO_0(2,3)$-Higgs bundle associated to $(\Vv,\beta,\gamma).$ 
Note that, for each of the $2^{2g}$ line bundles $I\in\Pic^0(\Sigma)$ with $I^2=\Oo,$ the $\sSO_0(2,3)$-Higgs bundles associated to $(V,\beta,\gamma)$ and $(V\otimes I,\beta,\gamma)$ are the same.

\subsection{Maximal $\sSp(4,\R)$-Higgs bundles}

We now focus on the case when the Toledo number of an $\sSp(4,\R)$-Higgs bundle $(\Vv,\beta,\gamma)$ is maximal.
Maximal $\sSp(4,\R)$-Higgs bundles have been studied in \cite{sp4GothenConnComp} and \cite{MaximalSP4}, our main goal here is to relate previous work with our description of the maximal $\sP\sSp(4,\R)$ components. Using the invariants established for $\sSO_0(2,3)$-maximal Higgs bundles we will write $\Mm^{\mathrm{max}}(\sP\sSp(4,\R))$ as
\[\bigsqcup\limits_{\substack{sw_1\in H^1(\Sigma,\Z_2)\setminus\{0\}\\sw_2\in H^2(\Sigma,\Z_2)}}\Mm_{sw_1}^{\mathrm{max},sw_2}(\sP\sSp(4,\R))\ \sqcup \bigsqcup\limits_{0\leq d\leq 4g-4}\Mm_{0,d}^\mathrm{max}(\sP\sSp(4,\R))~.\]

By Propositions \ref{SO(23)bundleofSp4Rbundle} and \ref{Prop gamma not zero}, for a maximal $\sSp(4,\R)$-Higgs bundle $(\Vv,\beta,\gamma),$ the map $\gamma:\Vv\to \Vv^*\otimes K$ is a holomorphic isomorphism\footnote{This holds more generally for maximal $\sSp(2n,\R)$-Higgs bundles (see \cite{HiggsbundlesSP2nR}).}. Thus, for each choice of square root $K^\haf$, we have an isomorphism
\[\gamma^*\circ\gamma:\Vv\otimes K^{-\haf}\to \Vv\otimes K^{-\haf}~.\]
Moreover, since $\gamma$ is symmetric, the pair $(\Vv\otimes K^{-\haf},\gamma^*\circ\gamma)$ defines a holomorphic $\sO(2,\C)$-bundle. The first and second Stiefel-Whitney class of $\Vv\otimes K^{-\haf}$ help distinguish the connected components of maximal $\sSp(4,\R)$-Higgs bundles. 

If the first Stiefel-Whitney classes of $\Vv\otimes K^{-\haf}$ vanishes, then there is a holomorphic line bundle $N$ with $\deg(N)\geq0$ such that 
\[\Vv=NK^{\haf}\oplus N^{-1}K^\haf.\]
In this case, poly-stability forces $\deg(N)\leq 2g-2$ since $\beta:\Vv^*\to \Vv\otimes K$
is given by 
\begin{equation}\label{eq sp4 beta form}
    \beta=\mtrx{a&b\\b&c}: N^{-1}K^{-\haf}\oplus NK^{-\haf}\to NK^{\frac{3}{2}}\oplus N^{-1}K^{\frac{3}{2}}
\end{equation}
and if $\deg(N)>2g-2,$ then $c=0$ and $NK^{-\haf}\subset \Vv\oplus \Vv^*$ is a positive degree which is invariant by the Higgs field $\Phi=\smtrx{0&\beta\\\gamma&0}.$ Note that when $\deg(N)=2g-2$, $N^2=K$ and we are in one of the $2^{2g}$ Hitchin components for $\sSp(4,\R)$. In the cases $0\leq \deg(N)<2g-2$ and $(sw_1,sw_2)\in H^1(\Sigma,\Z_2)\setminus\{0\}\times H^2(\Sigma,\Z_2)$, Gothen showed that the invariants of the orthogonal bundle distinguish the connected components of the moduli space of maximal $\sSp(4,\R)$-Higgs bundles $\Mm^{\mathrm{max}}(\sSp(4,\R)).$ 

\begin{Theorem}
    (\cite{sp4GothenConnComp}) Fix a square root $K^\haf$ of $K$ and let $\Mm_{sw_1}^{\mathrm{max},sw_2}(\sSp(4,\R))$ denote the set of maximal $\sSp(4,\R)$-Higgs bundles $(\Vv,\beta,\gamma)$ such that the Stiefel-Whitney classes of the orthogonal bundle $(\Vv\otimes K^{-\haf},\gamma^*\circ\gamma)$ are $sw_1$ and $sw_2$ with $sw_1\neq0.$ 
    For $0\leq d<2g-2,$ let $\Mm_{0,d}^{\mathrm{max}}(\sSp(4,\R))$ denote the set of maximal $\sSp(4,\R)$-Higgs bundles $(\Vv,\beta,\gamma)$ such that $\Vv\otimes K^{-\haf}=N\oplus N^{-1}$ for $\deg(N)=d.$ 
    Then the spaces $\Mm_{sw_1}^{\mathrm{max},sw_2}(\sSp(4,\R))$ and $\Mm_{0,d}^{\mathrm{max}}(\sSp(4,\R))$ are nonempty and connected.
\end{Theorem}
Counting the above invariants and adding the $2^{2g}$ $\sSp(4,\R)$-Hitchin components give the following corollary.
\begin{Corollary}
    The space $\Mm^{\mathrm{max}}(\sSp(4,\R))$ has $3\cdot2^{2g}+2g-4$ connected components. 
\end{Corollary}

To obtain the new invariants for a maximal $\sSp(4,\R)$-Higgs bundle we had to fix a square root of the canonical bundle $K$. The associated invariants depend on this choice in the following manner. 
\begin{Proposition}
    \label{Prop sw2 of Sp4R} The first Stiefel-Whitney class of the orthogonal bundle 
    $(\Vv\otimes K^{-\haf},\gamma^*\circ\gamma)$ does not depend on the choice of square root $K^\haf$. The second Stiefel-Whitney class does not depend on the choice of square root if and only if the first Stiefel-Whitney class vanishes.
\end{Proposition}
\begin{proof}
    Recall that two different square roots of $K$ differ by an $I\in\Pic^0(\Sigma)$ with $I^2=\Oo.$ Thus, we need to compare the Stiefel-Whitney classes of $\Vv\otimes K^{-\haf}\otimes I$ with those of $\Vv\otimes K^{-\haf}$. 
    If $\xi$ is a rank two bundle and $\eta$ is a line bundle, then the total Stiefel-Whitney class of $\xi\otimes\eta$ is given by  \cite[Exercise 7.C]{MilnorStasheff}:
    \begin{equation}
        \label{EQ: twisted O2 invariant}
        \xymatrix@=0em{sw(\xi\otimes \eta)=(1+sw_1(\xi)+sw_1(\eta))\wedge(1+ sw_2(\xi)+sw_1(\eta))\\
    =1+sw_1(\xi)+(sw_1(\xi)\wedge sw_1(\eta)+sw_2(\xi))~.}
    \end{equation}
    Thus, the first Stiefel-Whitney classes of $\Vv\otimes K^{-\haf}$ and $\Vv\otimes K^{-\haf}\otimes I$ are the same for all choices of $I.$ 
    The second Stiefel-Whitney classes of $\Vv\otimes K^{-\haf}$ and $\Vv\otimes K^{-\haf}\otimes I$ are the same for all choices of $I$ if and only if $sw_1(\Vv\otimes K^{-\haf})=0.$ 
 \end{proof}

\begin{Proposition}\label{Prop Sp(4) component covering}
Consider the map $\pi:\Mm^{\max}(\sSp(4,\R))\to\Mm^{\mathrm{max},0}(\sP\sSp(4,\R))$:
\begin{itemize}
\item For each $sw_1\in H^1(\Sigma,\Z_2)\setminus\{0\}$, $\pi^{-1}(\Mm_{sw_1}^{\mathrm{max},0}(\sP\sSp(4,\R)))=\Mm^{\mathrm{max}}_{sw_1}(\sSp(4,\R))$, in particular, it has two connected components. 
    \item For $0\leq d\leq 2g-2$, $\pi^{-1}(\Mm_{0,2d}^\mathrm{max}(\sP\sSp(4,\R)))=\Mm_{0,d}^\mathrm{max}(\sSp(4,\R)),$ in particular, it is connected for $d\in[0,2g-2)$ and has $2^{2g}$ connected components when $d=2g-2$.
    \item The inverse image of all the other components (when $sw_2=1$) is empty.
\end{itemize}
\end{Proposition}
\begin{proof}
    By Proposition \ref{liftingtoSp(4R)Prop}, $\Mm^{\mathrm{max}}(\sSp(4,\R))$ is a covering 
    of $\Mm^{\mathrm{max},0}(\sP\sSp(4,\R)).$ 
    Moreover, two maximal $\sSp(4,\R)$-Higgs bundles $(\Vv,\gamma,\beta)$ and $(\Vv',\gamma',\beta')$ map to the same $\sP\sSp(4,\R)$-Higgs bundle if and only if $\Vv'\cong\Vv\otimes I$ with $I^2=\Oo$, $\gamma\cong\gamma'$ and $\beta\cong\beta'.$ 
    For a fixed square root $K^\haf$ of $K,$ let $sw_1$ and $sw_2$ denote the Stiefel-Whitney classes of the orthogonal bundle $\Vv\otimes K^{-\haf}$. 
 The first Stiefel-Whitney class invariant of a maximal $\sSp(4,\R)$-Higgs bundle agrees with the first Stiefel-Whitney class invariant of the associated $\sSO_0(2,3)$-Higgs bundle since $\Lambda^2(\Vv)=\Lambda^2(\Vv\otimes K^{-\haf})\otimes K.$ 
    Thus, $\Mm_{sw_1}^{\mathrm{max}}(\sSp(4,\R))$ is a covering of $\Mm_{sw_1}^{\mathrm{max},0}(\sP\sSp(4,\R)).$

If $sw_1\neq0$, then $\Mm_{sw_1}^{\mathrm{max}}(\sSp(4,\R))$ has two connected components which are distinguished by the second Stiefel-Whitney class of $\Vv\otimes K^{-\haf}.$
    If $sw_1=0$ and $(\Vv,\gamma,\beta)$ is a maximal Higgs bundle in $\Mm_{0,d}^\mathrm{max}(\sSp(4,\R)),$ then $\Vv=NK^\haf\oplus N^{-1}K^\haf$ for some $N\in\Pic^d(\Sigma).$ The bundle $\Lambda^2(\Vv\oplus\Vv^*)$ is then given by 
    \[K\oplus N^2\oplus \Oo\oplus \Oo\oplus  N^{-2}\oplus K^{-1}~.\]
    Thus, for $0\leq d<2g-2$, the space $\Mm_{0,d}^\mathrm{max}(\sSp(4,\R))$ maps to
$\Mm_{0,2d}^\mathrm{max}(\sP\sSp(4,\R)),$ and the space $\Mm_{0,2g-2}^\mathrm{max}(\sSp(4,\R))$ maps to
$\Mm_{0,4g-4}^\mathrm{max}(\sP\sSp(4,\R))$.
 \end{proof}

\subsection{Parameterizing $\Mm^\mathrm{max}(\sSp(4,\R))$}\label{sec: param sp4} We now turn to parameterizing the connected components of $\Mm^{\mathrm{max}}(\sSp(4,\R))$ as coverings of the parameterizations of the components of $\Mm^{\mathrm{max},0}(\sP\sSp(4,\R))$ from Theorems \ref{THM d>0}, \ref{thm:zero_component} and \ref{THM HiggsParamsw1sw2orbifold}. Recall that the Abel-Jacobi map sends a divisor $D$ to the line bundle $\Oo(D);$ this defines a map 
 \[a:\xymatrix{\Sym^{m}(\Sigma)\ar[r]&\Pic^{m}(\Sigma)}~.\]
Recall also that the squaring map defines a $2^{2g}$-covering $s:\Pic^{m}(\Sigma)\to\Pic^{2m}(\Sigma).$
The fiber product 
\begin{equation}
    \label{EQ fiberproduct symm prod}a^*\Pic^{m}(\Sigma)=\{(D,L)\in\Sym^{2m}(\Sigma)\times\Pic^m(\Sigma)\ |\ a(D)=L^2\}
\end{equation}
thus defines a smooth $2^{2g}$-covering of the symmetric product $\Sym^{2m}(\Sigma).$
\begin{Theorem}\label{THM SP4 d>0}
Let $\Sigma$ be a Riemann surface with genus $g\geq 2.$ For $0<d\leq2g-2,$ the subspace $\Mm_{0,d}^\mathrm{max}(\sSp(4,\R)$ of the moduli space of $\sSp(4,\R)$-Higgs bundles on $\Sigma$ is diffeomorphic to $\pi^*\Ff_{2d}\times H^0(K^2)$ where 
\begin{itemize}
    \item $H^0(K^2)$ is the space of holomorphic quadratic differentials on $\Sigma$,
    \item$\Ff_{2d}$ is the rank $3g-3+2d$ holomorphic vector bundle over the symmetric product $\Sym^{4g-4-2d}(\Sigma)$ from Theorem \ref{THM d>0},
    \item$\pi:a^*\Pic^{2g-2-d}(\Sigma)\to\Sym^{4g-4-2d}(\Sigma)$ is the $2^{2g}$-covering given by \eqref{EQ fiberproduct symm prod}.
\end{itemize}
In particular, $\Mm_{0,d}^\mathrm{max}(\sSp(4,\R))$ is a $2^{2g}$-covering of $\Mm_{0,2d}^\mathrm{max}(\sP\sSp(4,\R)).$
\end{Theorem}

\begin{proof}
Similar to the proof of Theorem \ref{THM d>0}, set 
\[\widehat \Ff_{2d}=\{(M,\mu,\nu)\ |\ M\in \Pic^{2d}(\Sigma),\ \mu\in H^0(M^{-1}K^2)\setminus\{0\}, \ \nu\in H^0(MK^2) \}\]
and let $\Ff_{2d}=\widehat\Ff_{2d}/\C^*$ where $\lambda\in\C^*$ acts by $\lambda\cdot(M,\mu,\nu)=(M,\lambda\mu,\lambda^{-1}\nu).$ Recall that the map which associates to an equivalence class $[(M,\mu,\nu)]$ the projective class of $\mu$ turns $\Ff_{2d}$ into a rank ($3g-3+2d$)-vector bundle over the symmetric product $\Sym^{4g-4-2d}(\Sigma).$
 Let $s:\Pic^{2g-2-d}(\Sigma)\to\Pic^{4g-4-2d}(\Sigma)$ be the $2^{2g}$ covering defined by the squaring map. Pulling back this covering by the Abel-Jacobi map $a:\Sym^{4g-4-2d}(\Sigma)\to\Pic^{4g-4-2d}(\Sigma)$ defines a $2^{2g}$-covering 
 \[\pi:\xymatrix{a^*\Pic^{2g-2-d}(\Sigma)\ar[r]&\Sym^{4g-4-d}(\Sigma)}~.\] 
 This covering can be interpreted as the space of effective divisors $D$ of degree $4g-4-2d$ together with a choice of square root of $a(D).$ 
 Finally, the pullback $\pi^*\Ff_{2d}$ of the vector bundle $\Ff_{2d}$ to $a^*\Pic^{2g-2-d}(\Sigma)$ can be interpreted as the set of tuples consisting of a point in $\Ff_{2d}$ together with choice square root of the line bundle associated to the corresponding effective divisor. 

By Proposition \ref{Prop Sp(4) component covering}, for $0< d\leq 2g-2$, the space $\Mm_{0,d}^{\max}(\sSp(4,\R))$ is a connected covering of $\Mm_{0,2d}^{\max}(\sP\sSp(4,\R)).$ Recall from \eqref{eq sp4 beta form} that, after fixing a square root $K^\haf$ of $K,$ a Higgs bundle in $\Mm_{0,d}^{\max}(\sSp(4,\R))$ is determined by 
\[(\Vv,\beta,\gamma)=\left(NK^\haf\oplus N^{-1}K^\haf,\mtrx{\nu&q_2\\q_2&\mu},\mtrx{0&1\\1&0}\right)~,\]
where $N\in\Pic^d(\Sigma),$ $q_2\in H^0(K^2)$, $\nu\in H^0(N^{2}K^2)$ and $\mu\in H^0(N^{-2}K^2).$ As in the $\sP\sSp(4,\R)$-case, the section $\mu$ must be nonzero by stability. Thus, such a Higgs bundle is determined by a tuple $(\mu,\nu,q_2)$ and a choice of square root of the line bundle $N^{-2}K^2$. 

Two such Higgs bundles 
$\bigg(NK^\haf\oplus N^{-1}K^\haf,\mtrx{\nu&q_2\\q_2&\mu},\mtrx{0&1\\1&0}\bigg)$ 
and 
$\bigg(N'K^\haf\oplus N'^{-1}K^\haf, \mtrx{\nu'&q_2'\\q_2'&\mu'},\mtrx{0&1\\1&0}\bigg)$ 
are isomorphic if and only if $N=N'$ and there is a holomorphic gauge transformation $g:\Vv\to \Vv$ so that 
\[(g^{-1})^T\mtrx{0&1\\1&0}g^{-1}=\mtrx{0&1\\1&0}\ \ \ \ \ \text{and}\ \ \ \ \ g\mtrx{\nu&q_2\\q_2&\mu}g^T=\mtrx{\nu'&q_2'\\q_2'&\mu'}~.\] 
Thus, in the splitting $\Vv=NK^\haf\oplus N^{-1}K^\haf$, we have $g=\mtrx{\lambda&0\\0&\lambda^{-1}}$ for $\lambda\in\C^*$ and such a gauge transformation acts by $g\mtrx{\nu&q_2\\q_2&\mu}g^T=\mtrx{\lambda^2\nu&q_2\\q_2&\lambda^{-2}\mu}.$ 

In particular, if $[\mu]$ denotes the degree $4g-4-2d$ effective divisor associated to the projective class of $\mu,$ then the isomorphism class of such an $\sSp(4,\R)$-Higgs bundle is uniquely determined by the data $([\mu],\nu,q_2)$ and a choice of square root of $a([\mu])$. Thus, the component $\Mm_{0,d}^{\max}(\sSp(4,\R))$ is diffeomorphic to 
$\pi^*\Ff_{2d}\times H^0(K^2).$
\end{proof}
\begin{Remark}
    In the special case of $0<d<g-1$, a different parametrization of the components $\Mm_{0,d}^{\max}(\sSp(4,\R))$ was given in \cite{MaximalSP4}.
\end{Remark}
\begin{Theorem} \label{THM Sp4 d=0 }
Let $\Sigma$ be a Riemann surface with genus $g\geq 2.$ The connected component $\Mm_{0,0}^\mathrm{max}(\sSp(4,\R)$ of the moduli space of $\sSp(4,\R)$-Higgs bundles is homeomorphic to $s^*\Aa/\Z_2\times H^0(K^2)$ where 
\begin{itemize}
    \item $H^0(K^2)$ is the space of holomorphic differentials on $\Sigma$,
    \item $\Aa$ is the holomorphic fiber bundle over $\Pic^0(\Sigma)$ from Theorem \ref{thm:zero_component},
    \item $s:\Pic^0(\Sigma)\to\Pic^0(\Sigma)$ is the squaring map, 
    \item $\Z_2$ acts on $s^*\Aa$ by pullback by inversion on $\Pic^0(\Sigma).$ 
\end{itemize} 
In particular, $\Mm_{0,0}^\mathrm{max}(\sSp(4,\R))$ is a $2^{2g}$-covering of $\Mm_{0,0}^\mathrm{max}(\sP\sSp(4,\R)).$
\end{Theorem}
\begin{proof}
We will use the same notation as the proof of Theorem \ref{THM SP4 d>0}. As in the proof of Theorem \ref{thm:zero_component} consider the space
\[\widetilde\Aa=\{(L,\mu,\nu)\ |\ L\in\Pic^0(\Sigma),\ \mu\in H^0(L^{-1}K^2),\ \nu\in H^0(LK^2)\}~.\]
Let $\widehat\Aa\subset\widetilde\Aa$ denote the set of tuples $(L,\mu,\nu)$ with $\mu=0$ if and only if $\nu=0,$ and set $\Aa=\widehat\Aa/\C^*$ where $\lambda\in\C^*$ acts as $\lambda\cdot(L,\mu,\nu)=(L,\lambda\mu,\lambda^{-1}\nu).$ 
The map which takes an equivalence class $[(L,\mu,\nu)]$ to $L\in\Pic^0(\Sigma)$ turns $\Aa$ into a holomorphic bundle over $\Pic^0(\Sigma).$ 
The pullback $s^*\Aa$ by the squaring map is a holomorphic bundle over $\Pic^0(\Sigma)$ which parameterizes points in $\Aa$ together with a choice of square root of the associated line bundle $L\in\Pic^0(\Sigma).$ Recall finally that $\Z_2$ acts on $\Pic^0(\Sigma)$ by inversion and pullback lifts this action to a $\Z_2$ action on $\Aa$. Denote the the quotient of $\Aa$ by this action by $\Aa/\Z_2.$

By Proposition \ref{Prop Sp(4) component covering}, the space $\Mm_{0,0}^{\max}(\sSp(4,\R))$ is a connected covering of the component $\Mm_{0,0}^{\max}(\sP\sSp(4,\R)).$ Recall from \eqref{eq sp4 beta form} that, after fixing a square root $K^\haf$ of $K,$ a Higgs bundle in $\Mm_{0,0}^{\max}(\sSp(4,\R))$ is determined by 
\[(\Vv,\beta,\gamma)=\left(NK^\haf\oplus N^{-1}K^\haf,\mtrx{\nu&q_2\\q_2&\mu},\mtrx{0&1\\1&0}\right)~,\]
where $N\in\Pic^0(\Sigma),$ $q_2\in H^0(K^2)$, $\nu\in H^0(N^{2}K^2)$ and $\mu\in H^0(N^{-2}K^2).$ As in the $\sP\sSp(4,\R)$-case, polystability forces $\mu=0$ if and only if $\nu=0$. Thus, such a Higgs bundle is determined by a tuple $(N^2,\mu,\nu,q_2)$ together with a choice of square root of the line bundle $N^2\in\Pic^0(\Sigma)$.

Two such Higgs bundles $\bigg(NK^\haf\oplus N^{-1}K^\haf,\mtrx{\nu&q_2\\q_2&\mu},\mtrx{0&1\\1&0}\bigg)$ and $\bigg(N'K^\haf\oplus N'^{-1}K^\haf, \mtrx{\nu'&q_2'\\q_2'&\mu'},\mtrx{0&1\\1&0}\bigg)$ are isomorphic if and only if $N=N'$ or $N^{-1}=N'$ and there is a holomorphic gauge transformation $g:\Vv\to \Vv$ so that 
\[(g^{-1})^T\mtrx{0&1\\1&0}g^{-1}=\mtrx{0&1\\1&0}\ \ \ \ \ \text{and}\ \ \ \ \ g\mtrx{\nu&q_2\\q_2&\mu}g^T=\mtrx{\nu'&q_2'\\q_2'&\mu'}~.\] 
Thus, in the splitting $\Vv=NK^\haf\oplus N^{-1}K^\haf$, we have $g=\mtrx{\lambda&0\\0&\lambda^{-1}}$ or $g=\mtrx{0&\lambda\\\lambda^{-1}&0}$ for $\lambda\in\C^*$. Such gauge transformations act by
\[\mtrx{\lambda&0\\0&\lambda^{-1}}\mtrx{\nu&q_2\\q_2&\mu}\mtrx{\lambda&0\\0&\lambda^{-1}}=\mtrx{\lambda^2\nu&q_2\\q_2&\lambda^{-2}\mu}\] and
\[\mtrx{0&\lambda\\\lambda^{-1}&0}\mtrx{\nu&q_2\\q_2&\mu}\mtrx{0&\lambda^{-1}\\\lambda&0}=\mtrx{\lambda^2\mu&q_2\\q_2&\lambda^{-2}\nu}~.\]
In particular, the isomorphism class of a Higgs bundle in $\Mm_{0,0}^{\max}(\sSp(4,\R))$ is uniquely determined by a point it in $s^*\Aa/\Z_2$ and a holomorphic quadratic differential. Thus, $\Mm_{0,0}^{\max}(\sSp(4,\R))$ is homeomorphic to $s^*\Aa/\Z_2\times H^0(K^2).$
\end{proof}
For $sw_1\in H^1(\Sigma,\Z_2)\setminus\{0\},$ let $\Sigma_{sw_1}$ be the corresponding unramified covering of $\Sigma$ and denote the covering involution by $\iota.$ 

\begin{Proposition}\label{Prop squaring prym}
    For each $sw_1\in H^1(\Sigma,\Z_2)\setminus\{0\}$, the squaring map 
    \[\xymatrix@R=.2em{s:\Prym(\Sigma_{sw_1})\ar[r]&\Prym(\Sigma_{sw_1})\\M\ar@{|->}[r]&M^2}\] is a $2^{2g-1}$ covering of the connected component of the identity $\Prym^0(\Sigma_{sw_1})$.
\end{Proposition}
\begin{proof}
    Recall from Proposition \ref{Prop Prym disconnected} that $\Prym(\Sigma_{sw_1})\subset \Pic^0(\Sigma_{sw_1})$ is the disjoint union of two isomorphic connected components and the connected component of the identity $\Prym^0(\Sigma_{sw_1})$ is an abelian variety of dimension $g-1.$ Moreover, recall that a line bundle $M\in\Prym(\Sigma_{sw_1})$ lies in $\Prym^0(\Sigma_{sw_1})$ if $M=L\otimes\iota^*L^{-1}$ for $L$ an even degree line bundle on $\Sigma_{sw_1}$ and $M\in\Prym^1(\Sigma_{sw_1})$ if $M=L\otimes \iota^*L^{-1}$ for $L$ an odd degree line bundle on $\Sigma_{sw_1}.$ Thus, the square $M^2$ of a line bundle $M\in\Prym^1(\Sigma_{sw_1})$ lies in $\Prym^0(\Sigma_{sw_1}).$
    
 Since $\Prym^0(\Sigma_{sw_1})$ is an abelian variety of dimension $g-1$, the restriction of the map $s$ to $\Prym^{0}(\Sigma_{sw_1})$ is a $2^{2g-2}$ cover. As $\Prym^1(\Sigma_{sw_1})$ is a $\Prym^0(\Sigma_{sw_1})$-torsor, we conclude that the squaring map is a $2^{2g-1}$-covering of $\Prym^0(\Sigma_{sw_1}).$
\end{proof}

We will consider the direct sum $s^*\Ee \oplus \Jj$, a vector bundle of rank $6g-5$ over $\Prym(\Sigma_{sw_1})$ whose total space parametrizes the tuples $(N,j,\mu)$ where $N\in \Prym(\Sigma_{sw_1})$, $j\in \mathrm{End}(N, \iota^* N^{-1})$, $\mu \in H^0(\Sigma_{sw_1},N^{-2} K_{sw_1}^2)$. We will denote by $\Kk$ the open subset:
$$\{ (N,j,\mu) \in s^*\Ee \oplus \Jj \mid j\neq 0 \} $$
parameterizing the tuples $(N,j,\mu)$ where $j$ is an isomorphism. 
We define an action of $\C^*$ on the total space of $\Kk$ via the following formula:
\begin{equation}
    \label{eq C* Sp4 action}\lambda\cdot(N,j,\mu)=(N,\lambda^2j,\lambda^2 \mu)~.
\end{equation}

The quotient $\Kk'=\Kk/\C^*$ is the space parameterizing all the gauge equivalence classes of the triples $(N,j,\mu)$, with $j$ an isomorphism. This space is a vector bundle over $\Prym(\Sigma_{sw_1})$ whose fiber over the point $N$ is isomorphic to $H^0(N^{-2}K^2_{sw_1}).$
On the space $\Kk'$ we have a $\Z_2$-action given by 
\begin{equation}
    \label{eq Z2 action Sp4}\tau \cdot [(N,j,\mu)] = [(\iota^*N, \iota^*j, \iota^*\mu)]~.
\end{equation}

 Recall that the space $\Hh$ from \eqref{eq H definition} was defined to be the subset of $\Ee\oplus \Jj\to \Prym(\Sigma_{sw_1})$ consisting of pairs $(M,f)$ where $f:M\to\iota^*M^{-1}$ is an isomorphism. Let $\Hh_0$ denote the connected component of $\Hh$ which maps to the identity component of $\Prym(\Sigma_{sw_1}).$
The natural map
\[\xymatrix@R=0em{s^*\Ee\oplus\Jj\ar[r]& \Ee\oplus \Jj~.\\(N,j,\mu)\ar@{|->}[r]&(N^2,j^2,\mu)}\]
defines a map $\Kk\to \Hh_0$ which is a covering of degree $2^{2g-1}$. Moreover, this map is equivariant with respect to the $\C^*$ actions defined by \eqref{eq C* Sp4 action} and \eqref{eq C* SO23}, and thus descends to the quotients
\[p:\Kk'\to \Hh_0'~.\]
Moreover, $p([N,j,\mu])=[M,f,\mu']$ if and only if $M=N^2$, $f=j^2$ and $\mu'=\pm \mu,$ thus, the map  $p:\Kk'\to \Hh'_0$ is a covering of degree $2^{2g}.$

Finally, note that the map $p:\Kk'\to \Hh'_0$ is equivariant with respect to the $\Z_2$ actions defined by \eqref{eq Z2 action Sp4} and \eqref{eq Z2 action So23}. Thus, the map 
\begin{equation}
    \label{eq def of covering} p:\Kk'/\Z_2\to\Hh'_0/\Z_2
\end{equation}
is a $2^{2g}$-covering.

\begin{Theorem}\label{THM Higgs Sp4 sw1not0}
    Let $\Sigma$ be a Riemann surface of genus $g\geq 2.$ For each $sw_1\in H^1(\Sigma,\Z_2)\setminus\{0\},$ the subspace $\Mm_{sw_1}^\mathrm{max}(\sSp(4,\R))$ of the moduli space of $\sSp(4,\R)$-Higgs bundles on $\Sigma$ is diffeomorphic to $\Kk'/\Z_2\times H^0(K^2)$.
    Moreover, the natural map 
    $$\Mm_{sw_1}^{\mathrm{max}}(\sSp(4,\R)) \ra \Mm_{sw_1}^{\mathrm{max},0}(\sP\sSp(4,\R))$$
    is given by \eqref{eq def of covering}. Hence, $\Mm_{sw_1}^{\mathrm{max}}(\sSp(4,\R))$ is a $2^{2g}$-covering of $\Mm_{sw_1}^{\mathrm{max},0}(\sP\sSp(4,\R))$ with two connected components
\end{Theorem}
\begin{proof}
Fix $sw_1\in H^1(\Sigma,\Z_2)\setminus\{0\}$ and let $\pi:\Sigma_{sw_1}\to\Sigma$ be the associated double cover. By Proposition \ref{Prop Sp(4) component covering}, the space $\Mm_{sw_1}^{\max}(\sSp(4,\R))$ is the disjoint union of two isomorphic connected components which cover $\Mm_{sw_1}^{0,max}(\sP\sSp(4,\R)).$ 
Let $(V,\beta,\gamma)$ be an $\sSp(4,\R)$-Higgs bundle in $\Mm_{sw_1}^{\max}(\sSp(4,\R))$ and recall that the bundle $(\Lambda^2\Vv\otimes K^{-1})$ is a holomorphic $\sO(1,\C)$-bundle with first Stiefel-Whitney class $sw_1.$ 
By \eqref{eq sp4 beta form}, 
\[\pi^*\Vv\cong N\pi^*K^\haf\oplus N^{-1}\pi^*K^{\haf}, \ \ \gamma=\mtrx{0&1\\1&0},\ \ \beta=\mtrx{\nu&q_2'\\q_2'&\mu}~,\] where $N\in\Pic^0(\Sigma_{sw_1}),$ $\mu\in H^0(N^{-2}\pi^*K^2)$, $\nu\in H^0(N^2\pi^*K^2)$ and $q_2'\in H^0(\pi^*K^2).$ 
Moreover, since the pullback is invariant under the covering involution, we have 
\[N\in\Prym(\Sigma_{sw_1}),\ \ \iota^*\mu=\nu,\ \ \ \text{and}\ \ \ q_2'=\pi^*q_2 \]
where $q_2\in H^0(K^2).$ Thus, such a Higgs bundle is determined by a holomorphic quadratic differential $q_2\in H^0(K^2)$ and a point in $s^*\Ee.$

As in the proof of Theorem \ref{THM Sp4 d=0 }, two such Higgs bundles on $\Sigma_{sw_1}$ are gauge equivalent if and only if there is a gauge transformation of $N\pi^*K^\haf\oplus N^{-1}\pi^*K^\haf$ with the form $\mtrx{\lambda&0\\0&\lambda^{-1}}$ or $\mtrx{0&\lambda\\\lambda^{-1}&0}$ for $\lambda\in\C^*.$ Such a gauge transformation descends to a gauge transformation of $\Vv$ if and only if $\iota^*g= g,$ and hence $\lambda=\pm1.$ 
 The gauge transformation $\mtrx{-1&0\\0&-1}$ acts on $\beta$ by
\[\mtrx{-1&0\\0&-1}\mtrx{\iota^*\mu&\pi^*q_2\\\pi^*q_2&\mu}\mtrx{-1&0\\0&-1}=\mtrx{\iota^*\mu&\pi^*q_2\\\pi^*q_2&\mu}\]
and the gauge transformation $\mtrx{0&\pm1\\\pm1&0}$ sends $N\pi^*K^{\haf}$ to $N^{-1}\pi^*K^\haf$ and acts on $\beta$ by 
\[\mtrx{0&\pm1\\\pm1&0}\mtrx{\iota^*\mu&\pi^*q_2\\\pi^*q_2&\mu}\mtrx{0&\pm1\\\pm1&0}=\mtrx{\mu&\pi^*q_2\\\pi^*q_2&\iota^*\mu}~.\]
The proof that $\Mm_{sw_1}^\mathrm{max}(\sSp(4,\R))$ is diffeomorphic to $\Kk'/\Z_2\times H^0(K^2)$ follows from the same arguments as the proof of Theorem \ref{THM HiggsParamsw1sw2orbifold}.
\end{proof}

\section{Unique minimal immersions}      \label{minimal}

In this section we show that for each maximal $\sP\sSp(4,\R)$-representation $\rho$ there exists a unique $\rho$-equivariant minimal immersion 
from the universal cover of $S$ to the Riemannian symmetric space of $\sP\sSp(4,\R)$. We start by recalling some basic facts about harmonic maps and the nonabelian Hodge correspondence between Higgs bundles and character varieties.

\subsection{Harmonic metrics}
For a principal $\sG$-bundle $P$, a reduction of structure group to a subgroup $\sH\leq\sG$ is equivalent to a $\sG$-equivariant map $r:P\to\sG/\sH.$
A reduction of structure group for a flat bundle is equivalent a $\Gamma$-equivariant map $r:\widetilde S\to\sG/\sH$.
A {\em metric} on a $\sG$-bundle $P$ is defined to be a reduction of structure group to the {\em maximal compact subgroup} $\sH\leq\sG.$

Let $\rho:\Gamma\to\sG$ be a representation and let $h_\rho:\widetilde S\to \sG/\sH$ be a metric on the associated flat bundle $\widetilde S\times_\rho \sG$. 
Given a metric $g$ on $S$, one can define the norm $\Vert dh_\rho \Vert$ of $dh_\rho$ which, by equivariance of $h_\rho$, descends to a function on $S$. The {\em energy} of $h_\rho$ is the $L^2$-norm of $dh_\rho$, namely:
\[E(h_\rho)=\int\limits_S \Vert dh_\rho\Vert^2dvol_g~.\]
Note that the energy of $h_\rho$ depends only on the conformal class of the metric $g,$ and so, only on the Riemann surface structure $\Sigma$ associated to $g$.

\begin{Definition}
A metric $h_\rho: \widetilde \Sigma\rightarrow \sG/\sH$ on $\widetilde\Sigma\times_\rho\sG$ is \textit{harmonic} if it is a critical point of the energy functional.  
\end{Definition}
Let $\nabla^{0,1}$ denote the holomorphic structure on $\left( T^*\Sigma\otimes h_\rho^*T(\sG/\sH)\right)\otimes\C$ induced by the Levi-Civita connection on $\sG/\sH$. The following is classical:
\begin{Proposition}\label{p-harmonicholomorphic}
    A metric $h_\rho:\widetilde X\rightarrow \sG/\sH$ is harmonic if and only if the $(1,0)$ part $\partial h_\rho$ of $dh_\rho$ is holomorphic, that is 
    $\nabla^{0,1}\p h_\rho=0.$
\end{Proposition}

The following theorem, proven by Donaldson \cite{harmoicmetric} for $\sSL(2,\C)$ and Corlette \cite{canonicalmetrics} in general, is the starting point of our analysis.
\begin{Theorem}\label{CorletteTheorem}
    Let $\rho\in \Xx(\Gamma,\sG)$, for each Riemann surface structure $\Sigma$ on $S$ there is a metric $h_\rho:\widetilde\Sigma\to\sG/\sH$ on $\widetilde\Sigma\times_\rho\sG$ which is harmonic. Furthermore, $h_\rho$ is unique up to the action of the centralizer of $\rho.$
\end{Theorem}

A homogeneous space $\sG/\sH$ is called {\em reductive} if the Lie algebra $\fg$ has an $Ad_\sH$-invariant decomposition $\fg=\fh\oplus\fm.$ If $W$ is a linear representation of $\sH$, denote the associated bundle $\sG\times_\sH W\to \sG/\sH$ by $[W].$ 
The tangent bundle $T(\sG/\sH)$ of $\sG/\sH$ is isomorphic to $[\fm]$.
Since the action of $\sH$ on $\fg$ is the restriction of the $\sG$ action, the bundle $[\fg]$ is trivializable. 
Furthermore, the inclusion $T(\sG/\sH)\cong[\fm]\subset[\fg]\cong\sG/\sH\times\fg$ can be interpreted as an equivariant 1-form $\omega$ on $\sG/\sH$ valued in $\fg,$ $\omega\in\Omega^1(M,\fg).$ 

\begin{Definition}\label{MCFormDef}
    The equivariant $\fg$-valued 1-form $\omega\in\Omega^1(\sG/\sH,\fg)$ is called the {\em Maurer-Cartan form} of the homogeneous space $\sG/\sH.$
\end{Definition}
The Maurer-Cartan form $\omega_\sG\in\Omega^1(\sG,\fg)^\sG$ of $\sG$ is $\sG$-equivariant, and admits an $\sH$-equivariant splitting $\omega_\sG=pr_\fh\omega_\sG\oplus pr_\fm\omega_\sG$,
where:
\[\xymatrix{pr_\fh\omega_\sG\in\Omega^1(\sG,\fh)^\sH&\text{and}&  pr_\fm\omega_\sG\in\Omega^1(\sG,\fm)^\sH}~.\]
The form $pr_\fh\omega_\sG$ defines a connection on the principal $\sH$-bundle $\sG\to \sG/\sH$ called the {\em canonical connection}. We will denote the corresponding covariant derivative on an associate bundle by $\nabla^c.$
The form $pr_\fm\omega_\sG$ is an equivariant $1$-form which vanishes on vertical vector fields.
Thus, $pr_\fm\omega_\sG$ descends to a 1-form on $\sG/\sH$ valued in $[\fm]$ which is the Maurer-Cartan form $\omega.$ The following is classical (see chapter 1 of \cite{TwistorTheoryHarmonicMapsBOOK}).

\begin{Lemma}\label{CanonicalConnectionandTorsion}
    Let $f:\sG/\sH\to \sG/\sH\times V$ be a smooth section of the trivial bundle, then
    $df=\nabla^c f+\omega\cdot f.$
    If $V=\fg$ is the adjoint representation, then $\nabla^c=d-ad_\omega$ and the torsion is given by
    $T_{\nabla^c}=-\haf[\omega,\omega]^\fm.$
 \end{Lemma} 
For any reductive homogeneous space, a smooth map $f:S\to\sG/\sH$ defines a principal $\sH$-bundle $f^*\sG\to S$ with a connection $f^*\nabla^c;$ furthermore, the derivative $df\in\Omega^0(S,T^*S\otimes f^*T(\sG/\sH))=\Omega^0(S,T^*S\otimes f^*[\fm])$ is identified with the pullback of the Maurer-Cartan form $f^*\omega.$ 

Let $\Sigma$ be a Riemann surface structure on $S$.
Complexifying the splitting $\fg=\fh\oplus\fm$ gives an $Ad_{\sH_\C}$-invariant splitting $\fg_\C=\fh_\C\oplus\fm_\C$. 
Thus $T(\sG/\sH)\otimes\C=[\fm_\C]$ and the complex linear extension of the Maurer-Cartan form $\omega$ is a $1$-form valued in $[\fm_\C].$
The $(0,1)$-part of $f^*\nabla^c$ defines a holomorphic structure on the $\sH_\C$-bundle $f^*\sG\times_{\sH}\sH_\C.$ 

\begin{Example}\label{symmspaceFlat} 
 For a Cartan involution, the splitting $\fg=\fh\oplus\fm$ is orthogonal and $B_\fg$ is positive definite on $\fm$ and negative definite on $\fh.$ 
Thus, $\fh$ is the Lie algebra of a maximal compact subgroup $\sH\subset\sG$ and $B_\fg$ induces a $\sG$-invariant  Riemannian metric on $\sG/\sH$.  
Since $[\fm,\fm]\subset\fh$, by Lemma \ref{CanonicalConnectionandTorsion}, the canonical connection is the Levi-Civita connection on $(\sG/\sH,B_\fg)$.
The flatness equations of $\nabla^c+ad_\omega$ 
on $\sG/\sH\times\fg$ decompose as
 \begin{equation}\label{flateqSymmspace}
    \begin{dcases}
   F_{\nabla^c}+\haf[\omega,\omega]=0 & \text{ on } \fh,\\
   \nabla^c\omega=0& \text{ on } \fm,
\end{dcases}
 \end{equation}

and a map $f:\Sigma\to\sG/\sH$ is harmonic if and only if $(f^*\nabla^{c})^{0,1}(f^*\omega)^{1,0}=0.$ 
\end{Example}

Let $\lambda$ be the real conjugation giving $\fg_\C=\fg\otimes\C,$ and denote the extension of $\lambda$ to forms again by $\lambda:\Omega^*(\sG/\sH,[\fm_\C])\to \Omega^*(\sG/\sH, [\fm_\C]).$ 
Denote the compact real from $\theta\circ\lambda$ of $\fg_\C$ by $\tau,$ and note that $\tau=-\lambda$ on $\Omega^1(\sG/\sH,[\fm_\C]).$  
Pulling back $\lambda$ by a map $f:\Sigma\to\sG/\sH$ defines a conjugation on forms $f^*\lambda:\Omega^{i,j}(\Sigma,f^*[\fm_\C])\to\Omega^{j,i}(\Sigma,f^*[\fm_\C])$. 
Since $\omega$ is real, we have
\[f^*\omega=f^*\omega^{1,0}+f^*\omega^{0,1}=f^*\omega^{1,0}+f^*\lambda(f^*\omega^{1,0})=f^*\omega^{1,0}-f^*\tau(f^*\omega^{1,0})~.\]
Thus, putting the flatness and harmonic equations together yields:
\begin{Proposition}
   Let $f:\widetilde\Sigma\to\sG/\sH$ be an equivariant harmonic map, the flatness equations of $f^*(\nabla^c+\omega)$ decompose as: 
    \begin{equation}\label{HarmflateqSymmspace}
    \begin{dcases}
    F_{f^*\nabla^c}+[f^*\omega^{1,0},-f^*\tau(f^*\omega^{1,0})]=0~, \\
    (f^*\nabla^c)^{0,1}f^*\omega^{1,0}=0~.
\end{dcases}
 \end{equation} 
\end{Proposition}
\begin{Remark}\label{Remark 10 deriv harmmap Higgs}
    Let $\rho\in\Xx(\Gamma,\sG)$ and let $h_\rho:\widetilde\Sigma\to\sG/\sH$ be the corresponding $\rho$-equivariant harmonic metric from Theorem \ref{CorletteTheorem}. This data defines a $\sG$-Higgs bundle $(\Pp,\varphi)$ on $\Sigma$ as follows. 
    By equivariance, the holomorphic $\sH_\C$-bundle $h_\rho^*\sG[\sH_\C]$ on $\widetilde\Sigma$ descends to a holomorphic $\sH_\C$-bundle $\Pp_{\sH_\C}$ over $\Sigma.$ 
   Also, since $h_\rho$ is harmonic, the $(1,0)$-part of the pullback of the Maurer-Cartan form $h_\rho^*\omega^{1,0}$ is holomorphic and descends to a holomorphic section $\varphi\in H^0(\Sigma, \Pp_{\sH_\C}[\fm_\C]\otimes K).$
\end{Remark}

If $P_\sH\subset \Pp_{\sH_\C}$ is a reduction of structure group to the maximal compact subgroup $\sH,$ then the associated bundle $\Pp_{\sH_\C}[\fm_\C]=P_{\sH}[\fm_\C]$ decomposes as $P_\sH[\fm]\oplus P_\sH[i\fm].$ 
For such a reduction, the compact real form $\tau$ of $\fg_\C$ defines a conjugation 
\[\tau:\Omega^{1,0}(\Sigma, P_\sH[\fm_\C])\to\Omega^{0,1}(\Sigma,P_\sH[\fm_\C])~.\] 
Moreover, $-\tau(\psi)$ is the Hermitian adjoint of $\psi\in\Omega^{1,0}(\Sigma,P_\sH[\fm_\C])$ with respect to the metric induced by the Killing form on $P_\sH[\fm_\C].$
The following theorem was proven by Hitchin for $\sG = \sSL(2,\C)$ \cite{selfduality} and Simpson for $\sG$ complex semi-simple \cite{SimpsonVHS}. For the general statement below see \cite{HiggsPairsSTABILITY}.

\begin{Theorem}\label{HitchinEqTheorem}
    Let $(\Pp_{\sH_\C},\varphi)$ be a $\sG$-Higgs bundle, there exists a reduction of structure group $P_\sH\subset \Pp_{\sH_\C}$ to the maximal compact subgroup $\sH$, so that 
    \begin{equation}
        \label{Hitchinequation}F_{A}+[\varphi,-\tau(\varphi)]=0
    \end{equation}
    if and only if $(\Pp_{\sH_\C},\varphi) $ is poly-stable. Here $F_{A}$ denotes the curvature of the Chern connection of the reduction.
\end{Theorem}
Equation \eqref{Hitchinequation} is called the Hitchin equation. By definition of the Chern connection, $\nabla_A^{0,1}\varphi=0,$ thus, Hitchin's equation is the same as the decomposition of the pullback of the flatness equations \eqref{HarmflateqSymmspace} by an equivariant harmonic map. 
Given a solution to Hitchin's equation, the connection $A+\varphi-\tau(\varphi)$ is flat $\sG$-connection. 
Hence, for each Riemann surface structure $\Sigma$ on $S,$ Theorem \ref{CorletteTheorem} and Theorem \ref{HitchinEqTheorem} give a bijective correspondence between the moduli space of poly-stable $\sG$-Higgs bundles and the $\sG$-character variety of $\Gamma,$ $\Mm(\Sigma,\sG)\cong\Xx(\Gamma,\sG).$ 

\subsection{The Energy function and minimal surfaces}
Given $\rho\in\Xx(\Gamma,\sG)$ and a Riemann surface structure $\Sigma$ on $S$, let $h_\rho:\widetilde\Sigma\to\sG/\sH$ be the harmonic metric. If $\sG$ is a group of Hermitian type and $\rho$ is a maximal representation, then the centralizer of $\rho$ is compact \cite{BIWmaximalToledoAnnals}. 
Thus, the harmonic metric is unique for maximal representations.

\begin{Definition}\label{HopfDiferentialDEF}
    The {\em Hopf differential} of a harmonic map $f:\Sigma\to(N,g)$ is the holomorphic quadratic differential $q_f=(f^*g)^{(2,0)}\in H^0(\Sigma,K^2)$. 
\end{Definition}

The Hopf differential measures the failure of a map $f$ to be conformal. In particular, $q_f=0$ if and only if $f$ is a conformal immersion away from the singularities of $df.$ In this case, it is not hard to show that the rank of $df$ is either $0$ or $2$, and thus the only singularities of $df$ are branch points. This is equivalent to the image of $f$ being a branched minimal immersion \cite{MinImmofRiemannSurf,SchoenYauMinimalSurfEnergy}.

    \begin{Proposition}
Let $\sG$ be a real form of a reductive subgroup 
of $\sSL(n,\C)$ and consider $\rho\in\Xx(\Gamma,\sG)$. Fix a Riemann surface structure $\Sigma$ on $S$ and let $(\Pp_{\sH_\C},\varphi)$ be the $\sG$-Higgs bundle corresponding to $\rho$. The harmonic metric $h_\rho:\widetilde\Sigma\to\sG/\sH$ is a branched minimal immersion if and only if $\tr(\varphi^2)=0$. Moreover, $h_\rho$ is unbranched if and only if $\varphi$ is nowhere vanishing. 
    \end{Proposition}
    \begin{proof}
By Remark \ref{Remark 10 deriv harmmap Higgs} the Higgs field $\varphi$ is identified with the $(1,0)$-part of the derivative of the harmonic map. The metric on $\sG/\sH$ comes from the Killing form. Since $\sG$ is a real form of a subgroup of $\sSL(n,\C)$, the Hopf differential of a harmonic metric $h_\rho:\widetilde\Sigma\to\sG/\sH$ is a constant multiple of 
\[\tr(h_\rho^*\omega^{1,0}\otimes h_\rho^*\omega^{1,0})=\tr(\varphi^2)\in H^0(\Sigma,K^2)~.\]
Moreover, the branched minimal immersion is branch point free if and only if $h_\rho^*\omega^{1,0}=\varphi$ is nowhere vanishing. 
    \end{proof}

\begin{Remark} 
For a Lie group of Hermitian type, the Higgs field of a maximal Higgs bundle is nowhere vanishing, so the corresponding minimal immersions are always unbranched.  
\end{Remark}
For each representation $\rho\in\Xx(\Gamma,\sG)$ consider the {\em energy function} on Teichm\"uller space which gives the energy of the harmonic metric $h_\rho$
\begin{equation}
  \label{EnergyFunct}E_\rho:\xymatrix@R=0em{\Teich(S)\ar[r]&\R^{\geq0}\\\Sigma\ar@{|->}[r]&\haf\int\limits_\Sigma|dh_\rho|^2}~.
\end{equation} 
The critical points of $E_\rho$ are branched minimal immersions \cite{MinImmofRiemannSurf,SchoenYauMinimalSurfEnergy}. 
\begin{Remark}\label{Remark maximal are anosov}
If $\rho$ is an \emph{Anosov}\footnote{The definition of an Anosov representation is not necessary for our considerations, however we refer the reader to \cite{AnosovFlowsLabourie,GGKWAnosov,KLPAnosov1} for the appropriate definitions.} representation, then the energy function $E_\rho$ is smooth and proper \cite{CrossRatioAnosoveProperEnergy}. 
Thus, for each Anosov representation there {\em exists} a Riemann surface structure in which the harmonic metric is a branched minimal immersion. 
However, such a Riemann surface structure is not unique in general. Indeed, there are quasi-Fuchsian representations for which many such Riemann surfaces structures exist \cite{HuangWang15}.
\end{Remark}

For Hitchin representations, Labourie conjectured \cite{CrossRatioAnosoveProperEnergy} that the Riemann surface structure in which the harmonic metric is a branched minimal immersion is {\em unique}. Labourie's conjecture has been established for Hitchin representations into a rank two split Lie group \cite{cyclicSurfacesRank2}, but is open in general. 
Since maximal representations are examples of Anosov representations \cite{MaxRepsAnosov}, existence holds for all maximal representations. 
We now show that the branched minimal immersion is unique for all maximal $\sSO_0(2,3)$-representations.
\begin{Theorem}\label{UniqueMinSurface}
Let $\Gamma$ be the fundamental group of a closed oriented surface and
let $\Xx^\mathrm{max}(\Gamma,\sSO_0(2,3))$ be the character variety of maximal $\sSO_0(2,3)$-representations of $\Gamma$.
 For each $\rho\in\Xx^\mathrm{max}(\Gamma,\sSO_0(2,3))$ there is a unique Riemann surface structure $\Sigma$ in which the unique harmonic metric  $h_\rho:\widetilde\Sigma\to\sSO_0(2,3)/(\sSO(2)\times\sSO(3))$ is a minimal immersion with no branch points. 
\end{Theorem}
\begin{Remark} 
In \cite{MySp4Gothen}, Theorem \ref{UniqueMinSurface} was proven for representations in the connected components $\Xx^\mathrm{max}_{0,d}(\Gamma,\sSO_0(2,3))$ for $d\neq0$. Here we prove it for all of the other components of $\Xx^\mathrm{max}(\Gamma,\sSO_0(2,3)).$ This result has recently been extended to all maximal representations in any real rank two Lie group of Hermitian type in \cite{CollierTholozanToulisse} using very different methods.
\end{Remark}
It remains to prove Theorem \ref{UniqueMinSurface} for the components $\Xx_{0,0}^{\max}(\Gamma,\sSO_0(2,3))$ and $\Xx_{sw_1}^{max,sw_2}(\Gamma, \sSO_0(2,3)).$ 
We will first prove the statement for the smooth locus of $\Xx_{0,0}^{\max}(\Gamma,\sSO_0(2,3))$ by showing that the cyclic surface technology of \cite{cyclicSurfacesRank2} and \cite{MySp4Gothen} can be applied. 
For the non-smooth locus, we use our knowledge of the Zariski closure of such representations to establish uniqueness for all representations in $\Xx_{0,0}^{\max}(\Gamma,\sSO_0(2,3))$. 
Finally, for $\Xx^{\mathrm{max},sw_2}_{sw_1}(\Gamma,\sSO_0(2,3))$, after pulling back to the double cover $S_{sw_1}$ associated to $sw_1\in H^1(S,\Z_2)\setminus\{0\},$ we will use uniqueness for $\Xx_{0,0}^{\max}(\pi_1(S_{sw_1}),\sSO_0(2,3))$ to establish Theorem \ref{UniqueMinSurface} for the connected component $\Xx_{sw_1}^{\mathrm{max},sw_2}(\Gamma,\sSO_0(2,3))$. 

The proof of Theorem \ref{UniqueMinSurface} makes use of the following result.
\begin{Theorem}\label{DiffGeomTHM}
    (\cite[Theorem 8.1.1]{cyclicSurfacesRank2}) Let $\pi:P\to M$ be a smooth fiber bundle with connected fibers and $F : P\to\R$ be a positive smooth function. Define
\[N=\{x\in P\ |\  d_x(F|_{P_{\pi(x)}})=0\}\]
and assume for all $m\in M$ the function $F|_{P_m}$ is proper and that $N$ is connected and everywhere transverse to the fibers. Then $\pi$ is a diffeomorphism from $N$ onto $M$ and $F|_{P_m}$ has a unique critical point which is an absolute minimum.
\end{Theorem}

For $d>0$, let $P=\Teich(S)\times\Xx_{0,d}^{\max}(\sSO_0(2,3))$ and $\pi:P\to\Xx_{0,d}^{\max}(\sSO_0(2,3))$ denote the projection onto the second factor. Let $F:P\to\R$ be the function giving the energy of the associated harmonic metric $F(\Sigma,\rho)=\haf\int\limits_{\Sigma}|dh_\rho|^2.$

Recall that $d_{(\Sigma,\rho)}F_\rho=0$ if and only if the corresponding harmonic metric $h_\rho$ is a minimal immersion. This space is connected since it is homeomorphic to the product of $\Teich(S)$ with the vector bundle $\Ff_d$ from Theorem \ref{THM d>0}.
Thus, since $\Xx_{0,d}^{\max}(\sSO_0(2,3))$ is smooth for $d>0,$ to prove Theorem \ref{UniqueMinSurface} for $\Xx_{0,d}^{\max}(\sSO_0(2,3))$, it only remains to prove that the critical submanifold $N$ is everywhere transverse to the fibers of $P.$
This is done by associating a special `cyclic surface' to each minimal immersion associated to a point $(\Sigma,\rho)\in N$, and showing that there are no first order deformations of this cyclic surface which fix $\rho\in\Xx_{0,d}^{\max}(\Gamma,\sSO_0(2,3)).$  

\subsection{$\sSO_0(2,3)$-cyclic surfaces}
We now briefly recall the notion of a cyclic surface from \cite{cyclicSurfacesRank2} and the generalization of \cite{MySp4Gothen}. After this we will show how to associate a cyclic surface to a representation in the smooth locus of $\Xx_{0,0}^{\max}(\Gamma,\sSO_0(2,3)).$

Let $\fg=\fso(5,\C)$ denote the Lie algebra of $\sSO(5,\C)$, let $\fc\subset\fg$ be a Cartan subalgebra and $\Delta^+(\fg,\fc)\subset\Delta(\fg,\fc)$ denote a choice of positive roots. Recall that there are two simple roots $\{\alpha_1,\alpha_2\}$ and $\Delta^+=\{\alpha_1,\alpha_2,\alpha_1+\alpha_2,\alpha_1+2\alpha_2\}.$ 
The Lie algebra $\fg$ decomposes into root spaces: 
\begin{equation}\label{rootspacedecomp}
\fg_{-\alpha_1-2\alpha_2}\oplus\fg_{-\alpha_1-\alpha_2}\oplus\fg_{-\alpha_2}\oplus\fg_{-\alpha_1}\oplus\fc\oplus\fg_{\alpha_1}\oplus\fg_{\alpha_2}\oplus\fg_{\alpha_1+\alpha_2}\oplus\fg_{\alpha_1+2\alpha_2} ~.
\end{equation}
The following decomposition will be useful. 
\begin{equation}
    \label{4-grading}
    \xymatrix@R=0em{\fg_0=\fc~,&\fg_1=\fg_{-\alpha_1}\oplus\fg_{-\alpha_2}\oplus\fg_{\alpha_1+2\alpha_2}~,\\\fg_{2}=\fg_{-\alpha_1-\alpha_2}\oplus\fg_{\alpha_1+\alpha_2}~,&\fg_3=\fg_{\alpha_1}\oplus\fg_{\alpha_2}\oplus\fg_{-\alpha_1-2\alpha_2}}~.
\end{equation}
Fix a Cartan involution $\theta:\fg\to\fg$ which preserves $\fc$ and $\theta(\fg_\alpha)=\fg_{-\alpha}$ for all roots. 
Let $\ft\subset\fc$ be the fixed point locus of $\theta|_{\fc}$. Let $\sT\subset\sSO(5,\C)$ be the connected subgroup with Lie algebra $\ft,$ $\sT$ is a maximal compact torus. 
Since the root space splittings \eqref{rootspacedecomp} and \eqref{4-grading} are $Ad_T$-invariant, the homogeneous space $\sG/\sT$ is reductive, and the Maurer-Cartan form $\omega\in\Omega^1(\sG/\sT,\fg)$ decomposes as
\[\omega=\omega_{-\alpha_1-2\alpha_2}+\omega_{-\alpha_1-\alpha_2}+\omega_{-\alpha_1}+\omega_{-\alpha_2}+\omega_{i\ft}+\omega_{\alpha_1}+\omega_{\alpha_2}+\omega_{\alpha_1+\alpha_2}+\omega_{\alpha_1+2\alpha_2}\] and 
\[\omega=\omega_{i\ft}+\omega_{1}+\omega_{2}+\omega_{3}~.\]
\begin{Remark}
In \cite{cyclicSurfacesRank2} it is shown that homogeneous space $\sSO(5,\C)/\sT$ can be identified with the space of tuples  $(\Delta^+\subset\fc^*, \theta , \lambda)$ where
\begin{itemize}
    \item $\fc$ is a Cartan subalgebra
    \item $\Delta^+\subset\fc^*$ is a choice of positive roots
    \item $\theta$ is a Cartan involution which preserves $\fc$ and $\lambda$ is a split real form which commutes with $\theta$, and globally preserves $\fc$
\end{itemize}
     The involutions $\theta$ and $\lambda$ above also must satisfy certain compatibilities. Namely, both must globally preserve a principal three dimensional subalgebra $\fs$ which contains $x = \haf \sum\limits_{\alpha\in\Delta^+} H_\alpha$ and $\lambda(H_{\alpha_1+2\alpha_2}) = -H_{\alpha_1+2\alpha_2}$. 
     \end{Remark} 
 The trivial Lie algebra bundle $[\fg]\to\sSO(5,\C)/\sT$ admits two conjugate linear involutions $\Lambda$ and $\Theta$ given by 
 \[\xymatrix{\Lambda((\Delta^+\subset\fc^*, \theta , \lambda),v)=\lambda(v)&\text{and}&\Theta((\Delta^+\subset\fc^*, \theta , \lambda),v)=\theta(v)}~,\]
 where we have used the identification of a point in $\sSO(5,\C)/\sT$ with a tuple $(\Delta^+\subset\fc^*, \theta , \lambda)$ mentioned above. 
\begin{Definition}\label{cyclicsurfaceDEF}
    Let $\Sigma$ be a Riemann surface and let $\omega\in\Omega^1(\sG/\sT,\fg)$ be the Maurer-Cartan form of $\sG/\sT.$ A smooth map $f:\Sigma\to\sG/\sT$ is called a {\em cyclic surface} if  
    $f^*\omega_1$ is a $(1,0)$-form and
\[f^*\omega_2=f^*\omega_{i\ft}=f^*\omega+f^*\Theta(\omega)=f^*\omega-f^*\Lambda(\omega)=0~.\]
\end{Definition}

\subsection{Proof of Theorem \ref{UniqueMinSurface}}    
 For Higgs bundles in $\Mm^{\mathrm{max}}_0(\sSO_0(2,3)),$ there are two important reductions of structure group. The decomposition of the holomorphic bundle $\Ee$ as a direct sum of line bundles 
 \[\Ee=M\oplus K\oplus\ \Oo\oplus K^{-1}\oplus M^{-1}\] defines a reduction of structure group to the maximal complex torus $\sC$ of $\sSO(5,\C)$. On the other hand the metric solving the Hitchin equation gives a reduction of structure group to the maximal compact subgroup $\sSO(5)\subset\sSO(5,\C)$. 
 These two reductions of structure are compatible if and only if the holomorphic line bundle decomposition of $\Ee$ is orthogonal with respect the metric solving Hitchin's equation. 
 This is equivalent to having the following commuting diagrams of $\rho$-equivariant maps:
 \begin{equation}
 \label{compatible reductions} 
 \xymatrix{\sSO(5,\C)/\sC&\sSO(5,\C)/\sT\ar[l]\ar[d]\\\widetilde \Sigma\ar[r]_{h_\rho\ \ \ \ \ }\ar[u]^{H_\rho}\ar@{-->}[ur]^{f_\rho}&\sSO(5,\C)/\sSO(5)} ~.
 \end{equation}
\begin{Proposition}\label{cyclicsurfacesProp}
    If $\rho\in\Xx_{0,0}^{\max}(\Gamma,\sSO_0(2,3))$ is a smooth point, $\Sigma$ is a Riemann surface structure on $S$ and $h_\rho:\widetilde \Sigma\to\sSO_0(2,3)/\sSO(2)\times\sSO(3)\subset\sSO(5,\C)/\sSO(5)$ is the associate harmonic metric, then the holomorphic reduction $H_\rho$ from \eqref{compatible reductions} is compatible with the harmonic metric $h_\rho$ if and only if $h_\rho$ is a minimal immersion. 
    Moreover, the reduction $f_\rho$ from \eqref{compatible reductions} is an $\sSO_0(2,3)$-cyclic surface with $f_\rho^*\omega_{-\alpha_1}$ nonzero and $f_\rho^*\omega_{-\alpha_2}$ nowhere vanishing. 
\end{Proposition}
\begin{proof}
Let $\rho\in\Xx(\sSO_0(2,3))$ be a smooth point, and let $\Sigma$ be a Riemann surface structure on $S.$ Recall that the $\sSO(5,\C)$-Higgs bundle $(\Ee,Q,\Phi)$ corresponding to $\rho$ is given by 
\[\left( M\oplus K\oplus\Oo\oplus K^{-1}\oplus M^{-1},\ \mtrx{&&&&-1\\&&&1&\\&&-1&&\\&1&&&\\-1&&&&},\ \mtrx{0&0&0&\nu&0\\\mu&0&q_2&0&\nu\\0&1&0&q_2&0\\0&0&1&0&0\\0&0&0&\mu&0}\right)\]
with $\mu$ and $\nu$ both nonzero. Moreover, the harmonic metric $h_\rho$ is a minimal immersion if and only if $q_2=0.$ 
In this case, the diagonal $\sSO(2,\C)\times\sSO(3,\C)$-gauge transformation $g=diag(-1,-i,1,i,-1)$ acts as $Ad_g\Phi=i\Phi.$ 
The gauge transformation $g$ is therefore preserved by the metric connection $A$ solving Hitchin's equation.
Thus, the eigen-bundle splitting of $g$ is orthogonal with respect to the metric solving Hitchin's equations, and the holomorphic line bundle splitting of $\Ee$ is compatible with the metric reduction. 

Let $H=diag(h_1,h_2,1,h_2^{-1},h_1^{-1})$ 
be the metric solving Hitchin equations and $f_\rho:\widetilde \Sigma\to\sSO(5,\C)/\sT$ be the equivariant map of the associated reduction of structure group. 
The pullback of the complexification of the Maurer-Cartan form of $\sSO(5,\C)/\sT$ is given by $f_\rho^*\omega=\Phi+\Phi^*,$ with $f_\rho^*\omega_1=\Phi$ and $f_\rho^*\omega_3=\Phi^*.$ 
Since $\rho$ is an $\sSO_0(2,3)$-representation and $f_\rho$ lifts the metric solving Hitchin's equations, we have
\[f_\rho^*\omega-f_\rho^*\Lambda(\omega)=(\Phi+\Phi^*)-(\Phi+\Phi^*)=0\]and\[f_\rho^*\omega+f_\rho^*\Theta(\omega)=(\Phi+\Phi^*)+(-\Phi^*-\Phi)=0~.\]
Thus, $f_\rho$ is a $\rho$-equivariant cyclic surface. 
Finally, $f_\rho^*\omega_{-\alpha_1}=\mu$ and $f^*_\rho\omega_{-\alpha_2}=1$.
\end{proof}

\begin{Remark}\label{transversalityRemark}
      Let $\rho_t:\Gamma\to\sSO_0(2,3)$ be a $1$-parameter family of representations in $\Xx_{0,0}^{\max}(\Gamma,\sSO_0(2,3))$ and $f_{\rho_t}:\Sigma_{\rho_t}\to\sSO_0(2,3)$ be a 1-parameter family of cyclic surfaces with $f_{\rho_0}^*\omega_{-\alpha_1}$ nonzero and $f^*_{\rho_0}\omega_{-\alpha_2}$ nowhere vanishing. 
      By Theorem 7.5 of \cite{MySp4Gothen}, if $\frac{d}{dt}\Sigma_{\rho_t}|_{_{t=0}}$ is nonzero in $T_{\Sigma_{\rho_0}}\Teich(S)$, then $[\frac{d}{dt}|_{_{t=0}}\rho_t]\neq 0\in T_\rho\Xx_{0,0}^{\max}(\Gamma,\sSO_0(2,3))$. 
\end{Remark}

\begin{proof}[Proof of Theorem \ref{UniqueMinSurface}]
Let $\Xx_{0,0}^\mathrm{max}(\Gamma,\sSO_0(2,3))^{sm}$ denote the smooth locus of the connected component $\Xx_{0,0}^\mathrm{max}(\Gamma,\sSO_0(2,3)).$ 
Set \[P=\Teich(S)\times\Xx_{0,0}^\mathrm{max}(\Gamma,\sSO_0(2,3))^{sm}\] and let $\pi:P\to\Xx_{0,0}^\mathrm{max}(\Gamma,\sSO_0(2,3))^{sm}$ denote the projection onto the second factor. 
Define the energy function 
\[F:\Teich(S)\times\Xx_{0,0}^\mathrm{max}(\Gamma,\sSO_0(2,3))^{sm}\to\R\] as in \eqref{EnergyFunct}.
Recall that the restriction $F_\rho$ to the fibers of $P$ is smooth and proper \cite{CrossRatioAnosoveProperEnergy} and that the critical points of $F_\rho$ are the minimal immersions we seek. Set 
\[
    N=\{(\Sigma,\rho)\in P\ | \ d_{(\Sigma,\rho)}(F_\rho)=0 \}~.
\]

Since $F_\rho$ is proper, for each $\rho\in\Xx_{0,0}^\mathrm{max}(\Gamma,\sSO_0(2,3))$ there exist a Riemann surface $\Sigma$ in which the harmonic metric $h_\rho:\widetilde\Sigma\to\sSO_0(2,3)/(\sSO(2)\times\sSO(3))$ is a minimal immersion. 
Moreover, by Proposition \ref{cyclicsurfacesProp}, for each such pair $(\Sigma,\rho)$ there is an associated $\sSO_0(2,3)$-cyclic surface $f_\rho:\widetilde\Sigma\to\sSO_0(2,3)/\sT$ which lifts the harmonic metric. By Remark \ref{transversalityRemark}, if $v\in T_{(\Sigma,\rho)}N$, then $d\pi(v)\neq0,$ and hence $N$ is transverse to the fiber of $P$ at $(\Sigma,\rho)$. Finally, $N$ is homeomorphic to the product of $\Teich(S)$ with the smooth locus of $\Aa/\Z_2$ from Theorem \ref{thm:zero_component}. Since the smooth locus of $\Aa/\Z_2$ is connected, the space $N$ is also connected. 
By Theorem \ref{DiffGeomTHM}, for all $\rho\in\Xx_{0,0}^\mathrm{max}(\Gamma,\sSO_0(2,3))^{sm}$ there is a unique Riemann surface structure $\Sigma$ on $S$ in which the harmonic metric is a minimal immersion. 

If $\rho\in\Xx_{0,0}^\mathrm{max}(\Gamma,\sSO_0(2,3))$ is not a smooth point, then, by Proposition \ref{Prop Higgs reductions M0d}, $\rho$ factors through the product of either a maximal $\sSO_0(1,2)$ representations or a maximal $\sSO_0(2,2)$ representation with a compact group. Since uniqueness is known for maximal $\sSO_0(1,2)$ \cite{TeichOfHarmonic} and $\sSO_0(2,2)$ representations \cite{SL2xSL2uniqueMinSurface}, we are done. 

Now suppose $\rho\in\Xx_{sw_1}^{\mathrm{max},sw_2}(\Gamma,\sSO_0(2,3))$ and let $\pi: S_{sw_1}\to S$ be the double cover associated to $sw_1\in H^1(S,\Z_2)$. 
Recall from \eqref{pullbackHiggsBundles} that the representation $\pi_*\circ\rho:\pi_1( S_{sw_1})\to\sSO_0(2,3)$ is maximal and lies in $\Xx_{0,0}^{\mathrm{max}}(\pi_1( S_{sw_1}),\sSO_0(2,3)).$ 
Since there is a unique Riemann surface structure $\Sigma_{sw_1}$ on $S_{sw_1}$ in which the $\pi_*\rho$-equivariant harmonic metric $h_{\pi_*\rho}$
is a minimal immersion and $\pi^*:\Teich(S)\to\Teich(S_{sw_1})$ is injective, we conclude that, for each maximal representation $\rho\in\Xx_{sw_1}^{\mathrm{max},sw_2}(\Gamma,\sSO_0(2,3))$, there is a unique Riemann surface structure $\Sigma_\rho$ in which the harmonic metric $h_\rho:\widetilde\Sigma_\rho\to\sSO_0(2,3)/\sSO(2)\times\sSO(3)$ is a minimal immersion. 
\end{proof}

\section{Mapping class group invariant complex structure}     \label{sec:mcg_inv_cmplx_str}

In this section we show that the space of maximal representations into a real rank two Lie group of Hermitian type admits a mapping class group invariant complex structure. More generally, we prove the following theorem.

\begin{Theorem}\label{thm:invariant complex structures}
 Let $\sG$ be a semi-simple algebraic Lie group of Hermitian type,
and let $C \subset \Xx^{\mathrm{max}}(\Gamma,\sG)$ be an open $\MCG(S)$-invariant subset of maximal representations $\rho$ admitting a unique $\rho$-equivariant minimal surface in the symmetric space of $\sG$ (i.e. $C$ is an open set of maximal representations where Labourie's conjecture holds).
The space $C$ admits the structure of a complex analytic space such that $\MCG(S)$ acts on $C$ by holomorphic maps and such that the natural map $C\ra \Teich(S)$ given by the minimal surface is holomorphic. 
\end{Theorem}

\begin{Remark}
 The analog of Theorem \ref{thm:invariant complex structures} holds when $\sG$ is a split semi-simple Lie group and $C$ is the union of the Hitchin components on which Labourie's conjecture holds. 
 In this case, the proof is straight forward since for every $n$ there exists a holomorphic vector bundle $\Hh^n$ over Teichm\"uller space $\Teich(S)$ whose fiber at each point $\Sigma \in \Teich(S)$ is naturally identified with the vector space $H^0(\Sigma,K^n)$ of holomorphic $n$-differentials on $\Sigma$.
 
 As Labourie's conjecture has been established for the Hitchin components of the split semi-simple Lie groups of rank $2$, the Hitchin components for such groups admit the structure of a complex manifold which $\MCG(S)$ acts on holomorphically.
    \end{Remark}

The proof of Theorem \ref{thm:invariant complex structures} requires more work than its analog for the Hitchin component because the space of maximal representations has non-trivial topology. 
Moreover, the presence of singularities leads the technical complication that the components will not in general be complex manifolds but only complex analytic spaces. 
For these reasons, the proof we give uses only general principles, and thus avoids dealing with what the space of maximal representations looks like.

In Theorem \ref{UniqueMinSurface}, Labourie's conjecture was proven for maximal representations into $\sP\sSp(4,\R)$ and $\sSp(4,\R)$. As a corollary of Theorems \ref{UniqueMinSurface} and \ref{thm:invariant complex structures} we have:

\begin{Corollary}\label{COr: mapping class group invariant structure PSp4}
The spaces $\Xx^{\mathrm{max}}(\Gamma,\sP\sSp(4,\R))$ and $\Xx^{\mathrm{max}}(\Gamma, \sSp(4,\R))$ admit the structure of a complex analytic space on which $\MCG(S)$ acts by holomorphic maps. 
\end{Corollary}

In \cite{CollierTholozanToulisse}, Theorem \ref{UniqueMinSurface} has been extended to all maximal representations into any rank two semi-simple Lie group of Hermitian type. Thus, applying Theorem \ref{thm:invariant complex structures} we have the following extension of Corollary \ref{COr: mapping class group invariant structure PSp4}:

\begin{Corollary}
 Let $\Gamma$ be the fundamental group of a closed surface of genus $g\geq2.$ The space of maximal representations of $\Gamma$ into a rank two semi-simple Lie group of Hermitian type admits the structure of a complex analytic space on which $\MCG(S)$ acts by holomorphic maps, and such that the natural map to $\Teich(S)$ given by the minimal surface is holomorphic. 
\end{Corollary}

Recall that character varieties also carry a mapping class group invariant symplectic structure, usually called the \emph{Goldman symplectic form}. A natural question is whether or not this symplectic structure is compatible with the complex structure from Theorem \ref{thm:invariant complex structures} since it would then define a mapping class group invariant K\"ahler structure. This is known only for the $\sP\sSL(2,\R)$-Hitchin components, i.e. the Teichm\"uller space of $S$.  

The proof of Theorem \ref{thm:invariant complex structures} will be based on the following theorem.
\begin{Theorem}\label{thm:universal}
Given a complex reductive algebraic Lie group $\sG_\C$, there is a complex analytic space $\Mm(\Uu,\sG_\C)$ with a holomorphic map $\pi:\Mm(\Uu,\sG_\C) \to \Teich(S)$ such that 
\begin{enumerate}
\item for every $\Sigma \in \Teich(S)$, $\pi^{-1}(\Sigma)$ is biholomorphic to $\Mm(\Sigma,\sG_\C)$,
 \item $\pi$ is a trivial topological fiber bundle.
\item the pullback operation on Higgs bundles gives a natural action of $\MCG(S)$ on $\Mm(\Uu,\sG_\C)$ by holomorphic maps that lifts the action on $\Teich(S)$.
\end{enumerate}
\end{Theorem}

We will call the complex analytic space $\Mm(\Uu,\sG_\C)$ the \emph{universal moduli space of Higgs bundles}; it contains all the Higgs bundles with reference to all the possible complex structures on $S$. The proof of Theorem \ref{thm:universal} relies on Simpson's construction of the moduli space of Higgs bundles for Riemann surfaces over schemes of finite type over $\C$. More specifically we will use \cite[Corollary 6.7]{SimpsonModuli2}. We note that $\Teich(S)$ is not a scheme of finite type over $\C$.

\subsection{Complex analytic spaces}  \label{sec:complex analytic}

In this section we recall the definition of a complex analytic space. This is a necessary framework to discuss about complex structures for singular spaces.

A \emph{complex analytic variety} is a subset $V\subset \C^n$ such that for every point $z\in V$ there exists an open neighborhood $U$ of $z$ in $\C^n$ and a finite family $f_1, \dots, f_k \in \Oo(U)$ of holomorphic functions on $U$ such that
$$V \cap U = \{ x\in U \ |\ f_1(x)=\dots=f_k(x)=0 \}~. $$
Note that the set $V$ does not need to be closed in $\C^n$, but it is always locally closed (a closed subset of an open subset of $\C^n$). For example, every open subset of a complex analytic variety is a complex analytic variety.

For a subset $U\subset V$, a function $f:U\ra\C$ is \emph{holomorphic} if there exists an open neighborhood $U'$ of $U$ in $\C^n$ and a holomorphic function $f':U' \ra \C$ such that $f = f'|_U$. We will denote by $\Oo_V(U)$ the $\C$-algebra of holomorphic functions on $U$. These $\C$-algebras form a sheaf $\Oo_V$ called the \emph{sheaf of holomorphic functions} on $V$. The pair $(V,\Oo_V)$ is a \emph{locally ringed space}, i.e., a space with a sheaf of $\C$-algebras where every stalk has a unique maximal ideal.

Similarly, if $V \subset \C^n$ and $W\subset \C^m$ are two complex analytic varieties, a map $f:V \ra W$ is \emph{holomorphic} if, for every $z\in V$, there exists an open neighborhood $U$ of $z$ in $\C^n$ and a holomorphic map $f':U \ra \C^m$ such that $f|_{U\cap  V} = f'|_{U\cap  V}$.

A complex analytic space is a locally ringed space that is locally isomorphic to a complex analytic variety. More precisely, we have the following definition.

\begin{Definition}    \label{Def: Complex analytic space}
A \emph{complex analytic space} is a locally ringed space $(X,\Oo_X)$ where for every $x\in X$ there exists an open neighborhood $U$ of $x$ in $X$ and  a complex analytic variety $V\subset\C^n$ such that the sheaves $(U,\Oo_X|_U)$  and $(V,\Oo_V)$ are isomorphic.
\end{Definition}
\begin{Definition}
If $X$ and $Y$ are complex analytic spaces, a map $f:X \ra Y$ is \emph{holomorphic} if, for every $z\in X$, there exist open neighborhoods $U$ of $z$ and $U'$ of $f(z)$ such that $f(U) \subset U'$ and both $(U,\Oo_V|_U)$ and $(U',\Oo_W|_{U'})$ are complex analytic varieties such that $f$ defines a holomorphic map between them. 
\end{Definition}

\subsection{Universal Teichm\"uller curve} \label{sec:universal curve}

The Teichm\"uller space $\Teich(S)$ has a natural complex structure that turns it into a complex manifold of complex dimension $3g-3$. This structure was defined by Teichm\"uller as the unique complex structure for which $\Teich(S)$ is a \emph{fine moduli space}. This means that there is a \emph{universal family} $f:\Uu \ra \Teich(S)$, where 
\begin{itemize}
    \item $\Uu$ is a complex manifold of dimension $3g-2$, 
    \item $f$ is a holomorphic function that is also a trivial smooth fiber bundle, 
    \item for every $\Sigma \in\Teich(S)$, $f^{-1}(\Sigma)$ is a submanifold of $\Uu$ isomorphic to $\Sigma$.
\end{itemize}
 The universal family is usually called the \emph{universal Teichm\"uller curve}. For more details and an historical account, see the survey paper \cite{AthanaseSurvey2014}.

\begin{Proposition}
There is a unique action of the mapping class group on $\Uu$ which lifts the action on $\Teich(S)$. This action is properly discontinuous but not free.
\end{Proposition} 
\begin{proof}
For every mapping class $\phi \in \MCG(S)$, consider the diagram
\begin{equation}   \label{EQ Universal diagram}
 \begin{CD}
 \Uu        @.             \Uu\\
 @V{f}VV                  @V{f}VV\\
 \Teich(S)  @>{\phi}>>    \Teich(S)~.
 \end{CD}
\end{equation}
The map $\phi\circ f$ has all the properties of the universal Teichm\"uller curve $f$, so, by uniqueness of the universal Teichm\"uller curve, there exists a unique biholomorphism $\phi^\Uu: \Uu \to \Uu$ which makes \eqref{EQ Universal diagram} commute. 
This defines an action of $\MCG(S)$ because if $\phi$ and $\psi$ are two mapping classes, then $\phi^\Uu \circ \psi^\Uu$ and $(\phi\circ\psi)^\Uu$ agree since they both make \eqref{EQ Universal diagram} commute.

The map $f$ is now $\MCG(S)$-equivariant and proper, hence the action of $\MCG(S)$ on $\Uu$ is properly discontinuous since $\MCG(S)$ acts properly discontinuously on $\Teich(S)$. The action is not free since, for every $\Sigma \in \Teich(S)$ which is fixed by an element of $\MCG(S)$, there exists a point of $\Sigma$ which is fixed.
\end{proof}

We denote by $\Mod(S) = \Teich(S)/\MCG(S)$ the \emph{moduli space of Riemann surfaces}. Following \cite{DeligneMumford69}, we call the space $\Uu/\MCG(S)$ the \emph{moduli space of pointed Riemann surfaces}. 
It is a complex analytic space which has a natural holomorphic projection $\Uu/\MCG(S) \to \Mod(S)$. However, it does not have the nice properties of the map $\Uu \ra \Teich(S)$, since for a $\Sigma \in \Mod(S)$ which is fixed by some element of $\MCG(S)$, the fiber in $\Uu/\MCG(S)$ is not isomorphic to $\Sigma$ but to the quotient of $\Sigma$ by its stabilizer in $\MCG(S)$. This corresponds to the fact that $\Mod(S)$ is not a fine moduli space.

For each subgroup $\sH < \MCG(S)$, we have quotient spaces:
\[\Uu/\sH \to \Teich(S)/\sH~.\] 
The space $\Teich(S)/\sH$ can often be interpreted as the moduli spaces of Riemann surfaces with some kind of partial marking, and $\Uu/\sH$ is the pointed version. Especially interesting subgroups $\sH$ arise from the construction in the following example.
\begin{Example}\label{EXample Teichm}
Consider the action of $\MCG(S)$ on the cohomology group $H^1(S,\Z_m)$ with coefficients in the cyclic group $\Z_m$. This defines a representation 
\[h_m: \xymatrix{\MCG(S) \ar[r]& \sSp(2g,\Z_m)}~.\]
Denote the kernel of $h_m$ by $\sH_m$ and set
\begin{equation}
    \label{EQ teichm}\Teich_m(S) = \Teich(S)/\sH_m~.
\end{equation}
The space $\Teich_m(S)$ can be interpreted as the moduli space of pairs $(\Sigma,\alpha)$, where $\Sigma$ is an abstract Riemann surface homeomorphic to $S$ and $\alpha:\pi_1(\Sigma)\to \Z_m^{2g}$ is a surjective group homomorphism (a partial marking). Similarly, $\Uu_m = \Uu/H_m$ is a pointed version of this moduli space.
\end{Example}

The group $\sH_m\leq\MCG(S)$ from Example \ref{EXample Teichm} has finite index, furthermore, when $m\geq3,$ $\sH_m$ is torsion-free (see \cite[Thm 6.9]{PrimerMCG}). This gives the following proposition.
\begin{Proposition}
For each integer $m\geq 3$, the spaces $\Teich_m(S)$ and $\Uu_m$ are complex manifolds. Moreover, $\Teich(S) \to \Teich_m(S)$ and $\Uu \to \Uu_m$ are holomorphic coverings and $\Teich_m(S) \to \Mod(S)$ and $\Uu_m \to \Uu/\MCG(S)$ are  finite branched covering.
 \end{Proposition}

\subsection{Scheme structures}

Fix an integer $m\geq 3$. Some of the spaces introduced in Section \ref{sec:universal curve}, for instance, the moduli space $\Mod(S)$ and the pointed moduli space $\Uu/\MCG(S)$, are the analytification of quasi-projective algebraic schemes of finite type over $\C$. 
 Similarly, the spaces $\Teich_m(S)$ and $\Uu_m$ are also the analytification of smooth quasi-projective algebraic schemes of finite type over $\C$ (see \cite{DeligneMumford69}). 
Denote the corresponding quasi-projective schemes by 
\[\xymatrix{\Mod^{qp}(S)~,&(\Uu/\MCG(S))^{qp}~,&\Teich_m^{qp}(S)&\text{and}&\Uu_m^{qp}}~.\]
The Teichm\"uller space $\Teich(S)$ is {\em not} the analytification of an algebraic scheme. However, since we have a covering map $\Teich(S) \to \Teich_m(S)$, the Teichm\"uller space is locally biholomorphic to a complex manifold with this property. 

So far, we have only considered Higgs bundles on Riemann surfaces over $\C.$ However, for a complex algebraic reductive Lie group $\sG_\C$, Simpson \cite{SimpsonModuli1,SimpsonModuli2} constructed the moduli space $\Mm(C,\sG_\C)$ of $\sG_\C$-Higgs bundles on every scheme $C$ which is smooth and projective over a scheme of {\em finite type over $\C $}. Moreover, he showed that this moduli space is a quasi-projective scheme.
The scheme morphism $f_m^{qp}:\Uu_m^{qp} \ra \Teich_m^{qp}(S)$ is smooth and projective (see \cite{DeligneMumford69}), hence, applying Simpson's construction to $\Uu_m^{qp}$, we obtain the moduli space $\Mm(\Uu_m^{qp},\sG_\C)$ of $\sG_\C$-Higgs bundles on $\Uu_m^{qp}$. We summarize the above discussion in the following proposition, see \cite[Corollary 6.7]{SimpsonModuli2} for more details.
\begin{Proposition}
    The moduli space of $\sG_\C$-Higgs bundles on the Riemann surface $\Uu_m^{qp}$ defines a quasi-projective scheme $\Mm^{qp}(\Uu_m^{qp},\sG_\C)$ over $\Teich_m^{qp}(S)$ with the property that for every geometric point $(\Sigma,\alpha) \in \Teich_m^{qp}(S)$, the fiber of $(\Sigma,\alpha)$ in $\Mm^{qp}(\Uu_m^{qp},\sG_\C)$ is the moduli space $\Mm^{qp}((f_m^{qp})^{-1}(\Sigma,\alpha),\sG_\C) = \Mm^{qp}(\Sigma,\sG_\C)$.
\end{Proposition}

The analytifications $\Mm(\Uu_m,\sG_\C)$ and $\Mm(\Sigma,\sG_\C)$ are complex analytic spaces with the following properties.
\begin{Proposition}  \label{Prop:m-universal}
     There is a holomorphic map 
\[\Mm(\Uu_m,G_\C) \ra \Teich_m(S)\]
such that for every $(\Sigma,\alpha) \in \Teich_m(S)$, the fiber over  $(\Sigma,\alpha)$ in $\Mm(\Uu_m,\sG_\C)$ is biholomorphic to $\Mm(\Sigma,\sG_\C)$.
 \end{Proposition}  

 \subsection{Proof of Theorems \ref{thm:invariant complex structures} and \ref{thm:universal}}
We are now ready to prove Theorem \ref{thm:universal} which asserts that there is a complex analytic space $\Mm(\Uu,\sG_\C)$ with a holomorphic map $\pi:\Mm(\Uu,\sG_\C) \ra \Teich(S)$ such that 
for every $\Sigma \in \Teich(S)$, $\pi^{-1}(\Sigma)$ is biholomorphic to $\Mm(\Sigma,\sG_\C)$ and there is a unique lift of the action of the mapping class group on $\Teich(S)$ to $\Mm(\Uu,\sG_\C)$.

\begin{proof}[Proof of Theorem \ref{thm:universal}]
Fix an integer $m\geq 3$, and consider the covering $\Teich(S) \to \Teich_m(S)$ from \eqref{EQ teichm}. Let $\Mm(\Uu,\sG_\C)$ denote the topological space that is the fiber product of the maps $\Mm(\Uu_m,\sG_\C) \ra \Teich_m(S)$ and  $\Teich(S) \to \Teich_m(S)$. 

Since the map $\Teich(S) \to \Teich_m(S)$ is a local biholomorphism, the natural map  $\Mm(\Uu,\sG_\C) \to \Mm(\Uu_m,\sG_\C)$ is also locally invertible, so $\Mm(\Uu,\sG_\C)$ inherits a structure of complex analytic space from $\Mm(\Uu_m,\sG_\C)$. With respect to this structure, the map  
\[\pi:\xymatrix{\Mm(\Uu,\sG_\C) \ar[r]& \Teich(S)}\] is holomorphic.
Moreover, by Proposition \ref{Prop:m-universal}, for every $\Sigma \in \Teich(S)$, $\pi^{-1}(\Sigma)$ is biholomorphic to $\Mm(\Sigma,\sG_\C)$.

Consider the map 
\[\Pi:\xymatrix@R=.2em{\Mm(\Uu,\sG_\C)  \ar[r]&  \Xx(\Gamma,\sG_\C) \times \Teich(S)\\  E\ar@{|->}[r] &(N_{\pi(E)}(E), \pi(E))}~,\]
where $N_\Sigma: \Mm(\Sigma,\sG_\C) \ra \Xx(\Gamma,\sG_\C)$ is the homeomorphism given by Theorem \ref{THM: NAHC}. The map $\Pi$ is a homeomorphism which commutes with the projections to $\Teich(S)$, thus, $\pi:\Mm(\Uu,\sG_\C) \ra \Teich(S)$ is a trivial topological fiber bundle.  

The group $\MCG(S)$ acts on $\Mm(\Uu,\sG_\C)$ by pullback: let $\phi \in \MCG(S)$, and $E \in \Mm(\Uu,\sG_\C)$.  If $\pi(E) = \Sigma$, there is a unique holomorphic map $\hat{\varphi}:\phi(\Sigma) \to \Sigma$ that is homotopic to the identity. We define $\phi(E)$ to be the Higgs bundle $\hat\varphi^* E$ over $\phi(\Sigma)$. This gives the lift of the action of $\MCG(S)$ to $\Mm(\Uu,\sG_\C)$. 

Let $\sH_m\leq \MCG(S)$ be the subgroup whose quotient gives $\Teich_m(S)$. For $\phi \in \sH_m,$ the fact that the action of $\phi$ on $\Mm(\Uu,\sG_\C)$ is holomorphic follows from the construction of $\Mm(\Uu,\sG_\C)$ as a fiber product. 
If $\phi \not\in H_m$, then $\phi$ acts non-trivially on $\Teich_m^{qp}(S)$ and $\Uu_m^{qp}$: 
\[   
 \begin{CD}
 \Uu_m^{qp}        @>{\phi^{\Uu_m}}>>      \Uu_m^{qp}\\
 @V{f_m^{qp}}VV                  @V{f_m^{qp}}VV\\
 \Teich_m^{qp}(S)  @>{\phi}>>    \Teich_m^{qp}(S)~.
 \end{CD}
\]
The map $\phi\circ f_m^{qp} :\Uu_m^{qp} \to \Teich_m^{qp}(S)$ can also be used to define a moduli space of Higgs bundles $\Mm^{qp}(\Uu_m^{qp,\phi},\sG_\C)$. 
Using the fiber product as above, we can construct a complex analytic space that we denote by $\Mm(\Uu^\phi,\sG_\C)$. Pulling back by $\phi^{-1}$ gives a holomorphic map
\begin{equation} \label{eq:action universal 1}
\Mm(\Uu,\sG_\C) \to \Mm(\Uu^\phi,\sG_\C) ~.
\end{equation}
Note that the map $\phi^{\Uu_m}$ defines an isomorphism between the schemes 
\[\xymatrix{\phi\circ f_m^{qp} :\Uu_m^{qp} \to \Teich_m^{qp}(S)&\text{and}& f_m^{qp} :\Uu_m^{qp} \to \Teich_m^{qp}(S)}~.\] 
As Simpson's construction is functorial, the above isomorphisms induce an isomorphism between $\Mm^{qp}(\Uu_m^{qp,\phi},\sG_\C)$ and $\Mm^{qp}(\Uu_m^{qp},\sG_\C)$. Finally, this isomorphism induces a biholomorphism 
\begin{equation} \label{eq:action universal 2}
\Mm(\Uu^\phi,\sG_\C) \cong \Mm(\Uu,\sG_\C)~. 
\end{equation}

The action of $\phi$ on $\Mm(\Uu,\sG_\C)$ is given by the composition of the holomorphic maps in \eqref{eq:action universal 1} and \eqref{eq:action universal 2}, hence it is holomorphic.
\end{proof}

The following corollary is important for our application of Theorem \ref{thm:universal} to maximal representations. 

\begin{Corollary}
Let $\sG$ be a semi-simple algebraic Lie group of Hermitian type. 
For each integer $\tau \in \Z \setminus \{0\}$, there is a complex analytic space $\Mm^\tau(\Uu,\sG)$ with a holomorphic map $\pi:\Mm^\tau(\Uu,\sG) \ra \Teich(S)$ such that 
\begin{enumerate}
\item for each $\Sigma \in \Teich(S)$, the fiber $\pi^{-1}(\Sigma)$ is biholomorphic to the moduli space $\Mm^\tau(\Sigma,\sG)$ of $\sG$-Higgs bundles with Toledo invariant equal to $\tau$,
\item $\pi$ is a trivial topological fiber bundle.
\item the pullback operation on Higgs bundles gives a natural action of $\MCG(S)$ on $\Mm^\tau(\Uu,\sG_\C)$ by holomorphic maps that lifts the action on $\Teich(S)$.
\end{enumerate}
\end{Corollary}
\begin{proof}
Let $\sG_\C$ be the complexification of $\sG$. If $\sH$ is the maximal compact subgroup of $\sG$, let $\sH_\C < \sG_\C$ be its complexification. 
The group $\sH_\C$ is the fixed point set of a holomorphic involution $\sigma:\sG_\C\ra \sG_\C$. The involution $\sigma$ induces involutions
\[\xymatrix{\sigma: \Mm(\Sigma,\sG_\C) \to \Mm(\Sigma,\sG_\C)&\text{and}&\sigma:\Mm^{qp}(\Uu_m^{qp},\sG_\C) \ra \Mm^{qp}(\Uu_m^{qp},\sG_\C)}~.\]
 The involution on $\Mm^{qp}(\Uu_m^{qp},\sG_\C)$ induces a holomorphic involution on the universal moduli space of Higgs bundles $\sigma: \Mm(\Uu,\sG_\C) \ra \Mm(\Uu,\sG_\C)$. 

Let $\Mm(\Sigma,\sG_\C)^\sigma$ and $\Mm(\Uu,\sG_\C)^\sigma$ be the fixed point sets of the involutions $\sigma$, they are complex analytic subsets. For each $\tau \neq 0$, the map $\Mm^\tau(\Sigma,\sG) \ra \Mm(\Sigma,\sG_\C)^\sigma$ induced by the inclusion $\sG \ra \sG_\C$ is injective since we restrict our attention to a non-zero value of $\tau$. Moreover, its image is a union of connected components of $\Mm(\Sigma,\sG_\C)^\sigma$. 
If $\Mm^\tau(\Uu,\sG)$ is the union of the connected components of $\Mm(\Uu,\sG_\C)^\sigma$ that contain the images of $\Mm^\tau(\Sigma,\sG)$, then $\Mm^\tau(\Uu,\sG)$ is a complex analytic space has all the properties required by the theorem.
\end{proof}

We are now ready to prove Theorem \ref{thm:invariant complex structures} which asserts that if $C$ is an open $\MCG(S)$-invariant subset of maximal representations on which Labourie's conjecture holds, then $C$ admits a complex analytic structure such that the mapping class group acts on $C$ by holomorphic maps. 

\begin{proof}[Proof of Theorem \ref{thm:invariant complex structures}]

Let $\Mm^\mathrm{max}(\Uu,\sG)$ denote the space $\Mm^\tau(\Uu,\sG)$ for the maximal value that $\tau$ can assume for the group $\sG$. Let's consider the map
\[P: \xymatrix@R=.1em{\Mm^\mathrm{max}(\Uu,\sG)\ar[r]&  \Xx^\mathrm{max}(\Gamma,\sG)\\ E\ar@{|->}[r]&N_{\pi(E)}(E)}~, \]
where $N_\Sigma: \Mm^\mathrm{max}(\Sigma,\sG) \to \Xx(\Gamma,\sG)$ is the homeomorphism from Theorem \ref{THM: NAHC}, and $\pi:\Mm^\mathrm{max}(\Uu,\sG) \to \Teich(S)$ is the natural projection.
The map $P$ is surjective, but not injective since $P^{-1}(\Sigma)\cong\Teich(S)$. We will write 
\[C^\Uu = P^{-1}(C) \subset \Mm^\mathrm{max}(\Uu,\sG)~; \] 
$C^\Uu$ is open since $C$ is open. 

Recall that holomorphic tangent bundle of Teichm\"uller space is the bundle $\Hh^2$ whose fiber over a $\Sigma \in \Teich(S)$ is naturally identified with the vector space $H^0(\Sigma,K^2)$ of holomorphic differentials. Consider the holomorphic function
\begin{equation}
\label{EQ Tr_2}\mathrm{Tr}_2:\xymatrix@R=.2em{C^\Uu\ar[r]&\Hh^2\\(\Sigma,\Ee,\varphi) \ar@{|->}[r]&\tr(\varphi^2)~.}
\end{equation}
 
The subspace $C^\Uu_{0} = \mathrm{Tr}_2^{-1}(0)$ is a complex analytic subspace of $C^\Uu$. This space parametrizes the set of all the $\sG$-Higgs bundles where the harmonic map given by the solution of Hitchin's equations is a branched minimal immersion. 

Now consider the restriction of the map $P$ to the complex analytic subspace $C^\Uu_0$:
\[P|_{C^\Uu_0}: C^\Uu_0 \ra C~.\]
The restricted map is surjective. Indeed, maximal representations are Anosov \cite{MaxRepsAnosov} and hence admit at least one equivariant branched minimal immersion, see Remark \ref{Remark maximal are anosov}. 
The map is injective since we are assuming Labourie's conjecture holds on $C$.
This bijection gives $C$ the structure of complex analytic space such that $\MCG(S)$ acts by holomorphic maps.
\end{proof}

\section{Mapping class group equivariant Parametrization}     \label{equivariant}
In this section we will give a detailed description of the construction of Section \ref{sec:mcg_inv_cmplx_str} for the groups $\sP\sSp(4,\R)$ and $\sSp(4,\R)$. As a result, we obtain a parametrization of the components of maximal representations that is equivariant for the action of $\MCG(S)$. Moreover, we describe the quotient $\Xx^{\mathrm{max}}(\Gamma,\sG)/\MCG(S)$. 

\subsection{Description of $\Xx^{\mathrm{max}}(\Gamma,\sG)$}  \label{sec:mcg-invariant description}
Let $\sG=\sP\sSp(4,\R)$ or $\sSp(4,\R)$ and consider the holomorphic function 
\[\mathrm{Tr_2}:\Mm^\mathrm{max}(\Uu,\sG)\to\Hh^2\] from \eqref{EQ Tr_2}. 
Recall from Sections \ref{PSp4R} and \ref{Sp4R} that the moduli space of maximal Higgs bundles can be written as a product
\[\Mm^\mathrm{max}(\Sigma,\sG)=\Nn\times H^0(\Sigma, K^2)~.\]
Moreover, the $\sG_\C$-Higgs bundle $(\Ee,\varphi)$ associated to a point of $(x,q_2)\in\Nn\times H^0(\Sigma,K^2)$ has $\mathrm{Tr_2}(\Ee,\varphi,\Sigma)=0$ if and only $q_2=0$.

As we have seen in the proof of Theorem \ref{thm:invariant complex structures}, $\mathrm{Tr_2}^{-1}(0)$ is a complex analytic space which is  homeomorphic to $\Xx^{\mathrm{max}}(\Gamma,\sG)$ by Theorem \ref{UniqueMinSurface}. From now on, we will identify the two spaces
\[\Xx^{\mathrm{max}}(\Gamma,\sG) = \mathrm{Tr_2}^{-1}(0)~.\]

Consider the natural map 
\begin{equation}\label{EQ pi to Teich}
T:\xymatrix@R=.2em{\Xx^{\mathrm{max}}(\Gamma,\sG)\ar[r]&\Teich(S)\\(\Sigma,\Ee,\varphi)\ar@{|->}[r]&\Sigma}~.
\end{equation} 
The parameterization theorems in Sections \ref{PSp4R} and \ref{Sp4R} give an explicit description of the map $T$. We will restrict the map to the different connected components of $\Xx^{\mathrm{max}}(\Gamma,\sG)$, and describe each one of them.

For $\sG = \sP\sSp(4,\R)$, the following statements  follow directly from putting Theorems \ref{UniqueMinSurface} and \ref{thm:invariant complex structures} together with Theorems \ref{THM d>0}, \ref{thm:zero_component}, and \ref{THM HiggsParamsw1sw2orbifold}.

\begin{Corollary}
For every $d\in(0,4g-4]$, the map $T$ restricted to the component $\Xx^{\mathrm{max}}_{0,d}(\Gamma,\sP\sSp(4,\R))$ is a trivial fiber bundle over $\Teich(S).$ The fiber over $\Sigma \in \Teich(S)$ is biholomorphic to the rank $3g-3+d$ holomorphic vector bundle $\Ff_d$ over the $(4g-4-d)^{th}$-symmetric product of $\Sigma$ described in Theorem \ref{THM d>0}.
\end{Corollary}

\begin{Corollary}
The map $T$ restricted to the component $\Xx^{\mathrm{max}}_{0,0}(\Gamma,\sP\sSp(4,\R))$ is a trivial fiber bundle over $\Teich(S).$ The fiber over $\Sigma \in \Teich(S)$ is biholomorphic to $(\Aa/\Z_2)$, where, as described in  Theorem \ref{thm:zero_component}, $\Aa$ is the holomorphic fiber bundle over $\Pic^0(\Sigma)$ and $\Z_2$ acts on $\Aa$ by pullback by inversion on $\Pic^0(\Sigma).$
\end{Corollary}

\begin{Corollary}
For each $(sw_1,sw_2)\in H^1(\Sigma,\Z_2)\setminus\{0\}\times H^2(\Sigma,\Z_2)$, the map $T$ restricted to the component $\Xx^{\mathrm{max},sw_2}_{sw_1}(\Gamma,\sP\sSp(4,\R))$ is a trivial fiber bundle over $\Teich(S)$. The fiber over $\Sigma \in \Teich(S)$ is biholomorphic to $\left(\Hh' / \Z_2\right)$, where $\Hh'$ is the bundle over $\Prym^{sw_2}(X_{sw_1},\Sigma)$ described in Theorem \ref{THM HiggsParamsw1sw2orbifold}.
\end{Corollary}

Similarly, for $\sG = \sSp(4,\R)$, the following statements follow directly from putting Theorems \ref{UniqueMinSurface} and \ref{thm:invariant complex structures} together with Theorems \ref{THM SP4 d>0}, \ref{THM Sp4 d=0 }, \ref{THM Higgs Sp4 sw1not0}. As with $\sP\sSp(4,\R),$ we denote the connected components of the character variety $\Xx^\mathrm{max}(\Gamma,\sSp(4,\R))$ which correspond to the Higgs bundle connected components $\Mm_{d,0}^\mathrm{max}(\sSp(4,\R))$ and $\Mm_{sw_1}^{\mathrm{max},sw_2}(\sSp(4,\R))$ by $\Xx_{d,0}^\mathrm{max}(\Gamma,\sSp(4,\R))$ and $\Xx_{sw_1}^{\mathrm{max},sw_2}(\Gamma,\sSp(4,\R))$ respectively.

\begin{Corollary}
For each $d\in(0, 2g-2]$, the map $T$ restricted to the component $\Xx^{\mathrm{max}}_{0,d}(\Gamma,\sSp(4,\R))$ is a trivial fiber bundle over $\Teich(S).$ The fiber over $\Sigma \in \Teich(S)$ is biholomorphic to the space $\pi^* \Ff_{2d}$ described in Theorem \ref{THM SP4 d>0}.
\end{Corollary}

\begin{Corollary}
The map $T$ restricted to the component $\Xx^{\mathrm{max}}_{0,0}(\Gamma,\sSp(4,\R))$ is a trivial fiber bundle over $\Teich(S)$. The fiber over $\Sigma \in \Teich(S)$ is biholomorphic to the space $s^*\Aa/\Z_2$ described in Theorem \ref{THM Sp4 d=0 }.
\end{Corollary}

\begin{Corollary}
For $(sw_1,sw_2)\in H^1(\Sigma,\Z_2)\setminus\{0\}\times H^2(\Sigma,\Z_2)$, the map $T$ restricted to the component $\Xx^{\mathrm{max},sw_2}_{sw_1}(\Gamma,\sSp(4,\R))$ is a trivial fiber bundle over $\Teich(S).$ The fiber over $\Sigma \in \Teich(S)$ is biholomorphic to the space $\Kk'/\Z_2$ described in Theorem \ref{THM Higgs Sp4 sw1not0}.
\end{Corollary}

\subsection{Action of $\MCG(S)$}   \label{sec:action of mcg}

If we assume that we understand the action of $\MCG(S)$ on $\Teich(S)$, then the parameterizations given by the above corollaries allow us to understand the action of $\MCG(S)$ on $\Xx^{\mathrm{max}}(\Gamma,\sG)$ in an explicit way. We will now describe the quotient space 
\[\Qq(\Gamma,\sG) = \Xx^{\mathrm{max}}(\Gamma,\sG)/\MCG(S)~.\] 
The action of $\MCG(S)$ on $\Xx^{\mathrm{max}}(\Gamma,\sG)$ is properly discontinuous (see \cite{WienhardAction}, \cite{AnosovFlowsLabourie}), thus the quotient $\Qq(\Gamma,\sG)$ is again a complex analytic space. 

Recall that the fiber bundle map $T:\Xx^{\mathrm{max}}(\Gamma,\sG) \to \Teich(S)$ from \eqref{EQ pi to Teich} is $\MCG(S)$-equivariant, and so induces a map 
\begin{equation}
\label{EQ hat T}\widehat{T}: \Qq(\Gamma,\sG) \to \Mod(S)~.
\end{equation}
The map $\widehat{T}$ is a holomorphic submersion, but, is no longer a fiber bundle map. This is because the action of $\MCG(S)$ on $\Teich(S)$ is not free, and the Riemann surfaces with non-trivial stabilizer project to the singular points of $\Mod(S)$. 

We will write $\Teich(S)=\Teich(S)^\text{free}\sqcup\Teich(S)^\text{fix}$ where $\Teich(S)^\text{free}$ is the open dense subset where the action of $\MCG(S)$ is free, and $\Teich^{\text{fix}}(S)$ is the subset of points where the action of $\MCG(S)$ is not free. Similarly, we will write $\Mod(S)=\Mod^{\text{sm}}(S)\sqcup\Mod^{\text{sing}}(S)$
where $\Mod^{\text{sm}}(S)$ is the open dense subset consisting of smooth points and $\Mod^{\text{sing}}(S)$ is the subset consisting of orbifold singularities. 

On $T^{-1}(\Teich^{\text{free}}(S))$, the action of $\MCG(S)$ is properly discontinuous and free, and the quotient is $\widehat{T}^{-1}(\Mod^{\text{sm}}(S))$. The restriction of $\widehat{T}$ to this open dense subset of $\Qq(\Gamma,\sG)$ is a holomorphic fiber bundle over $\Mod^{\text{sm}}(S)$. The fibers over the points in $\Mod^{\text{sm}}(S)$ will be called \emph{generic fibers} and will be described below. The fibers over the points of $\Mod^{\text{sing}}(S)$ are harder to describe because it is necessary to take into account the action of the stabilizer of the corresponding point in $\Teich^{\text{fix}}(S)$.   

\subsection{The connected components of $\Qq(\Gamma,\sP\sSp(4,\R))$}
We first count the connected components of $\Qq(\Gamma,\sP\sSp(4,\R))$, this was done in \cite{TopInvariantsAnosov} for maximal $\sSp(4,\R)$-representations, we follow a similar line of argument here. The mapping class group acts on $H^1(S,\Z_2)$ and on $H^2(S,\Z_2)$. The action on the second homology is trivial, and the action on the first homology induces a surjective homomorphism 
\[\MCG(S) \ra \sSp(H^1(S,\Z_2))~.\]
In particular, $\MCG(S)$ acts transitively on all non-zero elements of $H^1(S,\Z_2)$, and so there are two orbits: the orbit of zero and the orbit through the nonzero elements. Given a non-zero element $sw_1 \in H^1(S,\Z_2)$ consider its stabilizer:
\[\PMCG(S) = \{ g\in \MCG(S) \ |\ g \cdot sw_1 = sw_1  \}~. \]
We will call $\PMCG(S)$ the \emph{parabolic mapping class group} because it is the inverse image of a parabolic subgroup of $\sSp(H^1(S,\Z_2))$; it is a subgroup of index $2^{2g}-1$.
The invariant $d$ is preserved by the action of $\MCG(S)$ since the pullback of a line bundle of degree $d$ by a holomorphic map still has degree $d$.
These simple considerations already tell us what are the connected components of $\Qq(\Gamma,\sP\sSp(4,\R))$:

\begin{Theorem} \label{Components of PSp4R Q}
For each $d\in[0,4g-4]$, the mapping class group $\MCG(S)$ preserves each connected component $\Xx_{d}^\mathrm{max}(\Gamma,\sP\sSp(4,\R))$ and for each $sw_2$ $\MCG(S)$ permutes the connected components $\Xx^{\mathrm{max},sw_2}_{sw_1}(\Gamma,\sP\sSp(4,\R)).$ In particular, $\Qq(\Gamma,\sP\sSp(4,\R))$ has $4g-1$ connected components. 
\end{Theorem}
The connected components of $\Qq(\Gamma,\sP\sSp(4,\R))$ will be denoted as follows
\[\Qq_d(\Gamma,\sP\sSp(4,\R))=\Xx_d^\mathrm{max}(\Gamma,\sP\sSp(4,\R))/\MCG(S)~,\]
\[\Qq^{sw_2}(\Gamma,\sP\sSp(4,\R))=\left(\bigsqcup\limits_{sw_1\in H^1(S,\Z_2)\setminus\{0\}}\Xx_{sw_1}^{\mathrm{max},sw_2}(\Gamma,\sP\sSp(4,\R))\right)\left\slash\right.\MCG(S)~.\]
We can now describe the topology of the components of $\Qq(\Gamma,\sP\sSp(4,\R))$, the next two theorems follow from Theorem \ref{Components of PSp4R Q}, the considerations in Section \ref{sec:action of mcg}, and the corollaries in Section \ref{sec:mcg-invariant description}.

\begin{Theorem}
For $0 \leq d \leq 4g-4$, the map $\widehat{T}:\Qq_d(\Gamma,\sP\sSp(4,\R)) \ra \Mod(S)$ is a holomorphic submersion over $\Mod(S)$. 
\begin{itemize}
    \item When $d> 0$, the generic fiber is biholomorphic to the rank $3g-3+d$ holomorphic vector bundle $\Ff_d$ over the $(4g-4-d)^{th}$-symmetric product of $\Sigma$ described in Theorem \ref{THM d>0}. 
    \item  When $d=0$, the generic fiber is biholomorphic to $(\Aa/\Z_2)$, where, as described in Theorem \ref{thm:zero_component}, $\Aa$ is the holomorphic fiber bundle over $\Pic^0(\Sigma)$ and $\Z_2$ acts on $\Aa$ by pullback by inversion on $\Pic^0(\Sigma).$
\end{itemize}
\end{Theorem}

The description of $\Qq^{sw_2}(\Gamma,\sP\sSp(4,\R))$ is slightly harder because an orbit of $\MCG(S)$ intersects $2^{2g}-1$ different components. The stabilizer of one of these components is the parabolic mapping class group $\PMCG(S)$. The quotient of the Teichm\"uller space by this subgroup is a  $(2^{2g}-1)$-orbifold cover of the moduli space:
\[\Teich(S)/\PMCG(S) \ra \Mod(S)~. \]
Each component $\Qq^{sw_2}(\Gamma,\sP\sSp(4,\R))$ can be seen as the quotient 
\[\Qq^{sw_2}(\Gamma,\sP\sSp(4,\R)) = \Xx^{\mathrm{max},sw_2}_{sw_1}(\Gamma, \sP\sSp(4,\R))/\PMCG(S) ~.\] 
Thus, there is a holomorphic submersion
\[\widehat{T}':\Qq^{sw_2}(\Gamma,\sP\sSp(4,\R)) \ra \Teich(S)/\PMCG(S)~.\]
In this way we find two descriptions of $\Qq^{sw_2}(\Gamma,\sP\sSp(4,\R))$, one describes it using a map to $\Mod(S)$ with a disconnected fiber, and the other using a map to $\Teich(S)/\PMCG(S)$ with a connected fiber.
\begin{Theorem}
Let $sw_2 \in H^2(S,\Z_2)$.
\begin{itemize}
     \item The map $\widehat{T}:\Qq^{sw_2}(\Gamma,\sP\sSp(4,\R)) \ra \Mod(S)$ is a holomorphic submersion with generic fiber biholomorphic to the disjoint union of $2^{2g}-1$ copies of $\left(\Hh'/\Z_2\right)$, where $\Hh'$ is the  bundle over $\Prym^{sw_2}(X_{sw_1},\Sigma)$ described in Proposition \ref{bundleoverPrym}.
\item The map $\widehat{T}':\Qq^{sw_2}(\Gamma,\sP\sSp(4,\R)) \ra \Teich(S)/\PMCG(S)$ is a holomorphic submersion with generic fiber biholomorphic to $\left(\Hh'/\Z_2\right)$. 
 \end{itemize} 
\end{Theorem}

\subsection{The connected components of $\Qq(\Gamma,\sSp(4,\R))$}

Counting the connected components of $\Qq(\Gamma,\sSp(4,\R))$ is slightly more complicated; they were counted in \cite[Theorem 10]{TopInvariantsAnosov}. To distinguish the components of $\Xx^\mathrm{max}(\Gamma,\sSp(4,\R))$ we have Higgs bundle invariants $sw_1 \in H^1(S,\Z_2)$, $sw_2 \in H^2(S,\Z_2)$, $d\in \Z$ and an extra invariant to distinguish between the $2^{2g}$ Hitchin components. 

The invariant that distinguishes the Hitchin components is a choice of a square root of $K$. There is a well known topological interpretation of this choice, it is equivalent to a choice of a spin structure on $S$. 
Spin structures have a topological invariant called the \emph{Arf invariant} which takes values in $\Z_2$ and is preserved by the action of $\MCG(S)$. A spin structure is called \emph{even} or \emph{odd} depending on the value of the Arf invariant. There are $2^{g-1}(2^g+1)$ even and $2^{g-1}(2^g-1)$ odd spin structures (see \cite{AtiyahSpinStructures}). The mapping class group acts transitively on the set of odd spin structures and on the set of even spin structures. 

 Recall from Section \ref{Sp4R} that $\Xx^\mathrm{max}(\Gamma,\sSp(4,\R))$ decomposes as
\[\bigsqcup\limits_{\substack{(sw_1,sw_2)\in\\ H^1(S,\Z_2)\setminus\{0\}\times H^2(S,\Z_2)}}\Xx^{\mathrm{max},sw_2}_{sw_1}(\Gamma,\sSp(4,\R))\ \sqcup\bigsqcup\limits_{d\in[0,2g-2]}\Xx_{d}^\mathrm{max}(\Gamma,\sSp(4,\R))~.\]
 Using the notation above, we now state the theorem which determines the connected components $\Qq(\Gamma,\sSp(4,\R))$.

\begin{Theorem}\cite[Theorem 10]{TopInvariantsAnosov}\label{Components of Sp4R Q}
For each $d\in[0,2g-2),$ the mapping class group $\MCG(S)$ preserves $\Xx_d^\mathrm{max}(\Gamma,\sSp(4,\R))$. For $sw_2\in H^2(S,\Z_2)$, the action of $\MCG(S)$ permutes the components $\Xx_{sw_1}^{\textrm{max},sw_1}(\Gamma,\sSp(4,\R))$. For $\Xx^\textrm{max}_{2g-2}(\Gamma,\sSp(4,\R))$, $\MCG(S)$ acts on the $2^{2g}$ connected components with two orbits distinguished by the Arf invariant. 
In particular, $\Qq(\Gamma,\sSp(4,\R))$ has $2g+2$ connected components.
\end{Theorem}
The connected components of $\Qq(\Gamma,\sSp(4,\R))$ will be denoted as follows
\[\xymatrix{\Qq_d(\Gamma,\sSp(4,\R))=\Xx_d^\mathrm{max}(\Gamma,\sSp(4,\R))/\MCG(S)&\text{for\ } d\in[0,2g-2)}~,\]
\[\Qq^{sw_2}(\Gamma,\sSp(4,\R))=\left(\bigsqcup\limits_{sw_1\in H^1(S,\Z_2)\setminus\{0\}}\Xx_{sw_1}^{\mathrm{max},sw_2}(\Gamma,\sSp(4,\R))\right)\slash\MCG(S)~,\]
\[\Qq_{2g-2,0}(\Gamma,\sSp(4,\R))\sqcup\Qq_{2g-2,1}(\Gamma,\sSp(4,\R))=\Xx_{2g-2}^\mathrm{max}(\Gamma,\sSp(4,\R))/\MCG(S)~.\]
where $\Qq_{2g-2,a}(\Gamma,\sSp(4,\R))$ denotes the quotient space for Arf invariant $a\in\Z_2.$

We can now describe the components of $\Qq(\Gamma,\sSp(4,\R))$. The next three theorems follow from Theorem \ref{Components of Sp4R Q}, the considerations in Section \ref{sec:action of mcg},  and the corollaries in Section \ref{sec:mcg-invariant description}.

\begin{Theorem}
For $0 \leq d < 2g-2$, the map $\widehat{T}:\Qq_d(\Gamma,\sSp(4,\R)) \to \Mod(S)$ from \eqref{EQ hat T} is a holomorphic submersion. 
\begin{itemize}
    \item When $0 < d < 2g-2$, the generic fiber is biholomorphic to the space $\pi^* \Ff_{2d}$ described in Theorem \ref{THM SP4 d>0}.
    \item When $d=0$, the generic fiber  is biholomorphic to the space $s^*\Aa/\Z_2$ described in Theorem \ref{THM Sp4 d=0 }.
\end{itemize}
\end{Theorem}

To describe the components corresponding to $sw_1 \neq 0$, we consider also the map
\[\widehat{T}':\Qq^{sw_2}(\Gamma,\sSp(4,\R) \to \Teich(S)/\PMCG(S)~.\]
Again, we find two descriptions of $\Qq^{sw_2}(\Gamma,\sSp(4,\R))$, one describes it using a map to $\Mod(S)$ with a disconnected fiber, and the other using a map to $\Teich(S)/\PMCG(S)$ with a connected fiber.
\begin{Theorem}
Let $sw_2 \in H^2(S,\Z_2)$.
\begin{itemize}
     \item The map $\widehat{T}:\Qq^{sw_2}(\Gamma,\sSp(4,\R)) \to \Mod(S)$ is a holomorphic submersion with generic fiber biholomorphic to the disjoint union of $2^{2g}-1$ copies of 
the space $\Kk'/\Z_2$ described in Theorem \ref{THM Higgs Sp4 sw1not0}.
\item The map $\widehat{T}':\Qq^{sw_2}(\Gamma,\sSp(4,\R)) \to \Teich(S)/\PMCG(S)$ is a holomorphic submersion with generic fiber biholomorphic to $\Kk'/\Z_2$. 
 \end{itemize} 
\end{Theorem}

Finally, we consider quotients of the Hitchin components $\Qq_{2g-2,a}(\Gamma,\sSp(4,\R))$.
Given a spin structure on $S$, the subgroup of the mapping class group preserving it is called the \emph{spin mapping class group}, see \cite{HarerSpin} for more information. There are two conjugacy classes of spin mapping class groups which depend on the Arf invariant of the preserved spin structure. They will be denoted by $\SMCG_a(S)$, where $a\in \Z_2$ is the corresponding Arf invariant. The group $\SMCG_0(S)$ is a subgroup of index $2^{g-1}(2^g+1)$, and $\SMCG_1(S)$ has index $2^{g-1}(2^g-1)$. 
The stabilizer in $\MCG(S)$ of a Hitchin component is the group $\SMCG_a(S)$. The quotient of the Teichm\"uller space by one of these subgroups is a finite orbifold cover of the moduli space:
\[\Teich(S)/\SMCG_a(S) \to \Mod(S)\]
Each component $\Qq_{2g-2,a}(\Gamma,\sSp(4,\R))$ can be seen as the quotient of one Hitchin component by the relevant spin mapping class group. Thus, there is a holomorphic submersion
\[\widehat{T}':\Qq_{2g-2,a}(\Gamma,\sSp(4,\R)) \ra \Teich(S)/\SMCG_a(S)~.\]
As before, we find two description of the spaces $\Qq_{2g-2,a}(\Gamma,\sSp(4,\R))$.
\begin{Theorem}
For each $a \in \Z_2$, the map $\widehat{T}:\Qq_{2g-2,a}(\Gamma,\sSp(4,\R)) \ra \Mod(S)$ is a holomorphic submersion with generic fiber biholomorphic to the disjoint union of $2^{g-1}(2^g+1)$ or $2^{g-1}(2^g-1)$ copies of the vector space $H^0(\Sigma,K^4)$. 
The map $\widehat{T}'$ is a holomorphic submersion with generic fiber biholomorphic to $H^0(\Sigma,K^4)$. 
\end{Theorem}

\bibliography{sample}{}
\bibliographystyle{plain}

\end{document}